\documentclass{amsart}
\usepackage[utf8]{inputenc}

\usepackage{geometry}
\usepackage{enumerate}%%序号
\usepackage{amsmath}%%数学公式
\usepackage{cite}%%%%%引用
\usepackage{ulem}
\usepackage{amsfonts}%%数学字体
\usepackage{mathrsfs}%%数学字体
\usepackage{galois}%%%%一些数学符号的改进
\newtheorem{Thm}{Theorem}[section]
\newtheorem{Def}[Thm]{Definition}
\newtheorem{Eg}[Thm]{Example}
\newtheorem{Lem}[Thm]{Lemma}
\newtheorem{Prop}[Thm]{Proposition}
\newtheorem{Cor}[Thm]{Corollary}
\newtheorem{Rem}[Thm]{Remark}

\numberwithin{equation}{section}

\newcommand{\AAa}{\mathbb{A}}
\newcommand{\BB}{\mathbb{B}}

\newcommand{\NN}{\mathbb{N}}

\newcommand{\RR}{\mathbb{R}}
\newcommand{\SSp}{\mathbb{S}}

\newcommand{\ZZ}{\mathbb{Z}}

\newcommand{\cA}{\mathcal{A}}
\newcommand{\cB}{\mathcal{B}}

\newcommand{\cF}{\mathcal{F}}
\newcommand{\cG}{\mathcal{G}}

\newcommand{\cI}{\mathcal{I}}

\newcommand{\cK}{\mathcal{K}}
\newcommand{\cL}{\mathcal{L}}
\newcommand{\cM}{\mathcal{M}}

\newcommand{\cR}{\mathcal{R}}
\newcommand{\cS}{\mathcal{S}}

\newcommand{\cU}{\mathcal{U}}
\newcommand{\cV}{\mathcal{V}}

\newcommand{\cZ}{\mathcal{Z}}

\newcommand{\scB}{\mathscr{B}}

\newcommand{\scD}{\mathscr{D}}
\newcommand{\scE}{\mathscr{E}}

\newcommand{\scH}{\mathscr{H}}

\newcommand{\scM}{\mathscr{M}}

\newcommand{\scR}{\mathscr{R}}

\newcommand{\scT}{\mathscr{T}}

\newcommand{\scV}{\mathscr{V}}

\newcommand{\scX}{\mathscr{X}}

\newcommand{\bfA}{\textbf{A}}

\newcommand{\bfF}{\textbf{F}}
\newcommand{\bfG}{\textbf{G}}

\newcommand{\bfI}{\textbf{I}}

\newcommand{\bfL}{\textbf{L}}
\newcommand{\bfM}{\textbf{M}}

\newcommand{\spt}{\text{spt}}

\author{Zhihan Wang} 
\address{Department of Mathematics, Princeton University, Princeton, NJ 08544}
\email{zhihanw@math.princeton.edu}

\title{Deformations of Singular Minimal Hypersurfaces I, Isolated Singularities}

\begin{document}

   \begin{abstract}
   Locally stable minimal hypersurface could have singularities in dimension$\geq 7$ in general, locally modeled on stable and area-minimizing cones in the Euclidean spaces. In this paper, we present different aspects of how these singularities may affect the local behavior of minimal hypersurfaces. 

   First, given a non-degenerate minimal hypersurface with strictly stable and strictly minimizing tangent cone at each singular point, under any small perturbation of the metric, we show the existence of a nearby minimal hypersurface under new metric. For a generic choice of perturbation direction, we show the entire smoothness of the resulting minimal hypersurface.  
   
   Second, given a strictly stable minimal hypersurface $\Sigma$ with strictly minimizing tangent cone at each singular point, we show that $\Sigma$ is uniquely homologically minimizing in its neighborhood. 
   
   Lastly, given a family of stable minimal hypersurfaces multiplicity $1$ converges to a stable minimal hypersurface $\Sigma$ in an eight manifold, we show that there exists a non-trivial Jacobi field on $\Sigma$ induced by these converging hypersurfaces, which generalizes a result by Ben Sharp in smooth case; Moreover, positivity of this Jacobi field implies smoothness of the converging sequence. 
   \end{abstract}

  \maketitle   
  \tableofcontents

%%%%%%%%%%%%%%%%%%%%%%%%%%%%%%%%%%%%%%%%%%%%%%%%%%%%%%%%%%%%%%%%%%%%%%%%%%%%
%%%%%%%%%%%%%%%%%%%%%%%%%%%%%%%%%%%%%%%%%%%%%%%%%%%%%%%%%%%%%%%%%%%%%%%%%%%%
%%%%%%%%%%%%%%%%%%%%%%%%  SECTION 1, INTRODUCTION  %%%%%%%%%%%%%%%%%%%%%%%%%
%%%%%%%%%%%%%%%%%%%%%%%%%%%%%%%%%%%%%%%%%%%%%%%%%%%%%%%%%%%%%%%%%%%%%%%%%%%%
%%%%%%%%%%%%%%%%%%%%%%%%%%%%%%%%%%%%%%%%%%%%%%%%%%%%%%%%%%%%%%%%%%%%%%%%%%%%

    \section{Introduction} \label{Sec, intro}%%%%%%第零章 
%     Minimal hypersurface in general dimension has long been studied since 1960s by [Fleming-Federer; De Giorgi; Reifenberg], first focusing on area-minimizing objects; 
%     Later, [Almgren-Pitts] establish the min-max theory of constructing locally stable minimal hypersurfaces in an arbitrary closed Riemannian manifold.
     Minimal hypersurface has been frequently studied in recent years in dimension between $2$ and $6$ through Almgren-Pitts min-max theory, see for example \cite{MarquesNeves14, MarquesNeves16_Index, MarquesNeves17_Infinite, MarquesNevesSong19, IrieMarquesNeves18, Song18_YauConj}. 

     In higher dimensions, the minimal hypersurfaces given by min-max constructions are known to have codimension $\geq 7$ singular set, locally modeled on stable minimal hypercones in the Euclidean spaces, see \cite{SchoenSimon81}. Two natural questions are then to study the structures of singular set, as well as how they affect the "behavior" of minimal hypersurfaces. In a measure theoretic sense, the former question was addressed in \cite{Almgren00_BigReg, Simon93_Cylin, Simon93_SurveyRect, White97, CheegerNaber13_Stra_Min_Surf, NaberValtorta20_MS_Rect}, providing a natural stratification of singular set together with a quantitative volume estimate, curvature estimate and rectifiability for each stratum.

     In this article, we study the latter question above in different aspect. \\
     
\noindent   \textbf{$\bullet$ Local Moduli.}
      \cite{White91, White17} studied the space $\Gamma$ of Riemannian metrics and smooth minimal surfaces pairs $(g, \Sigma)$, which admits a natural Banach manifold structure and makes the projection onto the first variable a Fredholm mapping with index $0$. 
      
      When $\Sigma$ is allowed to have singularities, no local structure result of $\Gamma$ is known in the literature. 
      In fact, when $g$ is Euclidean, possible singular structures for minimal surfaces in a neighborhood of $ \Sigma$ near singular set of $\Sigma$ was studied in \cite{CaffHardtSimon84, SimonSolomon86, ColomboEdelenSpolaor17,  EdelenSpolaor19}, for certain classes of singularity models. And it's shown that nearby minimal surface can have different singularity type from $\Sigma$, making it hard to study the local structure of $\Gamma$ using parametric approach.
      
      In this paper, we shall address the following local existence problem,\\     
     \textsl{\textbf{Question:} Given a closed singular minimal hypersurfaces $\Sigma_0$ in a Riemannian manifold $(M, g_0)$ and a smooth perturbed metric $g$ of $g_0$, is there always at least one, possibly singular, hypersurface $\Sigma \subset M$ minimal with respect to $g$ and staying close to $\Sigma_0$ ? }
          
     Our first result provide partial answers to the question above when the singular set are \textbf{strongly isolated}.
     \begin{Def}
      Call $p\in Sing(\Sigma)$ \textbf{strongly isolated}, if the tangent cone of $\Sigma$ at $p$ is of multiplicity $1$ and singular only at $0$. 
     \end{Def}

     \begin{Thm} \label{Thm_Intro, main thm, deform}
      Suppose $\Sigma \subset (M, g)$ be an optimal regular, locally stable, two-sided minimal hypersurface with only strongly isolated singularities. Also suppose that 
      \begin{enumerate}
       \item[(1)] $\Sigma$ is nondegenerate
       \item[(2)] the tangent cone at each singularity is strictly minimizing.
       \item[(3)] the tangent cone at each singularity is strictly stable 
      \end{enumerate}
      Then for every $\epsilon > 0$, there's some $\delta > 0$ depending on $\Sigma, M, g, \epsilon$, such that if $g'$ is a smooth metric with $\|g'-g\|_{C^4} \leq \delta $, then there exists a hypersurface $\Sigma' \subset M$ minimal, locally stable under $g'$ with optimal regularity lying in the $\epsilon$-neighborhood of $\Sigma$ and satisfying \[
       \bfF(\Sigma', \Sigma) \leq \epsilon  \]
     \end{Thm}

     The notions of nondegeneracy, index, strict stability and strictly minimizing property are discussed in section \ref{Sec, Prelim} and \ref{Sec, Linear Ana}. Note that the nondegeneracy condition is a generic condition, see proposition \ref{Prop_Linear, Basic prop for L^2-noncon} (6). Also note that theorem \ref{Thm_Intro, main thm, deform} fails if we drop the assumption (2), see proposition \ref{Prop_App, Deform Thm fails without minimizing assump}.
     
     Examples of minimal hypercones have been constructed in \cite{HsiangLawson71, FerusKarcherMunzner81, FerusKarcher85}, many of which has been verified to be area-minimizing by \cite{Lawson72, Simoes74, Lawlor91, SterlingWang94}. It's also worth mentioning that all examples of area-minimizing hypercones in the literature above are proved to be strictly stable and strictly minimizing, see \cite{SterlingWang94, TangZhang20}.

%     It is worthwhile to remark here that strictly minimizing pressumtions in Theorem~\ref{Intro, main thm, deform} (2) is essential for the statement to be true. See the discussion in section 8; the strict stability assumptions in Theorem \ref{Intro, main thm, deform} (3), on the other hand, may not be essential and is conjectured to be dropped, compared with Theroem \ref{Intro, loc minimiz} below. 
     
     Unlike its proof in case $Sing(\Sigma) = \emptyset$, theorem \ref{Thm_Intro, main thm, deform} is not derived by fix point method when singularities exist. In fact, unlike in the smooth setting, the uniqueness of such nearby minimal hypersurfaces under a fixed perturbed metric is not known, even when tangent cones are Simons cones. See the problem (P4.1) in section \ref{Sec, Apps and Discussion}. \\

\noindent   \textbf{$\bullet$ Generic Regularity.}
     When $\Sigma$ is the unique homologically area-minimizer in some homology class of $M$, theorem \ref{Thm_Intro, main thm, deform} is automatic by a compactness argument. 
     By appropriately choosing the perturbed metric $g'$ , \cite{Smale93} was able to establish the smoothness of homologically area-minimizer for $C^{\infty}$-generic metric in dimension $8$; Combined with a similar perturbation argument locally,\cite{ChodoshLioukumovichSpoloar20_GenericReg} used min-max approach to show that in generic Riemannian $8$-manifold with positive Ricci curvature, there exists a closed embedded smooth minimal hypersurface. 
     We prove a local analogue of these regularity theorems in section \ref{Sec, Apps and Discussion}.
     \begin{Thm} \label{Thm_Intro, generic smooth}
      Let $\Sigma\subset (M, g)$ be as in theorem \ref{Thm_Intro, main thm, deform} satisfying (1)-(3). 
      
      Then for $C^k$-generic $f\in C^{\infty}(M)$ on $M$ ($k\geq 4$), if $\{g_t = g(1+ tf) + o_4(t)\}_{t\in (-1, 1)}$ is a family of smooth Riemannian metric near $g$, 
      then for each sufficiently small $t$, there's a closed embedded \textbf{smooth} minimal hypersurface $\Sigma_t\subset (M ,g_t)$, and $\Sigma_t\to \Sigma$ in varifold sense when $t\to 0$. 
     \end{Thm}
     Here $o_4(t)$ means some symmetric $2$-tensor converges to $0$ in $C^4$ after dividing by $t$.\\

\noindent   \textbf{$\bullet$ Induced Jacobi Fields.}
     The proof of regularity in theorem \ref{Thm_Intro, generic smooth} is based on a general result deriving the smoothness of $\Sigma_t$ from the behavior of Jacobi field they induce on $\Sigma$. 

     Recall that by \cite{Sharp17}, if $\Sigma_j$ are a family of distinct minimal hypersurfaces in $(M, g)$ multiplicity $1$ converges to a smooth minimal hypersurface $\Sigma_\infty$, then the convergence is in $C^{\infty}$ graphical sense and the normalization of graphical functions sub-converges to a smooth Jacobi field on $\Sigma$, called an \textbf{induced Jacobi field} on $\Sigma$. As an analogue we have,
     \begin{Thm} \label{Thm_Intro, Induced Jacobi field}
      Let $\{(\Sigma_j, \nu_j)\}_{1\leq j\leq \infty}\subset (M, g)$ be a family of two-sided stable minimal hypersurfaces. 
      Suppose $\Sigma_j$ multiplicity $1$ converges to $\Sigma_{\infty}$ and $\Sigma_\infty$ has only strongly isolated singularities. 

      Then after passing to a subsequence, $\{\Sigma_j\}_{j\geq 1}$ induces a Jacobi field $u\in C^{\infty}(\Sigma)$, i.e. $\big(\Delta_{\Sigma}+ |A_{\Sigma}|^2 + Ric_M(\nu_\infty, \nu_\infty)\big) u = 0$ on $\Sigma$. 
      
      Moreover, if $u>0$ near some singularity $p\in Sing(\Sigma_\infty)$, then for infinitely many $j>>1$, $\Sigma_j$ are smooth near $p$.
     \end{Thm}
     
     We emphasis that here no further assumption is required on singularities of $\{\Sigma_j\}$. In particular, the theorem is always true in $8$ manifolds. In a upcoming paper joint with Yangyang Li, a slightly generalized version of this theorem is used to get the existence of closed embedded smooth minimal hypersurface in arbitrary $8$-manifolds for generic metric, which generalizes the results in \cite{ChodoshLioukumovichSpoloar20_GenericReg}.
     
     The index of a singular minimal hypersurface is introduced in \cite{Dey19_Cptness}, based on which, theorem \ref{Thm_Intro, Induced Jacobi field} also holds when replacing the stability assumption on $\{\Sigma_j\}_{1\leq j\leq \infty}$ by that they all have the same index. Certain metric variations are also allowed. See theorem \ref{Thm_Ass Jac, Main thm} for the precise statement.
     
     Along the way to theorem \ref{Thm_Intro, Induced Jacobi field}, the first step is to look at minimal hypersurfaces asymptotic to a cone near infinity. 
     This object was studied in \cite{BombieriDeGiorgiGiusti69_Bernstein, HardtSimon85, SimonSolomon86, White89, SterlingWang94, Chan97, Mazet14_MS_Asym_Simons, EdelenSpolaor19}, most focusing on minimizing hypersurfaces and hypercones. In section \ref{Subsec, Ass Jac_Asymp of min surf near infty}, we prove the following generalization of \cite{HardtSimon85} in stable case, which is itself interesting.
     \begin{Thm} \label{Thm_Intro, Asymp Rate Lower Bd and Regularity}
      Let $\Sigma\subset \RR^{n+1}$ be a stable minimal hypersurface with optimal regularity asymptotic to a regular minimal hypercone $C\subset \RR^{n+1}$ near infinity. If write $\Sigma$ as the graph over $C$ of some function $u$ outside a large ball, then we have \[
        \liminf_{R\to +\infty} R^{-2\gamma_1^- - n} \int_{A_{R, 2R}} u^2 > 0    \]
      where $A_{R, 2R}$ is the annulus in $C$ and $-(n-2)<\gamma_1^-\leq -(n-2)/2$ are some geometric constant depending on $C$, defined in section \ref{Subsec, Geom of cone}.
      
      Moreover, if $\Sigma$ lies on one side of $C$ near infinity, then $Sing(\Sigma) = \emptyset$ and it's contained entirely in that side and foliate the component of $\RR^{n+1}\setminus C$ containing it by rescaling.
     \end{Thm}

\noindent   \textbf{$\bullet$ Local Minimizing Property.}
     A natural question by Lawlor \cite{Lawlor91} is that given a (strictly) minimizing hypercone $C\subset \RR^{n+1}$, is every minimal hypersurface $\Sigma$ with $C$ to be the tangent cone at some singularity $p$ minimizing in some neighborhood of $p$?
     When $C$ is strictly stable or the asymptotic rate of $\Sigma$ to $C$ near $p$ is strictly greater than $1$, \cite[theorem 4.4]{HardtSimon85} confirmed this question.
     
     We get a positive answer in this article to a more general version of the question under stability assumption.
     \begin{Thm} \label{Thm_Intro, loc minimiz}
      Suppose $\Sigma \subset (M, g)$ be a stable minimal hypersurface with only strongly isolated singularities satisfying (1) and (2) in theorem \ref{Thm_Intro, main thm, deform}.
      
      Then there's some small neighborhood $U$ of $Clos(\Sigma)$ which admits a mean convex foliation. More precisely, there's a family of piecewise smooth mean convex neighborhood $\{U_t\}_{t\in (0,1)}$ of $Clos(\Sigma)$ such that 
      \begin{enumerate}
       \item[(1)] $U_{t_1} \subset U_{t_2}$ if $t_1\leq t_2$
       \item[(2)] $U = \bigcup_{t\in (0,1)} U_t$, $Clos(\Sigma) = \bigcap_{t\in (0,1]}U_t $; 
       $U_s = \bigcup_{t<s}U_t$; $Clos(U_s) = \bigcap_{t>s} U_t$, $\forall s\in (0,1)$.
      \end{enumerate}    
      In particular, $\Sigma$ is homologically area-minimizing in $U$ and is the unique stationary varifold in $U$ up to multiplicity.
     \end{Thm}
     For a precise description of mean convexity of piecewise smooth domain, see theorem \ref{Thm_One-sided Perturb, Main Thm, mean convex foliation}.
     
     Similar local minimizing property was proved by a contradiction argument in \cite[section 3, lemma 4]{Smale99} assuming  that each tangent cone are strictly stable and strictly minimizing, and that $\Sigma$ is a minimal cone in the Euclidean space near each singularity; By a calibration argument, \cite{ZhangYS18_Loc_Minimizing} obtained similar local minimizing result for singular minimal submanifolds of higher codimension with Lawlor's cones in \cite{Lawlor91} to be the tangent cone, and with specifically chosen metrics away from singular points. \\

    \paragraph{\textbf{Organization of the Article}}
     Section \ref{Sec, Prelim} contains the basic notations and assumption we use in this article. We also present a brief review of analysis on singular minimal surfaces, geometric features of minimal cones and asymptotic conical minimal hypersurfaces. They will be frequently used throughout this article. 
     
     In section \ref{Sec, Linear Ana}, we study the analytic behavior of the linearized minimal surfaces operator, namely, the Jacobi operator $L_{\Sigma}$, on a singular minimal hypersurface $\Sigma$. The key property is the $L^2$-compactness of a family of functions with uniformly bounded $L^2$-norm and uniformly bounded quadratic form associated to $L_{\Sigma}$. Similar compactness has been proved by Schoen-Yau using a uniform $L^2$-nonconcentration property and coercivity in \cite{SchoenYau17_PSC} for the operators associated to the specific quadratic form of minimal slicings and by Lohkamp using the associated skin structure in \cite{Lohkamp19_Potential_I} for the $\cS$-adapted Schr\"{o}dinger operators. In fact, their argument is sufficient to obtain the desired compactness if each tangent cone is strictly stable. A brief review of Schoen-Yau's argument in \cite{SchoenYau17_PSC} is contained in section \ref{Subsec, Linear ana of Jac oper}.
     
     For a general $\Sigma$, the coercivity of $L_{\Sigma}$ might fail. Hence, we establish another $L^2$-compactness in section \ref{Subsec, General linear Analysis} and \ref{Subsec, L^2 noncon for iso sing} when $\Sigma$ has strongly isolated singularities, which allows us to still get spectral decomposition of $L_{\Sigma}$ and define the Green's function. We also derive an asymptotic estimate of Green's function near singularities in section \ref{Subsec, Linear asymp near sing}, which is used in section \ref{Sec, One-sided Perturb} to deduce certain mean convexity of one-sided perturbations of minimal hypersurfaces.
     
     Section \ref{Sec, Asymp & Asso Jac field} studies the induced Jacobi fields of a family of converging stable minimal hypersurfaces. In particular, we prove some more general version of theorem \ref{Thm_Intro, Induced Jacobi field} and \ref{Thm_Intro, Asymp Rate Lower Bd and Regularity}, relating the asymptotic rate of induced Jacobi fields with specific blow-up limit of the converging sequence near each singularity. In particular, we show that the induced Jacobi fields are actually fall in some specific function space described in section \ref{Sec, Linear Ana}.
%     A first step toward this is to look at minimal hypersurfaces asymptotic to a regular cone $C$ near infinity, as is done in section []. This is based on the profound work of [Hardt-Simon '85]. In particular, we provide a new proof of the regularity of stable minimal hypersurface lying on the one side of $C$ near infinity, even without assuming $C$ to be area-minimizing, see Theorem \ref{Ass Jac, Reg of infty one-side perturb}. This method also inspires an asymptotic rate estimate of the minimal hypersurface towards the cone, see Theorem \ref{Ass Jac, Asymp Thm of ext. min. graph}.
%      Modeled on these hypersurfaces, the second part of section \ref{Sec, Asymp & Asso Jac field} focuses on the asymptotic rate estimate of Jacobi fields associated to a family of nearby minimal hypersurfaces, where strict stability of tangent cones plays an important role. 
     
     In section \ref{Sec, One-sided Perturb}, we prove theorem \ref{Thm_Intro, loc minimiz} . The key is an existence of minimal graph over compact subset of regular part with weighted estimate, see theorem \ref{Thm_One-sided Perturb, Ext graph with growth est}. Note that this estimate doesn't require the tangent cone to be minimizing.

%%%%%%%%%%%%%%%%%%%%%%%%%%%%%%%%%%%%%%%%%%%%%%%%%%%%%%%%%%%%%%%%%%%%%%%%%%
%%%%%%%%%%%%%%%%%%%%%%%%%%%%%%%%%%%%%%%%%%%%%%%%%%%%%%%%%%%%%%%%%%%%%%%%%%
%     For the constructions in unstable case, we need to solve an obstacle problem for min-max family, which we establish in section 4 and 5 for general dimensions and general ambient manifold with boundary. Our main result is 
%     \begin{Thm}
%      Let $(U, g)$ be an compact Riemannian $n+1$-manifold with boundary. $\Pi$ be an admissible family of sweepouts of integral n-cycles in $U$. Then there's a \textbf{constraint stationary}, \textbf{optimal regular} hypersurface $\Sigma\subset U$ realizing the width of $\Pi$
%     \end{Thm}
%%%%%%%%%%%%%%%%%%%%%%%%%%%%%%%%%%%%%%%%%%%%%%%%%%%%%%%%%%%%%%%%%%%%%%%%%%
%%%%%%%%%%%%%%%%%%%%%%%%%%%%%%%%%%%%%%%%%%%%%%%%%%%%%%%%%%%%%%%%%%%%%%%%%%

     In section \ref{Sec, Strong min-max prop}, inspired by \cite{White94}, we prove an analogue of strong min-max property for $\Sigma$ in theorem \ref{Thm_Intro, main thm, deform} in sufficiently small neighborhood $U$. The proof closely follows the argument in \cite{White94}, except that we have to estimate error terms near singularities using similar argument as \cite{Smale99}. 
     This allows us to construct an admissible family of sweepouts under any perturbed metric, with width close to the volume of $\Sigma$. Then by applying the min-max theory for obstacle problem in \cite{WangZH2020_Obst}, for any slightly perturbed metric $g'$, there's a constrained embedded minimal hypersurface $\Sigma'\subset (Clos(U), g')$. 
     
     To conclude that $\Sigma'$ doesn't intersect $\partial U$ for certain small $U$ and $g$ sufficiently close to $g_0$, we prove in section \ref{Sec, Rigidity of Constraint Minimizer} a local rigidity result for constrained embedded minimal hypersurfaces, as an analogue of \cite[Theorem 5.1]{WangZH2020_Obst}.

     In section \ref{Sec, Apps and Discussion}, we complete the proof of theorem \ref{Thm_Intro, main thm, deform} and theorem \ref{Thm_Intro, generic smooth}. 
%     We also discuss what happens when the tangent cone at some singularity is stable but not minimizing, or minimizing but not strictly minimizing. 
     In the end, we list some open problems related with singular minimal hypersurfaces.\\

    \paragraph{\textbf{Acknowledgements}}
     I am grateful to my advisor Fernando Cod\'a Marques for his constant support and guidance. I'm also thankful to Camillo De Lellis, Leon Simon and Brian White for sharing their knowledge about singular minimal hypersurfaces. I would like to thank Otis Chodosh, Yevgeny Liokumovich and Luca Spolaor for the inspiring conversations on deforming stable minimal hypersurfaces. I would also thank Chao Li, Yangyang Li, Zhenhua Liu, Akashdeep Dey, Yongsheng Zhang and Xin Zhou for their interest in this work and consistent encouragements.

    \section{Preliminaries} \label{Sec, Prelim} 
     Throughout this article, let $n\geq 3$, $(M, g)$ will be an $n+1$ dimensional closed oriented Riemannian manifold. We may assume that $(M, g)$ isometrically embedded in some $\RR^L$ if necessary Write 
     \begin{align*}
      Int(A)\ &\ \text{the interior of a subset }A\subset M \\
      Clos(A)\ &\ \text{the closure of a subset }A\subset M \\
      dist_M\ &\ \text{the intrinsic distant function on }M \\
      B^M_r(p)\ &\ \text{the open geodesic ball in } M \text{ of radius } r \text{ centered at }p \\
      B^M_r(E)\ &:=\ \bigcup_{p\in E} B^M_r(p), \text{the r neighborhood of }E\subset M \\
      A^M_{s, r}(p)\ &:=B^M_r(p)\setminus Clos(B^M_s(p))\ \text{the open geodesic annuli in } M \\
      \scH^k\ &\ k\text{-dimensional Hausdorff measure} \\
      \nabla_M\ &\ \text{the Levi-Civita connection on }M \\
      \exp ^M\ &\ \text{the exponential map of }M \\
      inj(M)\ &\ \text{the intrinsic injectivity radius of }M \\
      dvol_M\ &\ \text{the unit volume form on }M \\
      Ric_M\ &\ \text{the Ricci curvature tensor of }M
     \end{align*}
     For an open subset $U\subset M$, let
     \begin{align*}
      \scX_c(U)\ &\ \text{the space of compactly supported vector fields in }U  \\
      e^{tV}\ &\ \text{the one-parameter family of diffeomorphism generated by }V\in \scX_c(U) \\
      \scD^k(U)\ &\ \text{the space of compactly supported smooth k-forms in }U
     \end{align*} 
     For fixed $p\in M$, we usually identify the tangent space $T_p M$ wit $\RR^{n+1}$ and view $B_{inj(M)}(p)\hookrightarrow \RR^{n+1}$ by $exp^M$.
     In an $n+1$ dimensional Euclidean space $\RR^{n+1}$, write
     \begin{align*}
      \BB^{n+1}_r(p)\ &\ \text{the open ball of radius } r \text{ centered at }p \\
      \BB^{n+1}_r\ &:=\ \BB^{n+1}_r(0^{n+1}) \\
      \SSp^n_r(p)\ &\ \text{the sphere of radius } r \text{ centered at }p \\
      \AAa_{s, r}^{n+1}(p)\ &\ \text{the open annuli }\BB_r(p)\setminus Clos(\BB_s(p)) \text{ centered at }p \\
      \eta_{p, r}\ &\ \text{the map between }\RR^{n+1}, \text{ maps }x \text{ to }(x-p)/r \\
      \bfG(V, k)\ &\ \text{the Grassmannian of } k\text{ dimensional unoriented linear subspace in a real vector space }V \\
      \omega_k\ &\ \text{the volume of }k \text{-dimensional unit ball} \\
      \theta^k(x, \mu)\ &\ \text{the k-th density (if exists) of a Radon measure }\mu \text{ at }x, \text{ i.e. }\lim_{r\to 0}\frac{1}{\omega_k r^k}\mu(\BB_r^{n+1}(x))
     \end{align*}
     For $E\subset \RR^{n+1}$, we sometimes write for simplicity $\frac{1}{r}(E-p):= \eta_{p,r}(E)$. 
     
     Let $\Sigma\subset U$ be a two-sided smooth hypersurface with normal field $\nu$ (extended to its neighborhood by $\nabla^M_{\nu}\nu = 0$). Throughout this article, unless stated otherwise, every hypersurface will be two-sided and have \textbf{optimal regularity}, i.e. $\scH^{n-2}(U\cap Clos(\Sigma) \setminus \Sigma) =0$ and $\scH^n \llcorner \Sigma$ is locally finite. Let 
     \begin{align*}
      Reg(\Sigma)\ &:= \{x\in Clos(\Sigma)\cap U: Clos(\Sigma) \text{ is smooth, embedded hypersurface near }x \} \\
      Sing(\Sigma)& := U\cap Clos(\Sigma)\setminus Reg(\Sigma)    
     \end{align*}
     By adding points to $\Sigma$ if necessary, we may identify $\Sigma = Reg(\Sigma)$ for simplicity. Call $\Sigma$ \textbf{closed} in $U$, if $Clos(\Sigma)\subset \subset U$; Call $\Sigma$ \textbf{regular} in $U$, if $Sing(\Sigma)=\emptyset$.
     
     The Fermi coordinate of $M$ near $\Sigma$ is a parametrization of some neighborhood of $\Sigma$ in $M$ by a sub-domain $\Omega\subset\Sigma\times \RR$, \[
       \Omega\ni (x, z) \mapsto exp^M_x(z\cdot\nu(x)) \in M    \]

     The Riemannian metric on $\Sigma$ is induced from $(M, g)$. We usually omit the subscript for differential operators $\nabla_{\Sigma}$, $div_{\Sigma}$, $\nabla^k_{\Sigma}$ and $\Delta_{\Sigma}$ etc. of $\Sigma$. Also write 
     \begin{align*}
      dist_{\Sigma}\ &\ \text{the intrinsic distant on }Clos(\Sigma) \\
      B_r(p)\ &:=\ \Sigma\cap B^M_r(p) \\
      \cB_r(p)\ &:=\ \{y\in \Sigma: dist_{\Sigma}(p, y)<r\} \text{ the intrinsic open ball in }\Sigma,\ p\in Clos(\Sigma)\\
      \vec{H}_{\Sigma}\ &:=\ -div_{\Sigma} (\nu) \cdot \nu = H_{\Sigma}\cdot\nu,\ \text{ the mean curvature vector of }\Sigma \\
      A_{\Sigma}\ &:=\ -\nabla \nu,\ \text{ the 2nd fundamental form of }\Sigma
     \end{align*}
     Recall that $\Sigma$ is called \textbf{minimal} if $H_{\Sigma} = 0$. 

     A minimal hypersurface $(\Sigma,\nu) \subset M$ is called \textbf{stable} in $U$ if  \[
       \frac{d^2}{ds^2}{\Big|}_{s=0}\scH^n(e^{sX}(\Sigma)) \geq 0\ \ \ \ \forall X\in \scX_c(U)
     \]
     By \cite{SchoenSimon81}, this is equivalent to $\scH^{n-7+\epsilon}(Sing(\Sigma)) = 0$, $\forall \epsilon>0$ and that
     \begin{align}
      Q_{\Sigma}(\phi, \phi):=\int_{\Sigma} |\nabla \phi|^2 - (|A_{\Sigma}|^2 + Ric_M(\nu, \nu))\phi^2 \geq 0\ \ \ \ \forall \phi\in C^1_c(\Sigma \cap U)  \label{Pre, quadratic assoc to Jac oper} 
     \end{align}
     Let 
     \begin{align}
      L_{\Sigma}:= \Delta + |A_{\Sigma}|^2 + Ric_M(\nu, \nu)   \label{Pre, Def Jac oper} 
     \end{align}
     be the operator associated to $Q_{\Sigma}$, called \textbf{Jacobi operator}. Any solution $u\in C^2(\Sigma)$ to $L_{\Sigma} u = 0$ is called a \textbf{Jacobi field} on $\Sigma$. 
     
     For $E\subset \Sigma$ and $u: E\to \RR$, let \[
      graph_{\Sigma, g}(u):= \{exp^M_x(u(x)\cdot\nu(x)) : x\in E\}     \]
     be the graph of $u$ over $\Sigma$ under metric $g$, and we shall omit $g$ if there's no confusion. If $E$ is a smooth domain, for sufficiently small $u$, $graph_{\Sigma}(u)$ is a smooth hypersurface, oriented by taking normal field having positive inner product with $\nu$. Let $\scM_{\Sigma}$ be the Euler-Lagrangian operator associated to the area function of $graph_{\Sigma}(u)$, see section \ref{Subsec, Min surface near sing} for precise definition and discussion.

%%%%%%%%%%%%%%%%%%%%%%%%%%%%%%%%%%%%%%%%%%%%%%%%%%%%%%%%%%%%%%%%%%%%%%%%%%%%
%%%%%%%%%%%%%%%%%%%%%%%%%%%%%%%%%%%%%%%%%%%%%%%%%%%%%%%%%%%%%%%%%%%%%%%%%%%%
%%%%%%%%%%%%%%%%%%%%%%%%%%%%%%%%%%%%%%%%%%%%%%%%%%%%%%%%%%%%%%%%%%%%%%%%%%%%
%%%%%%%%%%%%%%%%%%%%%%%%%%%%%%%%%%%%%%%%%%%%%%%%%%%%%%%%%%%%%%%%%%%%%%%%%%%%
%%%%%%%%%%%%%% Subsection: Currents, varifolds and convergece %%%%%%%%%%%%%%
%%%%%%%%%%%%%%%%%%%%%%%%%%%%%%%%%%%%%%%%%%%%%%%%%%%%%%%%%%%%%%%%%%%%%%%%%%%%
%%%%%%%%%%%%%%%%%%%%%%%%%%%%%%%%%%%%%%%%%%%%%%%%%%%%%%%%%%%%%%%%%%%%%%%%%%%%
%%%%%%%%%%%%%%%%%%%%%%%%%%%%%%%%%%%%%%%%%%%%%%%%%%%%%%%%%%%%%%%%%%%%%%%%%%%%
%%%%%%%%%%%%%%%%%%%%%%%%%%%%%%%%%%%%%%%%%%%%%%%%%%%%%%%%%%%%%%%%%%%%%%%%%%%%
     
     \subsection{Currents, varifolds and convergence} \label{Subsec, Currents & varifolds}
     
      We recall the basic notions of currents and varifolds and refer the readers to \cite{Simon83_GMT} and \cite{Federer69} for details. 
      For $1\leq k\leq n+1$, let
      \begin{align*}
       \bfI_k(M)\ &\ \text{the spce of integral }k \text{ currents on }M \\
       \cZ_k(M)\ &:= \{T\in \bfI_k(M): \partial T = 0\} \\
       \cV_k(M)\ &\ \text{the closure of }\cZ_k(M) \text{ under varifold weak topology} \\
       \cI\cV_k(M)\ &\ \text{the space of integral }k\text{-varifolds on }M
      \end{align*}
      If $\Sigma\hookrightarrow M$ is an immersed $k$-dimensional submanifold with finite volume, then let $|\Sigma|\in \cI\cV_k(M)$ be the associated integral varifolds. If further $\Sigma$ is oriented, let $[\Sigma]\in \bfI(M)$ be the associated integral current. For $T\in \bfI_k(M)$, let $|T|$ and $\|T\|$ be its associated integral varifold and Radon measure correspondingly.
      
      For $U\subset M$ open, let $\cF_{U, g}$ and $\bfM_{U, g}$ be the flat metric and mass norm on $\bfI_k(U)$ with respect to metric $g$. We may omit the subscript $g$ if there's no confusion and the subscript $U$ if $U = M$. 
      Let $\bfF_g$ be the varifolds metric on $\cV_k(M)$, i.e. the metric compatible with weak converges. 
      The $\bfF_g$-metric on integral currents is defined to be $\bfF_g(S, T):= \cF_g(S, T)+ \bfF_g(|S|, |T|)$, $S, T\in \bfI_k(M)$.
      
      Let $spt(V)$ be the support of $V\in \cV_k(M)$; $spt(T):= spt(|T|)$ for $T\in \bfI_k(M)$.
      
      For a proper Lipschitz map $f$, $f_{\sharp}$ be the push forward of varifolds or currents.
      
      For $E\subset M$ measurable, $T\in \bfI_k(M)$ and $V\in \cV_k(M)$, $T\llcorner E$ and $V\llcorner E$ be the restrictions onto $E$.
      
      Let $T\in \cZ_k(M)$, $E\supset \spt T$. Call $T$ \textbf{homologically area minimizing} in $E$ if $\forall W\subset \subset M$, \[
       \bfM_W(T+\partial P) \geq \bfM_W(T)\ \ \ \ \forall P\in I_{k+1}(M)\text{ with } spt(P)\subset E\cap W     \]
            
      For $U\subset M$ open, $V\in \cV_k(U)$ and $\scX_c(U)$, let \[
       \delta V(X) = \frac{d}{dt}\Big|_{t=0} \bfM\big((e^{tX})_{\sharp}V \big) = \int div^{\pi}X\ dV(x, \pi)    \]
      be the first variation of $V$, where $div^{\pi}X = \sum \langle\nabla^M_{e_i} X, e_i \rangle$, $\{e_i\}$ be an othonormal basis of $\pi\in \bfG(T_x M, k)$.
      Call $V$ stationary in $U$ if $\delta V(X) = 0$ for every $X\in \scX_c(U)$.

%%%%%%%%%%%%%%%%%%%%%%%%%%%%%%%%%%%%%%%%%%%%%%%%%%%%%%%%%%%%%%%%%%%%%%%%%%%%
%%%%%%%%%%%%%%%%%%%%%%%%%%%%%%%%%%%%%%%%%%%%%%%%%%%%%%%%%%%%%%%%%%%%%%%%%%%%
%%%%%%%%%%%%%%%%%%%%%%%%%%%%%%%%%%%%%%%%%%%%%%%%%%%%%%%%%%%%%%%%%%%%%%%%%%%%
%%%%%%%%%%%%%%%%%%%%%%%%%%%%%%%%%%%%%%%%%%%%%%%%%%%%%%%%%%%%%%%%%%%%%%%%%%%%
%%%%%%%%%%%%%%%%%%%% Subsection: Geometry of hypercones %%%%%%%%%%%%%%%%%%%%
%%%%%%%%%%%%%%%%%%%%%%%%%%%%%%%%%%%%%%%%%%%%%%%%%%%%%%%%%%%%%%%%%%%%%%%%%%%%
%%%%%%%%%%%%%%%%%%%%%%%%%%%%%%%%%%%%%%%%%%%%%%%%%%%%%%%%%%%%%%%%%%%%%%%%%%%%
%%%%%%%%%%%%%%%%%%%%%%%%%%%%%%%%%%%%%%%%%%%%%%%%%%%%%%%%%%%%%%%%%%%%%%%%%%%%
%%%%%%%%%%%%%%%%%%%%%%%%%%%%%%%%%%%%%%%%%%%%%%%%%%%%%%%%%%%%%%%%%%%%%%%%%%%%

     \subsection{Geometry of hypercones} \label{Subsec, Geom of cone}
      Let $C$ be a hypercone in $\RR^{n+1}$ with normal field $\nu$, $S :=\partial \BB^{n+1}_1\cap C$ be the cross section of $C$. Call $C$ a \textbf{regular cone}, if $S$ is a closed smooth hypersurface in $\SSp^n$, in which case $Sing(C) = \{0\}$. Let $A_S$ be the second fundamental form of $(S, \nu|_S)$ in $\SSp^n$. Parametrize $C$ by 
      \begin{align}
       (0,+\infty)\times S \to C,\ \ \ (r,\omega)\mapsto x:=r\omega  \label{Pre, param cone}    
      \end{align}
     
      Now suppose $C\subset \RR^{n+1}$ be a regular stable minimal hypercone. Then by \cite{Simons68}, the Jacobi operator of $C$ is decomposed w.r.t. this parametrization, \[
       L_C = \partial_r^2 + \frac{n-1}{r}\partial_r + \frac{1}{r^2}(\Delta_S + |A_S|^2)   \]
      Let $\cL_S := \Delta_S +|A_S|^2$, $\mu_1 <\mu_2 \leq ... \nearrow +\infty$ be the eigenvalues of $-\cL_S$ on $L^2(S)$ and $w_1, w_2, ...$ be a corresponding family of $L^2$-orthonormal eigenfunctions, $ w_1>0 $. By the stability of $C$, \[
       \mu_1\geq -(\frac{n-2}{2})^2    \]
      Following \cite{CaffHardtSimon84}, call $C$ \textbf{strictly stable}, if $\mu_1 > (\frac{n-2}{2})^2 $.
     
      By \cite[section 1]{CaffHardtSimon84}, for every $0\leq R_1 < R_2 \leq +\infty$, general solution of $L_C u = f$ on $C\cap \AAa_{R_1, R_2}(0)$ is given by 
      \begin{align}
       u(r, \omega) = \sum_{k\geq 1} (u_k(r) + v_k(r))w_k(\omega)  \label{Pre, nonhom Sol Jacob}
      \end{align}
      where 
      \begin{align}
       u_k(r, f) :=
       \begin{cases}
        \frac{-1}{2b_k}{\big(}r^{\gamma^+_k} \int_r^1 s^{-\gamma_k^+ +1}f_k(s)\ ds - r^{\gamma^-_k} \int_r^1 s^{-\gamma_k^- +1}f_k(s)\ ds{\big)}, & b_k\neq 0;\\
        r^{-\frac{n-2}{2}}{\big(}\log r\int_0^r s^{\frac{n}{2}}f_k(s)\ ds - \int_0^r s^{\frac{n}{2}}\log sf_k(s)\ ds{\big)}, & b_k = 0.
       \end{cases} \label{Pre, nonhom Spe Sol Jacob}
      \end{align}
      be a special solution, \[ f_k(r) := \int_S f(r, \omega)w_k(\omega)\ d\omega \]
      be the Fourier coefficients of $f$; and 
      \begin{align}
       v_k(r) = v_k(r; c_k^+, c_k^-) := 
       \begin{cases}
        c_k^+ r^{\gamma^+_k} + c_k^- r^{\gamma^-_k} = r^{-\frac{n-2}{2}}(c_k^+ r^{b_k} +c_k^- r^{-b_k}), & b_k\neq 0; \\
        r^{-\frac{n-2}{2}}(c_k^+ + c_k^- \log r), & b_k = 0. 
       \end{cases} \label{Pre, hom Sol Jacob}
      \end{align}
      be the general homogeneous solutions, $c_k^{\pm}\in \RR$; $\gamma_k^{\pm} = -\frac{n-2}{2} \pm b_k$ be the two solutions of 
      \begin{align}
       \gamma^2 + (n-2)\gamma -\mu_k = 0    \label{Pre, power of cone Jac} 
      \end{align} 
      and $b_k = \sqrt{(\frac{n-2}{2})^2+\mu_k}$. 
      For later reference, write $\Gamma_C := \{\gamma_j^{\pm} : j\geq 1\}$ and $\forall \gamma\in \Gamma_C$, write $W_{\gamma}:= span_{\RR}\{w_j: \gamma_j^{\pm} = \gamma\}$ be the eigenspace of $\cL_S$. \\

      We list some corollary of this expansion formula. In what follows, $B_r := C\cap \BB^{n+1}_r$, $A_{r, s} = C\cap \AAa^{n+1}_{r,s}$.
      \begin{Lem} \label{Lem_Pre, Hardy-typed inequ}
       For each $\phi\in C_c^\infty(C)$ we have, \[
        \int_C |\nabla_C \phi|^2 - |A_C|^2\phi^2 \geq (\mu_1+(\frac{n-2}{2})^2)\int_C \frac{1}{r^2}\phi^2     \]
      \end{Lem}
      \begin{proof}
       Write $\phi(r, \omega) = \sum_{k\geq 1} \phi_k(r)w_k(\omega)$, plug in the integration and change variables. The lemma follows from the Hardy inequality \[
        \int_0^{+\infty} 4a'(r)^2 - a(r)^2/r^2\ dr \geq 0\ \ \ \ \forall a\in C_c^1(0, +\infty)      \]
      \end{proof}
      
%      \paragraph{Asymptotic rate bound}
      \begin{Lem} \label{Lem_Pre, Bdness of Jac field near 0/infty}
       If $\gamma \in \RR\setminus \Gamma_C$, $u$ is a Jacobi field on $B_R$ (resp. $C\setminus B_R$). Suppose \[ 
        I^2_u(r):= \int_{A_{r, 2r}} u^2(t, \omega)t^{-n}\ dx     \] 
       satisfies that for some $r_0 <R$ (resp. $r_0>R$) and $C>0$, 
       \begin{align}
        I^2_u(r) \leq Cr^{2\gamma} ,\ \ \ \forall\ r<r_0\ \ \ \ \ (\text{resp. } \forall\ r>r_0)
       \end{align}
       Then \[
        u(r, \omega) = \sum_{k\geq 1} v_k(r; c_k^+, c_k^- )w_k(\omega)     \]  
       where $c_k^{\pm} = 0$ if $\gamma_k^{\pm} < \gamma$ (resp. $\gamma_k^{\pm} > \gamma$); $v_k(r; c_k^+, c_k^-)$ is defined in (\ref{Pre, hom Sol Jacob}).
      \end{Lem}
      \begin{proof}
       We only prove for Jacobi fields on $B_R$ since the proof for Jacobi fields on $C\setminus B_R$ is similar.
       By (\ref{Pre, nonhom Sol Jacob}), \[
        v_k(r; c_k^+, c_k^-) = \int_S u(r,\omega) w_k(\omega)\ d\omega \leq (\int_S u(r, \omega)^2\ d\omega )^{1/2}     \]
       Hence, for all $r \in (0, r_0)$, \[
        \int_r^{2r} t^{-1} v_k(t; c_k^+, c_k^-)^2\ dt \leq \int_r^{2r}t^{-1}dt \int_S u(t, \omega)^2\ d\omega \leq C(n)I^2_u(r) \leq Cr^{2\gamma +1}     \]
       And then by the expression (\ref{Pre, hom Sol Jacob}) of $v_k(r; c_k^+, c_k^-)$, $c_k^{\pm} = 0$ if $\gamma_k^{\pm} < \gamma$.
      \end{proof}
      
      Combined with Harnack inequality and that $w_k$ changes signs when $k\geq 2$, the following lemma is a corollary of lemma \ref{Lem_Pre, Bdness of Jac field near 0/infty}, and the proof is left to readers. See also \cite[Lemma 4.2]{White89}.
      \begin{Lem} \label{Lem_Pre, pos Jacob field}
       If $u$ is a positive Jacobi field on $B_R$ (resp. $C\setminus B_R$), then 
       \begin{align*}
        u(r, \omega) &= v_1(r)w_1(\omega) + \sum_{k\geq 2}c_k r^{\gamma_k^+}w_k(\omega)  \\
        (\ \text{resp. } u(r, \omega) &= v_1(r)w_1(\omega) + \sum_{k\geq 2}c_k r^{\gamma_k^-}w_k(\omega)\ )
       \end{align*}
       for some $c_k \in \RR$, where $v_1$ is defined in (\ref{Pre, hom Sol Jacob}) for some $c_1^{\pm}$ such that $v_1>0$ on $B_R$ (resp. on $C\setminus B_R$).
       
       In particular, if $u$ is a positive Jacobi field on $C$, then 
       \begin{align*}
        u = \begin{cases}
           (c^+ r^{\gamma_1^+} + c^- r^{\gamma_1^-}) w_1, & \text{ if }\gamma_1^+ > \gamma_1^- \\
           cr^{-\frac{n-2}{2}} w_1, & \text{ if }\gamma_1^+ = \gamma_1^-            
         \end{cases}
       \end{align*}
       for some $c^{\pm}, c \geq 0$.
      \end{Lem}

%      \paragraph{Solving Jacobi field equation with prescribed asymptotics}
      (\ref{Pre, nonhom Sol Jacob}) also guarantees the existence of  solutions with prescribed growth or decaying bound, which is analogous to \cite[Theorem 1.1]{CaffHardtSimon84}. We start with some notations.
      
      For $\sigma\in \RR\setminus \Gamma_C$, define $\Pi^+_{\sigma}$ (resp. $\Pi^-_{\sigma}$) $: L^2(S)\to L^2(S)$ to be the orthogonal projection onto $\bigoplus_{\gamma_j^+ >\sigma} W_{\gamma_j^+}$ (resp. $\bigoplus_{\gamma_j^-<\sigma} W_{\gamma_j^-}$). Note that $1-\Pi_{\sigma}^{\pm}$ are the projection onto their $L^2$-orthogonal complement; when $\sigma<\gamma_1^+$ (resp. $\sigma>\gamma_1^-$), $\Pi^+_{\sigma} = id$ (resp. $\Pi^-_{\sigma} = id$).
      
      For $E\subset C$ and $f\in L^2_{loc}(E)$, define 
      \begin{align}
       \|f\|^2_{L^2_{\sigma}(E)}:= \int_E |f(x)|^2\cdot |x|^{-2\sigma - n}\ dx \label{Pre_def L^2_sigma}
      \end{align}
      Also define for $k\geq 1$, $\|f\|^2_{W^{k, 2}_{\sigma}(E)}:= \sum_{j=0}^k \||\nabla^j f|\|^2_{L^2_{\sigma-j}}$. $L^2_{\sigma}(E)$ and $W^{k,2}_{\sigma}(E)$ will be the space of functions with finite $\|\cdot\|_{L^2_{\sigma}}$ and $\|\cdot\|_{W^{k, 2}_{\sigma}}$ norm respectively.
      
      \begin{Lem} \label{Lem_Pre, Ext sol of Jac, prescribed asymp}
       Suppose $\sigma \in \RR \setminus \Gamma_C$; $\varphi, \psi \in L^2(S)$; $f\in L_{\sigma-2}^2(B_1)$ (resp. $f\in L_{\sigma-2}^2(C\setminus B_1)$). 
       Then there's a unique $u \in L^2_{\sigma}\cap W^{2,2}_{loc}(B_1)$ (resp. $u\in L^2_{\sigma} \cap W^{2,2}_{loc}(C\setminus B_1)$) such that
       \begin{align*}
        \begin{cases}
         L_C u = f \ & \text{ on }B_1 \ \ \ \ \ \ \text{ (resp. } L_C u = f \ \text{ on }C\setminus B_1 \text{)}\\
         \Pi^+_{\sigma} u = \Pi^+_{\sigma} \varphi\ & \text{ on }\partial B_1\ \ \ \ \text{ (resp. }\Pi^-_{\sigma} u = \Pi^-_{\sigma} \varphi\  \text{ on }\partial B_1 \text{)}\\
         (1-\Pi^-_{\sigma})\partial_r u = (1-\Pi^-_{\sigma})\psi\ & \text{ on }\partial B_1 \ \ \ \ \text{ (resp. }(1-\Pi^+_{\sigma}) \partial_r u = (1-\Pi^+_{\sigma}) \psi\  \text{ on }\partial B_1 \text{)}
        \end{cases} 
       \end{align*}
       Moreover, $u$ satisfies the estimate 
       \begin{align*}
         \sup_{t\in (0, 1)}\|u(t, \cdot)\|_{L^2(S)}\cdot t^{-\sigma} + \|u\|_{L^2_{\sigma}} &\leq C(C, \sigma)(\|f\|_{L^2_{\sigma-2}(B_1)} + \|\varphi\|_{L^2(S)} + \|\psi\|_{L^2(S)}) \\
        \text{( resp. }  \sup_{t\in (1, +\infty)}\|u(t, \cdot)\|_{L^2(S)}\cdot t^{-\sigma} + \|u\|_{L^2_{\sigma}} &\leq C(C, \sigma)(\|f\|_{L^2_{\sigma-2}(C\setminus B_1)} + \|\varphi\|_{L^2(S)} + \|\psi\|_{L^2(S)}) \text{ )}
       \end{align*}
      \end{Lem}
      The proof is almost the same as \cite[Theorem 1.1]{CaffHardtSimon84} and is left to readers.
%      \begin{Pf}
%       Following (\ref{Pre, nonhom Sol Jacob})-(\ref{Pre, power of cone Jac}), consider the solution $u$ in the form \[
%        u(r, \omega) = \sum_{k\geq 1} (u_k(r) + v_k(r; c_k^+, c_k^-))w_k(\omega)      \]
%       If at least one of $\gamma_k^{\pm} > \sigma$, appropriately choose $c_k^{\pm}$ such that 
%       \begin{align}
%        |u_k(r) + v_k(r; c_k^+, c_k^-)|^2 \leq C(C, \sigma)r^{2\sigma} \int_1^{+\infty} s^{-2\sigma + 3}f_k^2(s)\ ds  \label{Pre, Asymp of nonhom sol}       
%       \end{align}
%       And if both $\gamma_k^{\pm}<\sigma$, choose $c_k^{\pm} = 0$. It's easy to verify that (\ref{Pre, Asymp of nonhom sol}) is also true for such $k$. Hence, the lemma is proved by summing up (\ref{Pre, Asymp of nonhom sol}).
%      \end{Pf}

      The following perturbed version of lemma \ref{Lem_Pre, Ext sol of Jac, prescribed asymp} is used in section \ref{Sec, One-sided Perturb}.
      \begin{Cor} \label{Cor_Pre, sol perturbed Jac equ, prescribed asymp}
       Suppose $\sigma\in \RR\setminus \Gamma_C$. Then there exists $\varepsilon=\varepsilon(C, \sigma)>0$ such that, if $f\in L^2_{\sigma-2}(B_1)$; $\scR$ is a second order differential operator satisfying \[
        \|\scR u\|_{L^2_{\sigma-2}} \leq \epsilon\|u\|_{W^{2, 2}_{\sigma}}\ \ \ \forall u\in W^{2,2}_{\sigma}(u)    \]
       Then there exists a unique solution $u\in W^{2,2}_{\sigma}(B_1)$ to the equation 
       \begin{align}
       \begin{cases}
        L_C u + \scR u = f\ & \text{ on }B_1 \\
        \Pi_{\sigma}^+ u = 0\ & \text{ on }\partial B_1 \\
        (1-\Pi_{\sigma}^-) \partial_r u = 0\ & \text{ on }\partial B_1 
       \end{cases} \label{Pre_Perturbed Jac equ with weight sigma}
       \end{align}
       Moreover, $u$ satisfies the estimate \[
         \sup_{t\in (0, 1)}\|u(t, \cdot)\|_{L^2(S)}\cdot t^{-\sigma} + \|u\|_{L^2_{\sigma}} \leq C(C, \sigma)\|f\|_{L^2_{\sigma-2}(B_1)}   \]
      \end{Cor}
      \begin{proof}
       Apply Banach fixed point theorem on $T: W_{\sigma}^{2,2}(B_1)\to W_{\sigma}^{2,2}(B_1)$, $v= Tu$ solving
       \begin{align*}
       \begin{cases}
        L_C v  = f -\scR u \ & \text{ on }B_1 \\
        \Pi_{\sigma}^+ v = 0\ & \text{ on }\partial B_1 \\
        (1-\Pi_{\sigma}^-) \partial_r v = 0\ & \text{ on }\partial B_1 
       \end{cases} 
       \end{align*}
       Combine lemma \ref{Lem_Pre, Ext sol of Jac, prescribed asymp} and $L^2$ estimate for elliptic PDE to show that $T$ is a contraction map provided $\epsilon$ is small enough.
      \end{proof}

%      \paragraph{Growth rate monotonicity}
      To prove the finiteness of associated Jacobi field in section \ref{Sec, Asymp & Asso Jac field}, we need the following growth rate monotonicity.
      \begin{Lem}  \label{Lem_Pre, growth rate mon for Jac}
       Let $\sigma \in \RR\setminus \Gamma_C$. Then there exists $K_0 = K_0(\sigma, C)>2$ such that if $0\neq u\in L^2_{\sigma}(C\setminus B_1)$ is a Jacobi field, then 
       \begin{align}
        J^{\sigma}_u(r, s):= \int_{A_{r, s}} u^2(t, \omega)t^{-n-2\sigma}\ dx     \label{Pre, J^sigma_u(r, s)}       
       \end{align}
       satisfies $J^{\sigma}_u(Kr, K^2r) < J^{\sigma}_u(r, Kr)$, $\forall r>1$, $\forall K\geq K_0$. 
      \end{Lem}
      \begin{proof}
       By the same argument as lemma \ref{Lem_Pre, Bdness of Jac field near 0/infty}, $u = \sum_{k\geq 1}v_k(r; c_k^+, c_k^-) w_k(\omega)$, where $c_k^{\pm} = 0$ if $\gamma_k^{\pm}>\sigma$. And then, \[
        J^{\sigma}_u(r, Kr) = \int_r^{Kr}t^{n-1-n-2\sigma}\ dt \int_S u^2\ d\omega = \sum_{k\geq 1}\int_r^{Kr} t^{-1-2\sigma}v_k(t; c_k^+, c_k^-)^2\ dt   \]
       Hence, to show the lemma, it suffices to verify that for each $k\geq 1$, 
       \begin{align}
        \int_{Kr}^{K^2r} t^{-1-2\sigma}v_k(t; c_k^+, c_k^-)^2\ dt < \int_r^{Kr} t^{-1-2\sigma}v_k(t; c_k^+, c_k^-)^2\ dt  \label{Pre, single term growth rate mon}
       \end{align}
       provided $c_k^{\pm}$ do not vanish simultaneously. If one of $\gamma_k^{\pm}$ is greater than $\sigma$, then the corresponding $c_k^{\pm}$ vanishes and (\ref{Pre, single term growth rate mon}) immediately follows from  (\ref{Pre, hom Sol Jacob}). If not, (\ref{Pre, single term growth rate mon}) would be derived from (\ref{Pre, hom Sol Jacob}) and the following:\\
       \textbf{Claim 1}: suppose $\alpha<0$, $K\geq 2$; $\cI_K(r; c, c'):= \int_r^{Kr} (cs^{\alpha} + c's^{\alpha}\log s)^2s^{-1}\ ds$. Then $\exists\ K_1 = K_1(\alpha)\geq 2$ such that if $K\geq K_1$, then \[
         \cI_K(Kr; c, c') < \cI_K(r; c, c')\ \ \ \ \forall r>0 \text{ and } c,c'\in \RR \text{ doesn't vanish simultaneouly}    \] 
       \textbf{Claim 2}: suppose $\alpha<\beta<0$, $K\geq 2$; $\tilde{\cI}_K(r; c, c'):= \int_r^{Kr} (cs^{\alpha} + c's^{\beta})^2s^{-1}\ ds$. Then $\exists\ K_2=K_2(\alpha, \beta)\geq 2$ such that if $K\geq K_2$, then \[
         \tilde{\cI}_K(Kr; c, c') < \tilde{\cI}_K(r; c, c')\ \ \ \ \forall r>0 \text{ and } c,c'\in \RR \text{ doesn't vanish simultaneouly}    \] 
       We shall prove claim 1, the proof of claim 2 is similar.\\
       \textbf{Proof of claim 1}: Claim is trivially true when $c' = 0$; If $c'\neq 0$, by the arbitrariness of $r$, WLOG $(c, c')= (0, 1)$. Then, \[
        \cI_K(r) = \frac{1}{2\alpha} \big(s^{2\alpha}(\log s)^2\big)_r^{Kr} - \frac{1}{2\alpha^2}(s^{2\alpha}\log s)_r^{Kr} + \frac{1}{4\alpha^3}(s^{2\alpha})_r^{Kr}   \]
       and 
       \begin{align*}
        \cI_K(Kr) - \cI_K(r) = \frac{1}{2\alpha}\Big[ 
          & (K^{2\alpha}-1)^2\cdot r^{2\alpha}(\log r)^2 \\
        + & \big( -\frac{1}{\alpha}(K^{2\alpha}-1)^2 + 4(K^{2\alpha}-1)K^{2\alpha}\log K \big)\cdot r^{2\alpha}\log r \\
        + & \big( \frac{1}{2\alpha^2}(K^{2\alpha}-1)^2 - \frac{2}{\alpha}(K^{2\alpha}-1)K^{2\alpha}\log K + 2(2K^{2\alpha} - 1)K^{2\alpha}(\log K)^2 \big)\cdot r^{2\alpha} \Big]
       \end{align*}
       Hence, to see $\cI_K(Kr)< \cI_K(r)$ for all $r>0$, it suffices to show that the discriminant $\Delta(K)$ is negative. Since for $K>>1$, $K^{2\alpha}(\log K)^2<<1$, and the coefficient for the constant term of $\Delta(K)$ is $1/\alpha^2 - 4/(2\alpha^2) < 0$, we have $\Delta(K)<0$ for $K\geq K_1(\alpha)>>1$. This proves the claim 1.
      \end{proof}
      
      The following corollary is a quantitative version of lemma \ref{Lem_Pre, growth rate mon for Jac}, and is proved by contradiction. We left its proof to readers.
      \begin{Cor} \label{Cor_Pre, quant growth rate mon for Jac}
       Let $\sigma \in \RR\setminus \Gamma_C$, $K_0(\sigma, C)>2$ be as in lemma \ref{Lem_Pre, growth rate mon for Jac}, $K\geq K_0$. Then $\exists \delta_0 = \delta_0(C, \sigma, K)>0$, $N = N(C, \sigma, K)\geq 2$ such that if $0\neq u\in W^{1,2}_{loc}\cap L^2(A_{1, K^{N+1}})$ is a weak solution of 
       \begin{align}
        div_C\big(\nabla_C u +b_0(x, u, \nabla_C u) \big) + |A_C|^2u + b_1(x, u, \nabla_C u) = 0  \label{Pre, Pertub Jac field equ}       
       \end{align}
       on $A_{1, K^{N+1}}$, where \[
        |b_0(x, z, p)| + |b_1(x, z, p)| \leq \delta_0\big(|z|/|x|+ |p|+ J^{\sigma}_u(K, K^2)\big)\ \ \ \ \forall x\in C, z\in \RR, p\in T_x C  \] 
       And if further $\ \sup_{2\leq j\leq N}J^{\sigma}_u(K^j, K^{j+1}) \leq J^{\sigma}_u(K, K^2)$,
       then $J^{\sigma}_u(K, K^2)< J^{\sigma}_u(1, K)$.
      \end{Cor}

%      \paragraph{{\large Hardt-Simon foliation}}      
      Now suppose $(C,\nu) \subset \RR^{n+1}$ be a regular area-minimizing hypercone, i.e. $[C]$ is an area-minimizing current in $\RR^{n+1}$. Then $\RR^{n+1}\setminus C$ has exactly 2 connected component, denoted by $E_{\pm}$, with $\nu$ pointing in $E_+$. Recall the following existence and uniqueness theorem by Hardt and Simon.
      \begin{Thm}\cite[Theorem 2.1]{HardtSimon85} \label{Thm_Pre, H-S foliation}
       Let $E$ be either one of the component $E_{\pm}$. Then there esists a smooth area-minimizing hypersurface $\Sigma\subset E$ with $dist_{\RR^{n+1}}(0, \Sigma) = 1$, $[\Sigma]= \partial P$ for some $P\in \bfI_{n+1}(\RR^{n+1})$ with $spt(P) \subset E$. 
       
       Furthermore, $\exists h\in C^{\infty}(C)$ such that $\Sigma = graph_C(h)$ outside some large ball, and when $R\to +\infty$, either 
       \begin{align}
        h(R,\omega) = \begin{cases}
          cR^{\gamma_1^+}w_1(\omega) + O(R^{\gamma_1^+-\epsilon})\ &\text{ if }\mu_1 > (\frac{n-2}{2})^2 \\
          cR^{\gamma_1^+}\log R\ w_1(\omega) + O(R^{\gamma_1^+-\epsilon})\ &\text{ if }\mu_1 = (\frac{n-2}{2})^2
         \end{cases} \label{Pre, asymp of one-sided, strict}
       \end{align}
       or 
       \begin{align}
        h(R, \omega) = cR^{\gamma_1^-}w_1(\omega) + O(R^{\gamma_1^- - \epsilon}) \label{Pre, asymp of one-sided, nonstrict}
       \end{align}
       where $\gamma_1^{\pm}$ is defined in (\ref{Pre, power of cone Jac}).
       
       Moreover, $\Sigma$ foliates $E$ by rescaling. And any other area-minimizing integral $n$-cycle $\partial [Q]$ with $Q\subset E$ is a rescaling of $[\Sigma]$. 
      \end{Thm} 

      We shall prove an analogue of this regularity theorem in lemma \ref{Lem_Ass Jac, Reg of infty one-side perturb} using a different method. 
      
      Let $\Sigma_{\pm}\subset E_{\pm}$ be the smooth hypersurface in Theorem~\ref{Thm_Pre, H-S foliation}. Pick the normal field $\nu_{\pm}$ on $\Sigma_{\pm}$ by setting \[
       \pm\nabla_{\RR^{n+1}}(\frac{|x|^2}{2})\cdot \nu_{\pm}  > 0   \]
      And let $h_{\pm}$ be the graphical functions of $\Sigma_{\pm}$ outside a large ball on $C$. Following \cite{HardtSimon85}, call $C$ \textbf{strictly minimizing} if (\ref{Pre, asymp of one-sided, strict}) holds for $h_{\pm}$. Equivalent definitions of strict minimizing is discussed in \cite[Theorem 3.2]{HardtSimon85}.

%%%%%%%%%%%%%%%%%%%%%%%%%%%%%%%%%%%%%%%%%%%%%%%%%%%%%%%%%%%%%%%%%%%%%%%%%%%%
%%%%%%%%%%%%%%%%%%%%%%%%%%%%%%%%%%%%%%%%%%%%%%%%%%%%%%%%%%%%%%%%%%%%%%%%%%%%
%%%%%%%%%%%%%%%%%%%%%%%%%%%%%%%%%%%%%%%%%%%%%%%%%%%%%%%%%%%%%%%%%%%%%%%%%%%%
%%%%%%%%%%%%%%%%%%%%%%%%%%%%%%%%%%%%%%%%%%%%%%%%%%%%%%%%%%%%%%%%%%%%%%%%%%%%
%%%%%%%%%%%%% Subsection: Minimal hypersurface near singularity %%%%%%%%%%%%
%%%%%%%%%%%%%%%%%%%%%%%%%%%%%%%%%%%%%%%%%%%%%%%%%%%%%%%%%%%%%%%%%%%%%%%%%%%%
%%%%%%%%%%%%%%%%%%%%%%%%%%%%%%%%%%%%%%%%%%%%%%%%%%%%%%%%%%%%%%%%%%%%%%%%%%%%
%%%%%%%%%%%%%%%%%%%%%%%%%%%%%%%%%%%%%%%%%%%%%%%%%%%%%%%%%%%%%%%%%%%%%%%%%%%%
%%%%%%%%%%%%%%%%%%%%%%%%%%%%%%%%%%%%%%%%%%%%%%%%%%%%%%%%%%%%%%%%%%%%%%%%%%%%

     \subsection{Minimal hypersurface near singularity} \label{Subsec, Min surface near sing}
      Let $\Sigma \subset (M,g)$ to be a locally stable minimal hypersurface with normal field $\nu$, $p\in Sing(\Sigma)$. Suppose that $V$ is a tangent varifold of $|\Sigma|$ at $p$. By \cite{SchoenSimon81} and \cite[Theorem B]{Ilmanen96}, $V = |C|$ for some stable minimal hypercone $C\subset \RR^{n+1}$. Call $p$ a \textbf{strongly isolated singularity} if $C$ is a regular cone.

      Suppose now that $C$ is regular. Identify $C\subset T_p M = \RR^{n+1}$ and $\Sigma\cap B^M_{r_0}(p) \hookrightarrow (T_pM = \RR^{n+1}, \exp_p^*g)$ by the exponential map of $M$ at $p$. We shall abuse the notations and still write $\exp_p^*g$ as $g$. 
      
      Since $C$ is regular, by \cite{Simon83Asym}, $|C|$ is the unique tangent cone of $|\Sigma|$ at $p$ and there's a smooth function $v \in C^{\infty}(C)$ such that for each integer $j\geq 0$ \[
       r^{j-1}|\nabla^j v|(r,\cdot) \to 0 \ \ \ \ \text{as }r\to 0   \]
      and \[  \text{graph}_C(v)\cap B^M_{r_0} = \Sigma \cap B^M_{r_0}  \]
      For some $r_0>0$. Parametrize $\Sigma \cap B^M_{r_0}$ by 
      \begin{align} 
       \Sigma\cap B^M_{r_0} \to (0,+\infty)\times S\ \ \ x \mapsto (r,\omega)={\big(}dist_{\Sigma}(p, x), \Pi(x) {\big)}  
       \label{Pre, param of asymp conic at 0}      
      \end{align}
       and call it the \textbf{conic coordinates} near $p$, where $\Pi$ maps each $x$ to the closest point on $S$ from $x/dist_{\Sigma}(p, x)$. The proof of the following differential geometric lemma is left to readers.
      \begin{Lem} \label{Lem_Pre, geom of asymp at 0}
       There exist an increasing function $\tau=\tau_p: (0, 1]\to (0, +\infty)$ sufficiently small, depending on $\Sigma, M, g, p$ such that, 
       \begin{enumerate}
        \item[(1)] the map in (\ref{Pre, param of asymp conic at 0}) is a diffeomorphism from $\Sigma\cap B^M_{\tau(1)}(p)$ onto its image;
        \item[(2)] There's a smooth family of metrics $\{h_r\}_{r>0}$ on $S$ such that the Riemannian metric on $\Sigma\cap B^M_{\tau(1)}$ is $g = dr^2 + r^2h_r$, and $\forall \epsilon\in (0,1]$, \[
         \|r(h_r - g_S)\|_{C^3_*((0,\tau(\epsilon))\times S)} \leq \epsilon    \]
        where for a tensor $\phi$ on $(0, \tau)\times S$, \[
        \|\phi\|_{C^k_*((0,\tau)\times S)} := \sum_{0\leq i+j\leq k} r^{i+j-1}\|\partial^i_r \nabla_S^j \phi\|_{C^0}     \]
        \item[(3)] The second fundamental form of $\Sigma$ also satisfies the estimate \[
          \|r\cdot A_{\Sigma}{\Big|}_{(r, \cdot)} - \Pi^*A_{S} \|_{C^0((0,\tau(\epsilon))\times S)} \leq \epsilon   \]
        \item[(4)] The Jacobi operator of $\Sigma$ has decomposition \[
         L_{\Sigma} = \partial_r^2 +\frac{H_r}{r} \partial_r +\frac{1}{r^2}\cL_r    \]
         where $H_r$ is a smooth family of functions on $S$; $\cL_r = \partial_j(a_r^{ij}\partial_i) + V_r$ is a smooth family of self-adjoint elliptic operators on $S$ under normal coordinate of $S$, with \[
          \|H_r - (n-1)\|_{C^0((0,\tau(\epsilon))\times S)} + \|a_r^{ij}-\delta^{ij}\|_{C^0((0,\tau(\epsilon))\times S)} + \|V_r -|A_S|^2\|_{C^0((0,\tau(\epsilon))\times S)} \leq \epsilon    \].
       \end{enumerate}
      \end{Lem}
      If $\Sigma$ only has strongly isolated singularities, let
      \begin{align}
       \tau_{\Sigma}(\epsilon):=min \{\tau_p(\epsilon): p \in Sing(\Sigma) \}\cup \{dist_M(p, q)/2: p\neq q \in Sing(\Sigma)\}, \ \ \ \epsilon\in (0,1]  \label{Pre, conic rad of sing}
      \end{align} 
      
      We are interested in the behavior of nearby hypersurfaces of $\Sigma$ given by graphs over $\Sigma$. More precisely, let $(\Sigma, \nu)$ be a locally stable minimal hypersurface as above with only strongly isolated singularities; $\tau\in (0, \tau_{\Sigma}(1)/2)$ will be chosen differently from section to section, $\rho(x)$ will be of form $\rho(x):= \min\{dist_{\Sigma}(Sing(\Sigma), x ), 2\tau\}$. For $\Omega\subset \Sigma$ and $u\in C^k(\Omega)$, we shall frequently use the following norm, 
      \begin{align}
       \|u\|^*_{k; \Omega}:= \|u\|_{C^k_*(\Omega)}:= \sum_{j=0}^k \sup_{\Omega}\rho^{j-1}\cdot|\nabla^j u|  \label{Pre_C^k_* norm}
      \end{align}
      \begin{Lem} \label{Lem_Pre, Vol element and mean curv of graph I}
      $\exists\ \delta_1 = \delta(\Sigma, M, g,\tau)\in (0,1)$ such that if $u\in C^1(\Sigma)$ with \[
       \rho^{-1}\cdot|u|+|\nabla u|  \leq \delta_1\ \ \ \forall x\in\Sigma     \]
      Then $graph_{\Sigma}(u)$ is an embedded $C^1$ hypersurface.\\
      
      Moreover, if $g'=g+\beta$ is another $C^3$ Riemannian metric on $M$, let $F^{g'}(x, z, p)$ be the area integrand for graphs over $\Sigma$, i.e. \[
       Area_{g'}(graph_{\Sigma, g}(\phi)) = \int_{\Sigma} F^{g'}(x, \phi(x), \nabla_{\Sigma}\phi(x))\ dvol_{g|_{\Sigma}}(x)\ \ \ \ \forall \phi\in C^1(\Sigma) \text{ with }\|\phi\|^*_{1; \Sigma}\leq \delta_1     \]
      and let $\scM^{g'}$ be the Euler-Lagrangian operator, i.e. \[
       \scM^{g'}u := -div_{\Sigma}\big(\partial_p F^{g'}(x, u, \nabla u)\big) + \partial_z F^{g'}(x, u, \nabla u)    \]
      where $\partial_p F^{g'}$ could be written as 
      \begin{align*}
       \partial_p F^{g'}(x, z, p) & =: (a^{ij}(x, z, p)p_i + b^j(x, z, p)\big)\partial_j \\
       \partial_z F^{g'}(x, z, p) & =: -\big(|A_{\Sigma}|^2 + Ric_M(\nu, \nu)\big)z + \frac{1}{2}tr_{\Sigma}\nabla^M_{\nu}\beta + c(x, z, p) 
      \end{align*}
      under normal coordinates of $\Sigma$, where $a^{ij}, b^j, c\in C^1(\Sigma\times \RR\times T_*\Sigma)$. Here the derivatives and trace and all taken with respect to $g$.
      Then $\exists\ \delta_2=\delta_2(\Sigma, M, g, \tau)\in (0,\delta_1)$, $C = C(\Sigma, M, g, \tau)>1$ such that, 
      \begin{align*}
       |a^{ij}(x, z, p)-\delta^{ij}| &\leq C(\|\beta\|_{C^1(M)} + |p| + |z|/\rho(x)) \\
       \big|b^j(x, z, p)\partial_j - \iota_{\nu}\beta|_x\big| &\leq C(\|\beta\|_{C^1(M)} + |p| + |z|/\rho(x))^2 \\
       |c(x, z, p)| &\leq C(\|\beta\|_{C^2(M)} + |p| + |z|/\rho(x))^2/\rho(x)
      \end{align*}
      for $\forall (x, z, p)\in \{x\in \Sigma, |z|\leq \delta_2\rho(x), |p|\leq \delta_2\}\subset \Sigma\times \RR \times T_*\Sigma$ and $\forall \|\beta\|_{C^2(M)}\leq \delta_2 $.
     \end{Lem}
     When $g' = g$, we also have estimate for higher derivatives,

     \begin{Lem} \label{Lem_Pre, Vol element and mean curv of graph II}
      $\exists\ \delta_2'= \delta_2'(\Sigma, M, g, \tau)\in (0,\delta_1)$, $C'=C'(\Sigma, M, g, \tau)>1$ such that, if denote for simplicity $F:= F^g$, $\scM := \scM^{g}$ and $V_{\Sigma}:= (|A_{\Sigma}|^2 + Ric_M(\nu, \nu))$ in lemma \ref{Lem_Pre, Vol element and mean curv of graph I} , then we have under normal coordinate of $\Sigma$,
      \begin{align*}
      |\partial_p F(x,z, p)-p|+ \rho|\partial_z F(x,z,p) + V_{\Sigma}z| + \rho|\partial^2_{px}F(x,z,p)| &\leq C'(|z|/\rho + |p|)^2 \\
      \Big[\rho|\partial_x F|+ |\partial^2_{pp}F-\delta_{ij}| + \rho|\partial^2_{pz}F|+\rho^2 (|\partial^2_{zz}F + V_{\Sigma}|+|\partial^2_{zx}F|) &  \\
      +\ |\partial^3_{ppp}F| +\rho |\partial^3_{ppx} F|+ \rho^2(|\partial^3_{pzz}F| + |\partial^3_{pzx} F|) \Big] & \leq C'(|z|/\rho + |p|) \\
      \rho|\partial^3_{ppz}F| + \rho^3|\partial^3_{zzz}F| & \leq C'
      \end{align*}
      for $\forall (x, z, p)\in \{x\in \Sigma, |z|\leq \delta'_2\rho(x), |p|\leq \delta'_2\}\subset \Sigma\times \RR \times T_*\Sigma$.

      In particular, if denote $\scM u =: -L_{\Sigma}u + N(x,u,\nabla u)\cdot \nabla^2 u + P(x, u, \nabla u)$, then   
      \begin{align*}
       \rho(x)|P(x, z, p)| & \leq C'(|p|+|z|/\rho(x))^2  \\
       |N|+ |\partial_p N| + \rho(x)|\partial_p P| + \rho(x)^2|\partial_z P| & \leq C'(|p|+|z|/\rho(x)) \\
       \rho(x)|\partial_z N|  & \leq C'
      \end{align*}
      for $\forall (x, z, p)\in \{x\in \Sigma, |z|\leq \delta'_2\rho(x), |p|\leq \delta'_2\}\subset \Sigma\times \RR \times T_*\Sigma$.
     \end{Lem}
     Here, as a function defined on $T\Sigma \times \RR$, $\partial_x F$ denote that derivatives in horizontal directions and $\partial_p F$ be the derivatives in vertical directions.
     
     The calculation of both lemma are straightforward, and are left to readers.

%%%%%%%%%%%%%%%%%%%%%%%%%%%%%%%%%%%%%%%%%%%%%%%%%%%%%%%%%%%%%%%%%%%%%%%%%%%%
%%%%%%%%%%%%%%%%%%%%%%%%%%%%%%%%%%%%%%%%%%%%%%%%%%%%%%%%%%%%%%%%%%%%%%%%%%%%
%%%%%%%%%%%%%%%%%%%%%%%%%%%%%%%%%%%%%%%%%%%%%%%%%%%%%%%%%%%%%%%%%%%%%%%%%%%%
%%%%%%%%%%%%%%%%%%%%%%%%%%%%%%%%%%%%%%%%%%%%%%%%%%%%%%%%%%%%%%%%%%%%%%%%%%%%
%%%%%%%%%%%%%% Subsection: Linear analysis of Jacobi operator %%%%%%%%%%%%%%
%%%%%%%%%%%%%%%%%%%%%%%%%%%%%%%%%%%%%%%%%%%%%%%%%%%%%%%%%%%%%%%%%%%%%%%%%%%%
%%%%%%%%%%%%%%%%%%%%%%%%%%%%%%%%%%%%%%%%%%%%%%%%%%%%%%%%%%%%%%%%%%%%%%%%%%%%
%%%%%%%%%%%%%%%%%%%%%%%%%%%%%%%%%%%%%%%%%%%%%%%%%%%%%%%%%%%%%%%%%%%%%%%%%%%%
%%%%%%%%%%%%%%%%%%%%%%%%%%%%%%%%%%%%%%%%%%%%%%%%%%%%%%%%%%%%%%%%%%%%%%%%%%%%

     \subsection{Linear analysis of Jacobi operator} \label{Subsec, Linear ana of Jac oper}
     
      Let $(\Sigma, \nu)\subset (M, g)$ be a minimal hypersurface, $U\subset \subset M$ be a smooth domain whose boundary transversely intersect with $\Sigma$. Let $\cU:= U\cap \Sigma$ and $\bar{\cU}:=Clos(U)\cap \Sigma$. Suppose that $\Sigma$ is stable in $U$. The following result is a direct corollary of \cite[(6.18)]{CheegerNaber13_Stra_Min_Surf} by a cut-off argument,
      \begin{Lem} \label{Lem_Pre, Schoen-Simon trick of approx W^1,2}
       Suppose $\phi\in W^{1,2}_{loc}(\Sigma)$ such that $\phi|_{\Sigma\cap \partial U} = 0$ and that \[
        \int_{\cU} \phi^2 + |\nabla \phi|^2 < +\infty     \]
       Then $\exists\ \{\phi_j\}_{j\geq 1} \subset C^{\infty}_c(\cU)$ such that \[
        \int_{\cU} |\nabla(\phi - \phi_j)|^2 \to 0     \]
      \end{Lem} 
      From this lemma, we see that \[
       W_0^{1,2}(\cU):= \overline{C_c^\infty(\cU)}^{W^{1,2}} = \{\phi\in W^{1,2}_{loc}(\bar{\cU}): \|\phi\|^2_{W^{1,2}}:= \int_{\cU} \phi^2 + |\nabla \phi|^2 < +\infty, \ \phi|_{\partial U\cap \Sigma} = 0\}   \]
      
      In \cite[Proposition 3.1]{SchoenYau17_PSC}, a uniform $L^2$ nonconcentration property is proved for minimal slicings, which is used to establish the compactness for minimal slicings. The proof also works for stable minima hypersurfaces, which gives
      \begin{Lem}  \label{Lem_Pre, Uniform L^2-noncon}
       Suppose $\cS \subset M$ be a subset with $\scH^{n-1}(\cS) = 0$. Then for every $\epsilon > 0$, there exists a neighborhood $V_{\epsilon} \supset \cS$ in $M$, depending only on $M, g, U, \cS, \epsilon$, such that for every stable minimal hypersurfaces $\Sigma$ transversally intersect with $\partial U$, we have \[
        \int_{\Sigma \cap V_{\epsilon}} \phi^2 \leq \epsilon \int_{\Sigma} |\nabla \phi|^2 \ \ \ \ \forall\ \phi\in C_c^{\infty}(\Sigma\cap U)     \]
      \end{Lem}
      With this lemma, $W_0^{1,2}(\Sigma \cap U) \hookrightarrow L^2(\Sigma\cap U)$ is compact, and the first eigenvalue and eigenfunction are well-defined for Schr\"odinger operators $L = \Delta_{\Sigma} + V$ satisfying the coercivity bound for some $\delta> 0$:
      \begin{align} \label{Pre, Coercivity of Schrodinger oper}
        \int_{\Sigma}|\nabla \phi|^2 \leq \delta^{-1}(\int_{\Sigma} - \phi \cdot L\phi + \int_{\Sigma} \phi^2 )\ \ \ \ \forall\ \phi\in C_c^2(\Sigma \cap U)    
      \end{align}
       In particular, for each $s> 0$, the perturbed Jacobi operator $L_{\Sigma}^s := L_{\Sigma} - s|A_{\Sigma}|^2$ satisfies (\ref{Pre, Coercivity of Schrodinger oper}), since 
       \begin{align}
       \begin{split}
        Q_{\Sigma}^s (\phi, \phi) 
         & := \int_{\Sigma} - \phi \cdot L_{\Sigma}^s \phi = \int_{\Sigma} |\nabla \phi|^2 - (1-s)|A_{\Sigma}|^2 \phi^2 - Ric_M(\nu, \nu)\phi^2 \\
         & \geq (1-s)Q_{\Sigma}(\phi, \phi) + s\int_{\Sigma} |\nabla \phi|^2 - s\|Ric_M\|\int_{\Sigma} \phi^2  \\
         & \geq s\int_{\Sigma} |\nabla \phi|^2 - s\|Ric_M\|\int_{\Sigma} \phi^2
       \end{split} \label{Pre, Q^s_Sigma coercive}
       \end{align}
       where $\|Ric_M\|$ is the $L^{\infty}$ norm of $Ric_M$ on $M$.  
       However, it's not known whether (\ref{Pre, Coercivity of Schrodinger oper}) is true for Jacobi operators of stable minimal hypersurfaces in general.
       
      It is also mentioned in \cite{SchoenYau17_PSC} that the positive weighted functions of minimal slicings, which is the first eigenfunctions of the certain perturbed quadratic form associated to the minimal slicings, will tend to infinity near singular set. We give an alternative proof of an analogue of this fact for $L^s_{\Sigma}$ in a special case, 
      \begin{Lem} \label{Lem_Pre, Diverge of pos Jac field}
       Suppose $(M, g) = (\RR^{n+1}, |dx|^2)$ be the Euclidean space, $\Sigma$ be a stable minimal hypersurface in a domain $U \subset \RR^{n+1}$ with $Sing(\Sigma) \subset U$ being a compact subset; $s\in [0, \frac{2}{n}]$ is fixed. Let $0\neq v\in C^{\infty}_{loc}$ be a non-negative supersolution of $L_{\Sigma}^s v = 0$ on $U\cap \Sigma$, i.e. \[
         Q^s_{\Sigma}(v, \phi)\geq 0\ \ \ \forall \phi\in C_c^1(\Sigma\cap U, \RR_+)    \]
       Then \[
        v(x)\to +\infty \ \ \ \text{ as } x\to Sing(\Sigma)      \]
      \end{Lem}
      \begin{proof}
       The idea is to find an auxiliary function bounded from above by $v$ and tending to infinity near $Sing(\Sigma)$.

       First set $\cU'\subset \subset Clos(\Sigma)\cap U$ be a smooth domain containing $Sing(\Sigma)$. We can assume WLOG that $0<v\in W^{1,2}(\cU')$ is a solution of $L^{2/n}_{\Sigma} v = 0$. In fact, consider the minimizer $\bar{v}$ of $Q^{2/n}_{\Sigma}(\phi, \phi)$ among \[
         \{\phi \in W^{1,2}(\cU'): \phi = v \text{ on }\partial \cU',\ \phi\leq v \text{ on }\cU'\}   \]
       By weak and strong maximum principle, $\bar{v} \in C_{loc}^{\infty}(\cU'; \RR_+)$ is a solution of $L^{2/n}_{\Sigma} \bar{v} = 0$ with $0< \bar{v}\leq v$. Hence it suffices to prove the lemma for $\bar{v}$. From now on denote $s := 2/n$.
      
       Recall that by \cite{CheegerNaber13_Stra_Min_Surf}, if $\Sigma$ is area minimizing, then for every $\epsilon\in (0, 1)$, $|A_{\Sigma}|\in L^{7-\epsilon}$ and $|\nabla A_{\Sigma}|\in L^{7/2-\epsilon}$; This is also true for stable minimal hypersurface using the same proof. 
       Hence, $ B(x):= (1+ |A_{\Sigma}|^2)^{(1-s)/2} \in W^{1,2}(\cU') $.

       Also, recall the following identity proposed by \cite{Simons68}, \[
        \frac{1}{2}\Delta |A_{\Sigma}|^2 = |\nabla A_{\Sigma}|^2 - |A_{\Sigma}|^4     \]
       and the refined Kato's inequality for minimal hypersurfaces proposed in \cite{SchoenSimonYau75}, \[
        |\nabla |A_{\Sigma}||^2 \leq \frac{n}{n+2}|\nabla A_{\Sigma}|^2     \]
       Denote for simplicity, $A = A_{\Sigma}$, we have
       \begin{align} 
        \begin{split}
        L^s_{\Sigma} B = &\ \Delta (1+ |A|^2)^{\frac{1-s}{2}} + (1-s)|A|^2(1+ |A|^2)^{\frac{1-s}{2}} \\
         = &\ (1-s)(1 + |A|^2)^{\frac{1-s}{2}-2} {\Big[} (1+|A|^2)(|\nabla A|^2 - |A|^4) + (-s-1)|A|^2 |\nabla|A||^2 {\Big]} \\
          &  + (1-s)|A|^2(1+ |A|^2)^{\frac{1-s}{2}} \\
         = &\ (1-s)(1 + |A|^2)^{\frac{1-s}{2}-2} {\Big[} |\nabla A|^2 + |A|^2 {\big(} |\nabla A|^2 + (-s-1)|\nabla |A||^2 {\big)} + |A|^2 (1+ |A|^2)  {\Big]} \\
         \geq &\ 0
        \end{split}  \label{Pre, |A|^s be a subsol of L^s}
       \end{align}
       In other words, $B$ is a subsolution of $L^s_{\Sigma}$. 

       Now fixed $\epsilon >0 $ be sufficiently small such that $\epsilon B < v$ near $\partial \cU'$. Then $\phi:=(\epsilon B - v)^+ \in W^{1,2}_0(\cU')$ and  \[
        Q_{\Sigma}^s(\phi, \phi) = \int_{\Sigma} \nabla \phi \cdot \nabla(\epsilon B - v) - (1-s)|A_{\Sigma}|^2 \phi (\epsilon B - v) = \int_{\Sigma} -\phi \cdot L^s_{\Sigma} (\epsilon B - v) \leq 0   \]
       But since $\Sigma$ is stable in $U$, $Q^s_{\Sigma}$ is strictly positive on $W^{1,2}_0 (\cU')$, hence $(\epsilon B -v)^+ = 0$, i.e. $v\geq \epsilon B$. 
       
       Now we can complete the proof of lemma \ref{Lem_Pre, Diverge of pos Jac field}. For any sequence of points $\Sigma \ni x_j\to p \in Sing(\Sigma)$, let $r_j := dist_{\RR^{n+1}} (x_j, p)$ and write for simplicity $B_r(x) := \BB^{n+1}_r(x)\cap \Sigma$. Since $\Delta v \leq L^s_{\Sigma} v = 0 $, by applying the Harnack inequality in \cite{BombieriGiusti72_Harnack} to $v$ and combining the fact that $v \geq \epsilon B \geq \epsilon |A_{\Sigma}|^{1-s}$, we have \[
         v(x_j) \geq \frac{C(\Sigma, n)}{(2r_j)^n}\int_{B_{2r_j}(x_j)} v \geq \frac{C(\Sigma, n)\epsilon}{r_j^n} \int_{B_{r_j}(p)} |A_{\Sigma}|^{1-s} \to +\infty  \]
        when $j\to \infty$. This finish the proof. 
      \end{proof}
      
      This lemma is used in section \ref{Sec, Asymp & Asso Jac field} to give a new proof of Hardt-Simon typed regularity.

    \section{Linear analysis for strongly isolated singularities} \label{Sec, Linear Ana}
     Throughout this section, let $\Sigma\subset (M, g)$ be a locally stable minimal hypersurface in a Riemannian manifold with normal field $\nu$; $U\subset \subset M$ will always be a smooth domain with $\partial U$ transversely intersect with $Reg(\Sigma)$; $\cU:= U\cap \Sigma$ and $\bar{\cU} := Clos(U)\cap \Sigma$.
     Let $L_{\Sigma}$, $Q_{\Sigma}$ be the Jacobi operator and its associated quadratic form  as in (\ref{Pre, Def Jac oper}), (\ref{Pre, quadratic assoc to Jac oper}). For our applications below, we may also deal with the weighted quadratic forms \[
       Q^h_{\Sigma}(\phi, \phi) := Q_{\Sigma}(\phi, \phi) + \int_{\Sigma} h\phi^2  \]
     for a given $h \in L^{\infty}(\Sigma)$, and its associated operator $L_{\Sigma} - h$.

%%%%%%%%%%%%%%%%%%%%%%%%%%%%%%%%%%%%%%%%%%%%%%%%%%%%%%%%%%%%%%%%%%%%%%%%%%%%
%%%%%%%%%%%%%%%%%%%%%%%%%%%%%%%%%%%%%%%%%%%%%%%%%%%%%%%%%%%%%%%%%%%%%%%%%%%%
%%%%%%%%%%%%%%%%%%%%%%%%%%%%%%%%%%%%%%%%%%%%%%%%%%%%%%%%%%%%%%%%%%%%%%%%%%%%
%%%%%%%%%%%%%%%%%%%%%%%%%%%%%%%%%%%%%%%%%%%%%%%%%%%%%%%%%%%%%%%%%%%%%%%%%%%%
%%%%%%%%%%%%%%%%%%% Subsection: Linear analysis in large %%%%%%%%%%%%%%%%%%%
%%%%%%%%%%%%%%%%%%%%%%%%%%%%%%%%%%%%%%%%%%%%%%%%%%%%%%%%%%%%%%%%%%%%%%%%%%%%
%%%%%%%%%%%%%%%%%%%%%%%%%%%%%%%%%%%%%%%%%%%%%%%%%%%%%%%%%%%%%%%%%%%%%%%%%%%%
%%%%%%%%%%%%%%%%%%%%%%%%%%%%%%%%%%%%%%%%%%%%%%%%%%%%%%%%%%%%%%%%%%%%%%%%%%%%
%%%%%%%%%%%%%%%%%%%%%%%%%%%%%%%%%%%%%%%%%%%%%%%%%%%%%%%%%%%%%%%%%%%%%%%%%%%%
     
     \subsection{Linear analysis in large} \label{Subsec, General linear Analysis}
      We start with a basic observation 
      \begin{Lem} \label{Linear, essential positivity of Q} 
       There's a constant $C_0 = C_0(\Sigma, M, g, U) > 0$ such that \[
        Q_{\Sigma}(\phi, \phi) + C_0\|\phi\|^2_{L^2(\Sigma)} \geq 0 \ \ \ \ \forall \phi\in C^1_c(\cU) \] 
      \end{Lem}
      \begin{proof}
       Let $\{U_j\subset M\}$ be a finite open cover of $Clos(\cU)$ such that $\Sigma$ is stable in each $U_j$; Let $\{\eta_j^2\}$ be a partition of unity on $\{U_j\}$ such that $\eta_j$ is also smooth. Then $\forall \phi \in C_c^1(U\cap\Sigma)$, since $\eta_j\cdot\phi$ is support in $U_j$,
       \begin{align}
       \begin{split}
        Q_{\Sigma}(\phi,\phi) & = \int_{\Sigma} \sum_j |\nabla \phi|^2 \eta_j^2 - \sum_j(|A_{\Sigma}|^2+Ric_M(\nu, \nu))\phi^2\eta_j^2 \\
        & = \sum_j Q_{\Sigma}(\phi\eta_j, \phi\eta_j) - \sum_j \int_{\Sigma} {\big(} \nabla (\phi^2)\cdot \eta_j\nabla \eta_j + \phi^2 |\nabla\eta_j|^2 {\big)} \\
        & = \sum_j Q_{\Sigma}(\phi\eta_j, \phi\eta_j) + \sum_j \int_{\Sigma}  \phi^2 ( div(\eta_j\nabla \eta_j) - |\nabla\eta_j|^2 ) \\
        & \geq \int_{\Sigma} (\sum_j \eta_j\cdot \Delta \eta_j) \phi^2 
       \end{split} \label{Linear, Bd Q from below by L^2}
       \end{align}
       Since $\Sigma$ is minimal, \[
        |\Delta \eta_j| = |\Delta_M \eta_j - \nabla_M^2 \eta_j (\nu, \nu)| \leq n |\nabla^2_M \eta_j|    \]
       Hence, $\sum_j \eta_j\cdot \Delta \eta_j$ is bounded from below by some constant $-C_0 = -C_0 (M, g, \Sigma, U)$. This proves the lemma.
      \end{proof}

      Such $C_0$ in the Lemma will be fixed throughout this section. Consider the Hilbert space $\scB(\cU) := \overline{C_c^1(\bar{\cU})}^{\|\cdot\|_{\scB}}$ and $\scB_0(\cU):= \overline{C_c^1(\cU)}^{\|\cdot\|_{\scB}}$ where 
      \begin{align}
       \|\phi\|^2_{\scB(\cU)} := Q_{\Sigma}(\phi, \phi) + (C_0+1)\|\phi\|^2_{L^2(\cU)}   \label{Linear, Def scB-norm}      
      \end{align}
      Note that by definition and lemma \ref{Linear, essential positivity of Q}, \[
       -C_0\|\phi\|^2_{\scB(\cU)} \leq Q_{\Sigma}(\phi, \phi) \leq \|\phi\|^2_{\scB(\cU)}   \ \ \ \forall \phi\in C_c^1(\bar{\cU}) \]
      Hence, the quadratic form $Q_{\Sigma}$ extends to a bilinear form on $\scB(\cU)$.

      $\scB(\cU)$ (resp. $\scB_0(\cU)$) will play the role as $W^{1,2}(\cU)$ (resp. $W_0^{1,2}(\cU)$) for inverting $L_{\Sigma}$ or finding eigenfunctions. In fact, $\scB_0(\cU) = W_0^{1,2}(\cU)$ provided $L = L_{\Sigma}$ satisfies the coercivity condition (\ref{Pre, Coercivity of Schrodinger oper}) for some $\delta >0$. In particular, this is true when every singularity is strongly isolated and strictly stable. However, when $\Sigma$ has non-strictly stable tangent cones, this will always fail, and $W_0^{1,2}\subset \scB_0$. 
      
      On the other hand, lemma \ref{Linear, essential positivity of Q} asserts that $\|\phi\|_{\scB} \geq \|\phi\|_{L^2}$, hence there's a natural continuous embedding $\scB(\cU) \hookrightarrow L^2(\cU)$. It's convenient to have the following characterization for $\phi \in L^2(\cU)$ to be in $\scB(\cU)$.
      
      \begin{Lem} \label{Linear, Equi def of scB}
       Suppose $\phi \in L^2(\cU)$. Then the following are equivalent, 
       \begin{enumerate}
        \item[(1)] $\phi \in \scB(\cU)$ (resp. $\scB_0(\cU)$);
        \item[(2)] $\phi \in W^{1,2}_{loc}(\cU)$; And for any decreasing family $... V_j \supset \supset V_{j+1} ...$ of neighborhood of $Sing(\Sigma)$ with $\bigcap_j V_j = Sing(\Sigma)$, there are $\phi_j \in W^{1,2}(\cU)$ (resp. $W_0^{1,2}(\cU)$) such that $\spt(\phi_j - \phi) \subset V_j$, $\spt (\phi_j) \cap V_{N_j} = \emptyset $ for some $N_j>>j$, and $ \sup_j \|\phi_j\|_{\scB} < +\infty $.
        \item[(3)] $\exists \phi_j\in \scB(\cU)\cap L^2$ (resp. $\scB_0(\cU)\cap L^2$) such that $\phi_j \to \phi$ a.e. and $\limsup \|\phi_j\|_{\scB} < +\infty $.
       \end{enumerate}
      \end{Lem}
      \begin{proof}
       Clearly (2) implies (3);
       To see (3) implies (1), first observe that by weak compactness of Hilbert space and Banach-Saks theorem, if $\phi_j$ is uniformly bounded in $\scB$, then there's a subsequence $\phi_{k_j}$ such that $(\sum_{j=1}^m \phi_{k_j})/m \to \phi'$ in $\scB(\cU)$ (resp. in $\scB_0(\cU)$). Since $\scB \hookrightarrow L^2$ and that $\phi_j \to \phi$ a.e., we have $\phi = \phi'\in \scB(\cU)$ (resp. $\phi \in \scB_0(\cU)$).  
       
       Now we shall prove (1) implies (2) beginning with $\phi\in \scB_0(\Sigma)$; the proof when $\phi\in \scB(\Sigma)$ is similar. \\
       \textbf{Claim}: If $\Omega\subset \subset \cU$, then for every $\varphi\in C_c^1(\cU) $, \[
         \|\varphi\|_{W^{1,2}(\Omega)} \leq C(\Omega, \cU, \Sigma, M, g)\|\varphi\|_{\scB(\cU)}    \]
       \textbf{Proof of the claim}: Take a finite open cover $\{W_l\}_{l\geq 0}$ of $Clos(\cU)$ in $M$ such that $Clos(\Omega)\subset W_0 \subset \subset M\setminus Sing(\Sigma)$, and that $W_l\cap \Omega = \emptyset$ and $\Sigma$ is stable in $W_l$ for $l\geq 1$. Let $\{\eta_l^2\}_{l\geq 0}$ be a partition of unity for $\{W_l\}$ such that $\eta_l$ are also smooth. By the same computation in (\ref{Linear, Bd Q from below by L^2}) we have,
       \begin{align*}
        Q_{\Sigma}(\varphi, \varphi) & = Q_{\Sigma}(\eta_0 \varphi, \eta_0\varphi) + \sum_{l\geq 1}Q_{\Sigma}(\eta_l \varphi, \eta_l\varphi) + \sum_{l\geq 0}\int \varphi^2(\eta_l\Delta_{\Sigma}\eta_l) \\
        & \geq Q_{\Sigma}(\eta_0 \varphi, \eta_0\varphi) - C(\{\eta_l\})\|\varphi\|_{L^2}^2  \geq \int_{\Omega}|\nabla \varphi|^2  - C(\Omega, \{\eta_l\})\|\varphi\|_{L^2}^2
       \end{align*}
       Together with the fact that $\|\varphi\|_{L^2}\leq \|\varphi\|_{\scB}$, this completes the proof of claim.\\
       
       Now to prove (1) implies (2), suppose $\varphi_k \in C_c^1(\cU)$ be a Cauchy sequence in $\scB_0(\cU)$ which converges to $\phi$. Then by the claim, $\phi\in W^{1,2}_{loc}(\cU)$.
       
       For every nested neighborhood $\{V_j\}_{j\geq 1}$ in (2), let $\xi_j\in C_c^\infty (V_j, [0,1])$ be a cut-off which equals to $1$ on $V_{j+1}$. For each fixed $j$, let $\psi_k:=\psi^{(j)}_k := (1-\xi_j)\phi + \xi_j\varphi_k$. Then by the similar computation as in (\ref{Linear, Bd Q from below by L^2}) we have,
       \begin{align*}
        & Q_{\Sigma}(\psi_k-\psi_{k'}, \psi_k-\psi_{k'}) \\
       = & \int_{\cU} - \xi_j\Delta \xi_j(\varphi_k - \varphi_{k'})^2 + \xi_j^2 \big(|\nabla (\varphi_k - \varphi_{k'})|^2 - (|A_{\Sigma}|^2+Ric_M(\nu,\nu))(\varphi_k - \varphi_{k'})^2 \big)  \\ 
       \leq &\  C(\xi_j)\|\varphi_k - \varphi_{k'}\|_{L^2}^2 + Q_{\Sigma}(\varphi_k - \varphi_{k'}, \varphi_k - \varphi_{k'})  \\
       & + \int_{\cU\setminus V_{j+1}} |\nabla (\varphi_k - \varphi_{k'})|^2 + \big||A_{\Sigma}|^2+Ric_M(\nu,\nu)\big|(\varphi_k - \varphi_{k'})^2 \\
       \leq &\  C(\xi_j, V_j, V_{j+1}) \|\varphi_k - \varphi_{k'}\|^2_{\scB}
       \end{align*}
       where the last inequality follows form the claim. Hence $\{\psi_k\}_{k\geq 1}$ is also a Cauchy sequence in $\scB_0(\cU)$, which definitely converges to $\phi$. Thus we can choose $k_j>>j$ such that $\|\psi^{(j)}_{k_j} - \phi\|_{\scB}\leq 2^{-j} $, and $N_j>>j$ such that $spt(\psi^{j}_{k_j})\cap V_{N_j} = \emptyset$. Thus $\phi_j := \psi^{(j)}_{k_j}$ is what we want in (2).
      \end{proof}

      \begin{Eg}
       Let $\Sigma = C \subset \RR^{n+1}$ be a regular stable minimal hypercone which is not strictly stable, i.e. $\mu_1 = -(\frac{n-2}{2})^2$; Let $U =\BB^{n+1}_1$ and $\cU = B_1\subset C$. Use the notations and parametrization as in (\ref{Pre, param cone}). Let $u (r,\omega) := r^{-(n-2)/2}w_1(\omega)$. Then \[
        \|\nabla u\|^2_{L^2(B_1(0))} = \int_0^1 r^{n-1}\ dr\int_S |\partial_r u|^2 + \frac{1}{r^2}|\nabla_S u|^2 \ d\omega = +\infty     \]
       But we can construct \[
        u_j(r, \omega ):= max\{r, \frac{1}{j}\}^{-(n-2)/2}\cdot w_1(\omega)     \]
       Clearly $u_j \in W^{1,2} (U\cap C) \subset \scB(\cU)$ and $u_j \to u$ pointwisely. We verify now that $\sup_j \|u_j\|_{\scB} < +\infty$. In fact, notice that $L_C u = 0$ on $B_1\setminus B_{1/j}$, thus
       \begin{align*}
        Q_{C}(u_j, u_j) &= \int_{B_1\setminus B_{1/j}} |\nabla u|^2 - |A_C|^2 u^2 + \int_{B_{1/j}} |\nabla u_j|^2 - |A_C|^2 u_j^2 \\
          & = \int_{\partial B_1} u \partial_r u - \int_{\partial B_{1/j}} u\partial_r u - \int_{B_1\setminus B_{1/j}} u\cdot L_C u \\ 
          &\ + \int_0^{1/j} r^{n-1}\ dr\int_S r^{-2}(|\nabla_S u_j|^2 - |A_S|^2 u_j^2)\ d\omega \\
          & = \int_{\partial B_1} (-\frac{n-2}{2})r^{-n/2 - (n-2)/2}w_1^2 - \int_{\partial B_{1/j}} (-\frac{n-2}{2})r^{-n/2 - (n-2)/2}w_1^2 - 0 \\
          & \ + \int_0^{1/j} r^{n-3}\ dr \int_S j^{n-2} \cdot( |\nabla_S w_1|^2 - |A_S|^2 w_1^2 )\ d\omega \\
          & \leq C(n)           
       \end{align*}
       And $\|u_j\|_{L^2}^2 \leq \|u\|_{L^2}^2 \leq C(n)$. Therefore, by lemma \ref{Linear, Equi def of scB}, $u \in \scB(\cU)$. \\
       
       On the other hand, let $v = -r^{\frac{n-2}{2}}\log r \cdot w_1$ on $B_1$, then $v \in W_{loc}^{1,2} (C)$ and $v$ solves the equation $L_C v = 0$ on $B_1$. However, we shall see that that $v\notin \scB(\cU)$. In fact, by Lemma \ref{Linear, Equi def of scB}, consider $\forall 0<s < t < 1$, the function $v_{s, t}$ that achieves \[
       \inf \{Q_C(\phi, \phi) : \phi \in C^1_c(C),\ \phi = v \text{ on }C\setminus B_t,\ \spt(\phi) \cap B_s = \emptyset \}      \]
      By writing down the E-L equation and solving an ODE, 
      \begin{align*}
       v_{s,t}(r, \omega) = \begin{cases} v(r, \omega),\ &\text{ if } t\leq r\leq 1 \\
        (\frac{\log t \cdot \log s}{\log t - \log s} - \frac{\log t}{\log t - \log s}\log r )r^{-\frac{n-2}{2}}w_1(\omega),\ &\text{ if } s\leq r \leq t \\
        0,\ &\text{ if } 0 < r < s
       \end{cases}
      \end{align*}
      Therefore, \[
       Q_C(v_{s,t}, v_{s,t}) = (-\log t)+ \frac{(-\log t)^2}{\log t - \log s} \geq (-\log t)      \]
      And then $\|v_{s, t}\|_{\scB} \to +\infty$ when $t\to 0$. Hence by lemma \ref{Linear, Equi def of scB}, $v \notin \scB(\cU)$.
      \end{Eg}

      To study the spetral theory of $L_{\Sigma}$, we may wish that $\scB \hookrightarrow L^2$ is also compact.
      
      \begin{Def}
       Call $\Sigma$ satisfies the \textbf{$L^2$-nonconcentration property} in $U$, if for any $\epsilon >0$, there's a sufficiently small neighborhood  $V_{\epsilon} \supset Sing(\Sigma)\cap U$ such that 
       \begin{align}
        \int_{\Sigma\cap V_{\epsilon}} \phi^2 \leq \epsilon\cdot \|\phi\|^2_{\scB}\ \ \ \ \forall \phi\in C^1_c(\cU)   \label{Linear, Def L^2 noncon}       
       \end{align}
      \end{Def}

      \begin{Prop} \label{Prop_Linear, Basic prop for L^2-noncon}
       Suppose $\Sigma$ satisfies the $L^2$-nonconcentration property in $U$; $h\in L^\infty(\cU)$. Then,
       \begin{enumerate}
        \item[(1)] $\scB_0(\cU) \hookrightarrow L^2(\cU)$ is compact.
        \item[(2)] $\exists\ \lambda_1 < \lambda_2 < \lambda_3 < ... \nearrow +\infty$ and finite dimensional pairwise $L^2$-orthogonal linear subspaces $\{E_j\}_{j\geq 1}$ of $\scB_0(\cU)\cap C^{\infty}(\cU)$ such that \[
          (-L_{\Sigma} + h) \phi = \lambda_j \phi\ \ \ \ \forall \phi\in E_j    \]
        and that \[  L^2(\cU) = \overline{\bigoplus_{j\geq 1}E_j}^{L^2};\ \ \ \ \scB_0(\cU) = \overline{\bigoplus_{j\geq 1}E_j}^{\scB}    \]
        Moreover, $\dim E_1 = 1$ and any $0\neq \phi_1\in E_1$ is nowhere vanishing in $\cU$.
        \item[(3)] If $(L_{\Sigma}- h){\big|}_{\scB_0}$ is \textbf{nondegenerate}, i.e. $\lambda_j \neq 0$ for all $j\geq 1$ in (2). Then for each $f\in L^2(\Sigma)$, \[
         (-L_{\Sigma}+h)\phi = f \ \ \ \ \text{ on }\cU\]
        has a unique solution $\phi\in \scB_0(\cU)$. Moreover, we have energy-type estimate \[
         \|\phi\|_{\scB(\cU)} \leq C(h, \Sigma, U, M, g)\|f\|_{L^2}   \]
        \item[(4)] Suppose $-L_{\Sigma}+h$ is \textbf{strictly positive}, i.e. $\lambda_1>0$ in (2). Suppose $u\in \scB(\cU)$ satisfies $(-L_{\Sigma}+ h) u \geq 0$ in the distribution sense, i.e. \[
          \int_{\cU} \nabla u\cdot \nabla \phi - (|A_{\Sigma}|^2 + Ric_M(\nu,\nu) - h)u\cdot \phi \geq 0\ \ \ \forall \phi\in C^{\infty}_c(\cU; \RR_+)     \]
         Also assume that either $\partial U\cap Sing(\Sigma)=\emptyset$ and $u|_{\partial \cU} \geq 0$, or $u\in \scB_0(\cU)$. 
         Then $u\geq 0$ on $\cU$. 
        \item[(5)] There exists $\vartheta \in C_c^{\infty}(\bar{\cU})$ such that $-L_{\Sigma}+ h + \vartheta$ is strictly positive on $\scB_0(\cU)$.
        \item[(6)] $\{f\in C_c^\infty(\cU): (-L_\Sigma + h + f)|_{\scB_0} \text{ is nondegenerate}\}$ is open and dense in $C_c^\infty(\cU)$.
%        \item[(7)] If $-L_{\Sigma}+h$ is strictly positive, i.e. $\lambda_1>0$ in (2), then $\forall p\in \bar{\Sigma}\cap U$ fixed, there's a smooth positive function (viewed as Green's function at $p$) $G_p:\Sigma\cap U\setminus \{p\}\to \RR$ such that for any $\eta\in C_c^{\infty}(M\setminus \{p\})$, $\eta\cdot G_p \in \scB$, and \[
%         (-L_{\Sigma}+h)G_p = 0\ \ \ \ \text{ on }\Sigma\cap U\setminus \{p\}   \]
%       Moreover, such $G_p$ is unique up to renormalization.
       \end{enumerate}
      \end{Prop}
%      \begin{Rem}
%       The notion of non-degeneracy is a generic property. More precisely, for $\Sigma$ in proposition \ref{Prop_Linear, Basic prop for L^2-noncon}, $\{h\in C_c^\infty(\Sigma): -L_\Sigma + h \text{ is nondegenerate}\}$ is open and dense in $C_c^\infty(\Sigma)$. Openness follows from proposition \ref{Prop_Linear, Basic prop for L^2-noncon} (2); Denseness can be seen in the following way:
%      \end{Rem}
      
      \begin{proof}
       (1) Suppose $\{\phi_k\}$ be a bounded sequence in $\scB_0(U)$, WLOG $\|\phi_k\|_{\scB} \leq 1$. Then by the proof of lemma \ref{Linear, Equi def of scB}, for every subdomain $\Omega \subset \subset \Sigma$, $\|\phi_k\|_{W^{1,2}} \leq C(\Sigma, M, g, \Omega)$. By Reilly theorem [G-T] and lemma \ref{Linear, Equi def of scB}, up to a subsequence, $\phi_k \to \phi$ for some $\phi \in \scB_0(\cU)$ in $L^2(\Omega)$, $\forall \Omega\subset \subset \Sigma$. We shall see that $\|\phi_k - \phi\|_{L^2} \to 0$, which proves the compactness. In fact, $\forall \epsilon > 0$, let $V_{\epsilon}$ be the neighborhood of $Sing(\Sigma)$ in definition (\ref{Linear, Def L^2 noncon}), then 
       \begin{align*}
        \limsup_{k\to \infty} \int_{\Sigma} (\phi_k-\phi)^2 & \leq \limsup_{k\to \infty} \int_{V_{\epsilon}} (\phi_k - \phi)^2 + \lim_{k\to \infty}\int_{V_{\epsilon}^c} (\phi_k -\phi)^2 \\    
        & \leq 4\epsilon \|\phi_k\|_{\scB}^2 + 0 \leq 4\epsilon
       \end{align*}
       Let $\epsilon \to 0$, we see that $\phi_k \to \phi$ in $L^2(\Sigma)$.\\       
       (2) and (3) follows directly from (1) and the standard method in spectral theory. We omit the proof here.
       
       To prove (4), the basic strategy is to multiply the differential inequality by $u^-$ and integration by part. However, it's unclear whether $u^-$ still lies in $\scB(\cU)$. Hence we do this by approximation. Consider for each $s\in (0, 1)$ the approximated operators $L^s:= L_{\Sigma} - h - s|A_{\Sigma}|^2$. Denote for simplicity $P:=|A_{\Sigma}|^2 + Ric_M(\nu, \nu) -h$. Let \[
        \scE^s(v):= \frac{1}{2}\int_{\cU}|\nabla v|^2 + (-P + s|A_{\Sigma}|^2)v^2\ dx - \int_{\cU} \nabla u\cdot \nabla v - Pu\cdot v\ \ \ \ \forall v\in C^1_c(\bar{\cU})     \]
       Note that since the second integration above is clearly bounded from above by $(1+\|h\|_{L^\infty}+\|Ric_M\|_{L^\infty})\cdot\|u\|_{\scB(\cU)}\cdot \|v\|_{\scB(\cU)}$, $\scE^s$ extends to be defined on $W^{1,2}(\cU)$.
       
       Let $v^{(s)}$ be the minimizer of $\scE^s$ among $\{v\in W^{1,2}(\cU): v = u \text{ on }\partial \cU\}$. Then by lemma \ref{Lem_Pre, Uniform L^2-noncon} and (\ref{Pre, Q^s_Sigma coercive}), $v^{(s)}\in W^{1,2}(\cU)$, satisfies the equation $-L^s v^{(s)} = (-L_{\Sigma}+h)u$ in the distribution sense on $\cU$ and $v^{(s)}\geq 0$ on $\partial \cU$. 
       Multiply the equation by $v^{(s)}_-:= -min\{0, v^{(s)}\} \in W^{1,2}_0(\cU)$ and take integration by part to get \[
        \int_{\cU} |\nabla v^{(s)}_-|^2 + (-P + s|A_{\Sigma}|^2)(v^{(s)}_-)^2 \leq 0  \]
       Since $-L_{\Sigma} + h$ is strictly positive on $\scB_0(\cU)$, so is $-L^s$. Thus we conclude that $v^{(s)}_- = 0$, i.e. $v^{(s)}\geq 0$.
       
       On the other hand, by the equation satisfied by $v^{(s)}$, we see $\sup \|v^{(s)}\|_{\scB(\cU)} \leq C(\Sigma, h)\|u\|_{\scB(\cU)}$ for every $s\in (0, 1)$. Hence, by (1), up to a subsequence of $s\to 0_+$, $v^{(s)}\to v^{(0)}$ in $L^2(\cU)$ for some $0\leq v^{(0)}\in \scB(\cU)$ equal to $u$ on $\partial \cU$ and satisfying the equation $(-L_{\Sigma} + h)(v^{(0)}- u) = 0$. By the nondegeneracy of $-L_{\Sigma}+h$ we see that $u = v^{(0)}\geq 0$, which completes the proof of (4).
       
       To prove (5), first observe that since $\Sigma$ satisfies the $L^2$-nonconcentration property in $U$, there exists a small neighborhood $\cV_0 \supset Sing(\Sigma)$ such that \[
        Q^h_{\Sigma}(\phi, \phi):= Q_{\Sigma}(\phi, \phi) + \int_{\Sigma} h\phi^2 \geq 0\ \ \ \forall \phi\in C_c^1(\cU\cap \cV_0)     \]
       Let $\cV_1\subset\subset M\setminus Sing(\Sigma)$ be an open subset such that $\cV_0 \cup \cV_1 \supset Clos(U)$. Let $\{\eta^2_0, \eta^2_1\}$ be the  partition of unity corresponding to $\cV_0, \cV_1$. Let $\beta(>-\infty)$ be the first eigenvalue of $-L_{\Sigma}+h$ on $\scB_0(\cV_1\cap \Sigma)$. Then by the virtual of (\ref{Linear, Bd Q from below by L^2}), for each $\phi\in C_c^{\infty}(\cU)$,  
       \begin{align}
       \begin{split}
        Q^h_{\Sigma}(\phi, \phi) 
        & = Q^h_{\Sigma}(\phi\eta_0, \phi\eta_0) + Q^h_{\Sigma}(\phi\eta_1, \phi\eta_1) + \int_{\Sigma} \phi^2(\eta_0\Delta\eta_0 + \eta_1\Delta\eta_1) \\
        & \geq \int_{\cV_1} \phi^2(\beta\eta_1^2+\eta_0\Delta\eta_0 + \eta_1\Delta\eta_1 )
       \end{split} \label{Linear_Q^h > beta L^2} 
       \end{align}
       Hence take $\vartheta\in C_c^\infty(\bar{\cU}; \RR_+)$ such that $\vartheta \geq 1+ |\beta + \eta_0\Delta\eta_0 + \eta_1\Delta\eta_1 |$ on $\cV_1$, we see from (\ref{Linear_Q^h > beta L^2}) that $-L_{\Sigma}+h+\vartheta$ is strictly positive on $\scB_0(\cU)$.
       
       To prove (6), first note that openness follows directly from (2); To see denseness, denote for simplicity $L^h:= L_\Sigma - h$ and $\cK:= Ker(-L^h)\subset \scB_0(\cU)$. By arbitrariness of $h$, it suffices to show that if $\cK \neq 0$, then there exists $f\in C_c^\infty(\cU)$ such that $-L^h + tf$ is non-degenerate on $\scB_0(\cU)$ for sufficiently small $|t|$.
       
       Choose $f$ by the following approach. Let $\{\cU_l\subset\subset \cU\}_{l\geq 1}$ be an increasing exhaustion of $\cU$, and $f_l\in C_c^\infty(\cU; [0,1])$ be a family of cut-off which restricts to $1$ on $\cU_l$. Clearly, $f_l\to 1$ on $\cU$ weakly in $L^\infty$, hence for sufficiently large $l$, the bilinear form on $\cK\times \cK$, $(\phi_1, \phi_2)\mapsto \int_{\cU} f_l\phi_1\phi_2\ dx$ is nondegenerate. Choose $f = f_l$ for such $l$.
       
       To show $-L^h + tf$ is non-degenerate for $|t|<<1$, suppose for contradiction that $\exists\ t_i\to 0$ and $u_i \in Ker(-L^h + t_if)\cap \scB_0(\cU)$ with $\|u_i\|_{L^2} = 1$. Let $v_i$ be the $L^2$-orthogonal projection of $u_i$ onto $\cK$; And let $\{E_j\}_{j\geq 1}$ be the $L^2$-orthogonal eigensubspaces of $-L^h$ given in (2), with corresponding eigenvalues $\lambda_1<\lambda_2<... \nearrow +\infty$. 
       Multiply the equation $L^h(u_i - v_i) = t_i f u_i$ with $\phi\in E_j$, integrate by parts and take sum over $j$ we see that \[
        \|u_i - v_i\|_{L^2}\leq \lambda^{-1}|t_i|\|f\|_{L^\infty}   \]
       where $\lambda := \inf\{|\lambda_j|: \lambda_j\neq 0\}>0$. Let $\hat{u}_i:= (u_i - v_i)/t_i$, since $\hat{u}_i$ satisfies the equation $L^h \hat{u}_i = fu_i$ and the $L^2$ bound above, by integration by part we have $\limsup_i\|\hat{u}_i\|_{\scB(\cU)}\leq C<+\infty$; Also, since $v_i\in\cK$ and $\|v_i\|_{L^2}\leq 1$, we have $\limsup_i\|v_i\|_{\scB} <+\infty$.

       By compactness assumption for $\scB_0(\cU)\hookrightarrow L^2$ we conclude that up to a subsequence, $\hat{u}_i\to \hat{u}_\infty$, $v_i\to v_\infty$ both in $L^2(\cU)$; And they satisfy $\|v_\infty\|_{L^2} = 1$, $L^h \hat{u}_\infty = fv_\infty$. Multiply by $\phi\in \cK$ and take integration by part, we see that \[
        \int_{\cU} f v_\infty\cdot \phi = 0\ \ \ \text{ for every }\phi\in \cK   \] 
       This contradicts to the choice of $f$.
      \end{proof}
      
      \begin{Rem}
       Proposition \ref{Prop_Linear, Basic prop for L^2-noncon} (4) is usually called \textbf{weak maximum principle} for $-L_{\Sigma} +h$. The assumption on boundary value of $u$, which is used in the proof to guarantee the solutions to perturbed equations fall in $W^{1,2}$, is technical and subtle and is conjectured to be dropped. 
       
       Proposition \ref{Prop_Linear, Basic prop for L^2-noncon} (6) shows that $-L_\Sigma + h$ being non-degenerate is a generic property.
      \end{Rem}

      We finish this subsection by discussing a basic corollary of the $L^2$-nonconcentration property. Following \cite{Dey19_Cptness}, the \textbf{index} of a singular minimal hypersurface $\Sigma$ in $U\subset M$ is defined by 
      \begin{align*}
       ind(\Sigma, U) := \sup {\Big\{ }\dim \scV  :&\ \scV \subset \scX_c(U) \text{\ be a linear subspace s.t. }\\
       &\ \ \frac{d^2}{ds^2}{\Big|}_{s=0}\scH^n(e^{sX}(\Sigma)) <0\ \ \forall X\in \scV \ {\Big \}}      
      \end{align*}
      \begin{Cor} \label{Linear, Finiteness of index assum L^2 noncon}
       Suppose $\Sigma$ satisfies the $L^2$-nonconcentration property in $U$. Let $\lambda_1 < \lambda_2 < \lambda_3 < ... \nearrow +\infty$ be the spectrum of $L_{\Sigma}$ on $\scB$ as in Proposition \ref{Prop_Linear, Basic prop for L^2-noncon}, where $h = 0$; and $E_j$ be the eigenspaces. Then \[
        ind(\Sigma, U) = \sum_{\lambda_j < 0}\dim E_j   \]
       In particular, $\Sigma$ has finite index.
      \end{Cor}
      Combined with lemma \ref{Linear, L^2 noncon for iso} in the next subsection, this partially answers a question proposed by A. Neves.
      \begin{Cor} \label{Linear, Finiteness of index in dim 8}
       Let $\Sigma$ be a closed, locally stable minimal hypersurface in an $8$ dimensional manifold, then $\Sigma$ has finite index.
      \end{Cor}
      
      \begin{proof}[Proof of corollary \ref{Linear, Finiteness of index assum L^2 noncon}]
       First recall that by lemma \ref{Lem_Pre, Schoen-Simon trick of approx W^1,2}, if $X\in \scX_c(U)$, then $\phi:= X\cdot \nu \in W_0^{1,2}(\cU)$, and 
       \begin{align}
        \frac{d^2}{ds^2}{\Big|}_{s=0}\scH^n(e^{sX}(\Sigma)) < 0 \ \ \text{ iff }\ \   Q_{\Sigma} (\phi, \phi) < 0    \label{Linear, Equi notion of area decrease quadratically}
       \end{align}
       Hence, define 
       \begin{align*}
        ind_c(\Sigma, U) & := \sup \{\dim \scV_c : \scV_c \subset C_c^{\infty}(\Sigma\cap U) \text{ and } Q_{\Sigma}(\phi, \phi)< 0 \ \ \forall 0\neq\phi\in \scV_c \} \\
        ind_{\scB}(\Sigma, U) & := \sup \{\dim \scV_{\scB} : \scV_{\scB} \subset \scB_0(\cU) \text{ and } Q_{\Sigma}(\phi, \phi)< 0 \ \ \forall 0\neq\phi\in \scV_{\scB} \}
       \end{align*}
       Since for each $\phi \in C_c^{\infty}(\cU)$, $\phi\nu$ could be extended to a smooth vector field in $M$ supported in $U$, by (\ref{Linear, Equi notion of area decrease quadratically}) and proposition \ref{Prop_Linear, Basic prop for L^2-noncon}, \[
        ind_c(\Sigma, U) \leq ind(\Sigma, U) \leq ind_{\scB}(\Sigma, U) = \sum_{\lambda_j < 0}\dim E_j     \]
       On the  other hand, since $C_c^{\infty}(\cU)$ is dense in $\scB_0(\cU)$, we conclude that $ind_c(\Sigma, U) = ind_{\scB}(\Sigma, U)$. 
      \end{proof}

%%%%%%%%%%%%%%%%%%%%%%%%%%%%%%%%%%%%%%%%%%%%%%%%%%%%%%%%%%%%%%%%%%%%%%%%%%%%
%%%%%%%%%%%%%%%%%%%%%%%%%%%%%%%%%%%%%%%%%%%%%%%%%%%%%%%%%%%%%%%%%%%%%%%%%%%%
%%%%%%%%%%%%%%%%%%%%%%%%%%%%%%%%%%%%%%%%%%%%%%%%%%%%%%%%%%%%%%%%%%%%%%%%%%%%
%%%%%%%%%%%%%%%%%%%%%%%%%%%%%%%%%%%%%%%%%%%%%%%%%%%%%%%%%%%%%%%%%%%%%%%%%%%%
%%%%%%%% Subsection: L^2 nonconcentration for isolated singularities %%%%%%%
%%%%%%%%%%%%%%%%%%%%%%%%%%%%%%%%%%%%%%%%%%%%%%%%%%%%%%%%%%%%%%%%%%%%%%%%%%%%
%%%%%%%%%%%%%%%%%%%%%%%%%%%%%%%%%%%%%%%%%%%%%%%%%%%%%%%%%%%%%%%%%%%%%%%%%%%%
%%%%%%%%%%%%%%%%%%%%%%%%%%%%%%%%%%%%%%%%%%%%%%%%%%%%%%%%%%%%%%%%%%%%%%%%%%%%
%%%%%%%%%%%%%%%%%%%%%%%%%%%%%%%%%%%%%%%%%%%%%%%%%%%%%%%%%%%%%%%%%%%%%%%%%%%%

     \subsection{$L^2$-nonconcentration for strongly isolated singularities} \label{Subsec, L^2 noncon for iso sing}
     
      In this and next subsection, we keep assuming that $\Sigma$ has only strongly isolated singularities. The goal of this subsection is the following 
      \begin{Lem} \label{Linear, L^2 noncon for iso}
       Suppose $\Sigma$ has only strongly isolated singularities and $\partial U \cap Sing(\Sigma) = \emptyset$. Then, $\Sigma$ satisfies the $L^2$-nonconcentration property in $U$.
      \end{Lem}

      The proof is based on a refined analysis near singularities. Since $\Sigma$ has only strongly isolated singularities, by a cut-off argument similar to the proof of lemma \ref{Linear, essential positivity of Q}, it suffices to work near each singularity $p$. In the proof below, we keep using the parametrization and notations of $\Sigma\cap B^M_{\tau_\Sigma(1)}(p)$ as in section \ref{Subsec, Min surface near sing}. Recall that $\cB_r(p)$ denotes the ball in $\Sigma$ with intrinsic distance, and since we have parametrized $\Sigma \cap B^M_{\tau_\Sigma(1)}(p)$ by its tangent cone $C_p \subset T_p M = \RR^{n+1}$ as in lemma \ref{Lem_Pre, geom of asymp at 0}, $p = 0$ and $\cB_r(p)$ is abbreviated as $B_r$. 
      Let $r_2 = r_2(\Sigma, M, g)\in (0, \{\tau(1/10n)/4)$ be fixed such that $\Sigma$ is stable in $B^M_{3r_2}(p)$.

      We shall prove lemma \ref{Linear, L^2 noncon for iso} when $U \cap \Sigma = B_{r_2}$. To do so, first observe that if $C_p$ is strictly stable, then by lemma \ref{Lem_Pre, Hardy-typed inequ} and \ref{Lem_Pre, geom of asymp at 0}, $\exists r_3 = r_3(\Sigma, p)\in (0, r_2)$ and $\delta_{\Sigma}>0$ such that \[
       \delta_{\Sigma}\int_{\Sigma} |\nabla \phi|^2 \leq Q_{\Sigma}(\phi, \phi)\ \ \ \forall \phi\in C^1_c(B_{r_3})     \]
      Hence lemma \ref{Linear, L^2 noncon for iso} holds directly by lemma \ref{Lem_Pre, Uniform L^2-noncon} and a cut off argument. Thus, it suffices to deal with non-strictly stable tangent cones. From now on, assume $\mu_1(C_p) = -(n-2)^2/4$ (see the notations in section \ref{Subsec, Geom of cone}). 
      
      Consider for each $s \in (0,1)$, the perturbed Jacobi operator \[
        L_{\Sigma}^{s} := L_{\Sigma} - s|A_{\Sigma}|^2 = \Delta_{\Sigma} + (1-s)|A_{\Sigma}| + Ric(\nu,\nu)    \]
      and its assciated quadratic form \[
        Q^s_{\Sigma}(\phi, \phi) = Q_{\Sigma}(\phi, \phi) + s\int_{\Sigma}|A_{\Sigma}|^2\phi^2 = \int_{\Sigma} |\nabla^{\Sigma}\phi|^2 - {\big(}(1-s)|A_{\Sigma}|^2 + Ric(\nu, \nu){\big)}\phi^2    \]

      By Lemma \ref{Lem_Pre, Uniform L^2-noncon} and (\ref{Pre, Q^s_Sigma coercive}), there's a continuous function $1 \leq \alpha(r,\omega) \leq 1 + \log(r_1/r)$ on $B_{2r_2}$ (independent of choice of $s$) such that 
      \begin{align} 
        \alpha(x)\to \infty\ \ \ \text{ as } x\to Sing(\Sigma)= \{ p \}  \label{Linear, Weight alpha diverge near sing} 
      \end{align}
      and that for each $\epsilon > 0$, there's a neighborhood $V_{\epsilon, s} \ni p$ such that \[
        \int_{\Sigma\cap V_{\epsilon, s}} \phi^2\cdot \alpha \leq \epsilon\cdot Q_{\Sigma}^s(\phi, \phi)\ \ \ \ \forall \phi\in C_c^1( B_{r_2} )   \]

      Hence, by the similar argument of proposition \ref{Prop_Linear, Basic prop for L^2-noncon}, there's a unique $u^{(s)}\in C_{loc}^0(B_{2r_2})$ with 
      \begin{align} 
       u^{(s)}(r, \cdot) 
        \begin{cases} =0  & \text{ if } r\geq r_2 \\
            >0  & \text{ if } r<r_2
        \end{cases};\ \ \ \ \ \int_{\Sigma} (u^{(s)})^2\cdot \alpha = 1    \label{Linear, Sol u^(s) vanish outside, pos inside} 
      \end{align}
      and 
      \begin{align}
       \lambda_s := Q_{\Sigma}^s(u^{(s)}, u^{(s)}) = \inf\{Q^s_{\Sigma}(\phi, \phi) : \phi\in C_c^1(B_{r_2}),\ \int_{\Sigma} \phi^2\cdot \alpha = 1\}  \label{Linear, Sol u^(s) achieve inf Q^s}
      \end{align}
      In particular, $u^{(s)}$ satisfies the equation 
      \begin{align}
        -L^s_{\Sigma}u^{(s)} = \lambda_s \alpha \cdot u^{(s)} \ \ \ \ \text{ on } B_{r_2}  \label{Linear, Sol u^(s) of L^s=0} 
      \end{align}
      Also by definition and stability of $\Sigma$ in $B_{2r_2}$, for any $0<s_1<s_2<1$ and any $\phi \in C^1_c( B_{r_2})$, $0\leq Q_{\Sigma}^{s_1}(\phi, \phi) \leq Q_{\Sigma}^{s_2}(\phi, \phi)$. Together with (\ref{Linear, Sol u^(s) vanish outside, pos inside}) we see, $0\leq \lambda_{s_1}<\lambda_{s_2}$.

      We need a uniform lower bound on $u^{(s)}$ away from singularities. To derive it, we need the following lemma \ref{Linear, Uniform lower est on u^(s)}. In fact, this is the only place where structures of strongly isolated singularity is essentially made use of.
      \begin{Lem} \label{Linear, Uniform lower est on u^(s)}
       $\exists\ s_0 = s_0(\Sigma, M, g)\in (0,1)$, $\delta = \delta(\Sigma, M, g) > 0$ s.t. for any $0<s<s_0$, \[
        \sup \{ u^{(s)}(r,\omega)^2 : \frac{r_2}{2}< r< r_2 \} \geq \delta    \]
      \end{Lem}
      \begin{proof}
       For each $0<r<r_2$, $s\in (0,1)$, let $m^{(s)}(r):= \inf\{u^{(s)}(r,\omega): \omega\in S\}$. By (\ref{Linear, Sol u^(s) vanish outside, pos inside}) and Harnack inequality, there's a $C_1 = C_1(\Sigma, M, g)>0$ (independent of $s$ and $r$) such that 
       \begin{align}
        u^{(s)}(r, \omega) \leq C_1 m^{(s)}(r)\ \ \ \ \forall 0<r< r_2/2,\ \forall s\in (0,1),\ \forall \omega\in S   \label{Linear, Harnack inequ for u^(s)}
       \end{align} 
       On the other hand, let $0< s_0<<1$ such that \[
         r^2|A_{\Sigma}|_{(r,\cdot)}^2 w_1 \leq \frac{1}{10n\cdot s_0}\ \ \ \ \forall\ 0< r< r_1    \]
       Hence, let $U(r):= (\frac{r}{r_2})^{-(n-1)/2} - (\frac{r}{r_2})^{-(n-3)/2}$. For any $s\in (0, s_0)$, any constant $D > 0$,  
       \begin{align*}
         & L_{\Sigma}^s ( B\cdot U(r)w_1(\omega)) \\ 
       =\ & D\cdot {\big(} U''(r)w_1(\omega) + \frac{U'(r)}{r}H_r(\omega)w_1(\omega) + \frac{U(r)}{r^2}(\cL_r w_1 - sr^2|A_{\Sigma}|^2 w_1) {\big)} \\
       \geq\ & D\cdot {\big(} U''(r) + (n-1 + \frac{1}{10n})\frac{U'(r)}{r}+ (\mu_1 -\frac{2}{10n})\frac{U(r)}{r^2} {\big)} w_1(\omega) \\
       \geq\ & 0\ \ \geq -\lambda_s\alpha\cdot u^{(s)} = L_{\Sigma}^s u^{(s)}
       \end{align*}
       on $B_{r_2}$, where the first equality and the first inequality follow from lemma \ref{Lem_Pre, geom of asymp at 0} and the choice of $r_2$, the second inequality is a direct calculation since $\mu_1 = -(n-2)^2/4$ as assumed, and the third is by (\ref{Linear, Sol u^(s) vanish outside, pos inside}). Hence, by weak maximum principle, $\forall 0< r_3< r_2$, if \[
        D\cdot U(r_3)\sup_S w_1 \leq m^{(s)}(r_3)   \]
       then 
       \begin{align}
        D\cdot U(r)\sup_S w_1 \leq m^{(s)}(r)\ \ \ \ \forall r_3 < r <r_2  \label{Linear, Growth pre-est for u^(s)}  
       \end{align}
       Choose $D>0$ such that \[
        D\cdot U(r_2/2)\sup_S w_1 = m^{(s)}(r_2/2)   \]
       By (\ref{Linear, Growth pre-est for u^(s)}), for any $0<r<r_2/2$,
       \begin{align}
        m^{(s)}(r) \leq D\cdot U(r)\sup_S w_1 = m^{(s)}(r_2/2)\cdot\frac{U(r)}{U(r_2/2)} \label{Linear, Growth est for u^(s)}
       \end{align}
       Thus, by (\ref{Linear, Sol u^(s) vanish outside, pos inside}),
       \begin{align*}
        1 = \int_{\Sigma} (u^{(s)})^2\cdot \alpha = & \int_{ B_{r_2}\setminus B_{r_2/2}} (u^{(s)})^2\cdot \alpha + \int_{B_{r_2/2}} (u^{(s)})^2\cdot \alpha \\
          \leq &\ C(\Sigma, M, g){\Big(} \sup_{B_{r_2}\setminus B_{r_2/2}} (u^{(s)})^2  + \int_0^{r_2/2} r^{n-1}\log r \int_S (u^{(s)})^2(r, \omega)\ d\omega\ dr {\Big)} \\
          \leq &\ C(\Sigma, M, g){\Big(} \sup_{B_{r_2}\setminus B_{r_2/2}} (u^{(s)})^2  + \int_0^{r_2/2} r^{n-1}\log r \cdot m^{(s)}(r)^2\ dr {\Big)} \\
          \leq &\ C(\Sigma, M, g, r_2){\Big(} \sup_{B_{r_2}\setminus B_{r_2/2}} (u^{(s)})^2  + \int_0^{r_2/2} r^{n-1}\log r \cdot m^{(s)}(r_2/2)^2 U(r)^2\ dr {\Big)} \\
          \leq &\ C_2(\Sigma, M, g, r_2)\cdot \sup_{ B_{r_2}\setminus B_{r_2/2}} (u^{(s)})^2  
       \end{align*}
       Where the first inequality follows by definition of $\alpha$ and lemma \ref{Lem_Pre, geom of asymp at 0}, the second inequality follows from (\ref{Linear, Harnack inequ for u^(s)}), the third inequality is by (\ref{Linear, Growth est for u^(s)}) and the last inequality is by \[
        \int_0^1 r^{n-1}\log r\cdot U(r)^2 < +\infty \ dr   \] 
       Hence, the lemma is proved by choosing $\delta = C_2(\Sigma, M, g, r_2)^{-1}$.
      \end{proof}

      \begin{proof}[Proof of lemma \ref{Linear, L^2 noncon for iso}.] 
      Recall that $u^{(s)}$ satisfies the equation (\ref{Linear, Sol u^(s) of L^s=0}) and (\ref{Linear, Sol u^(s) vanish outside, pos inside}), where $\lambda_s > 0$ is decreasing in $s$.
      Let $s\to 0$. Then by classical elliptic estimates \cite{GilbargTrudinger01}, up to a subsequence, 
      \begin{align} 
       u^{(s)} \to u_0\ \ \text{ in } C^{\infty}_{loc}(B_{r_2})\cap W^{1,2}_{loc}\cap C^0_{loc}(B_{r_1}),\ \ \ \ \lambda_s\to \lambda_0  \label{Linear, Sol u^(s) converges to u_0} 
      \end{align}
      Clearly, \[ 
       \lambda_0 \geq \inf\{Q_{\Sigma}(\phi, \phi) : \phi\in C_c^1(\Sigma\cap \BB_{r_2}),\ \int_{\Sigma} \phi^2\cdot \alpha = 1\} \geq 0   \] 
      and the equation holds 
      \begin{align}
       -L_{\Sigma} u_0 = \lambda_0\alpha \cdot u_0   \label{Linear, Sol u_0 of L = -lambda_0}
      \end{align}
      Also, by lemma \ref{Linear, Uniform lower est on u^(s)}, $u_0 \neq 0$.\\
      \textbf{Claim}: $\lambda_0 > 0$. \\
      We first complete the proof assuming this claim to be true: For each $\epsilon > 0$, let $\cV_{\epsilon} := \{x\in B_{r_2}: \alpha(x) >1/(\lambda_0\cdot \epsilon) \}$. By (\ref{Linear, Weight alpha diverge near sing}), $\cV_{\epsilon} \neq \emptyset$ is a neighborhood of $Sing(\Sigma) \cap B^M_{r_1}(p) = \{ p \}$. For any $\phi \in C_c^1(B_{r_2})$ fixed, 
      \begin{align*}
       Q_{\Sigma}(\phi, \phi) =  \lim_{s\to 0} Q_{\Sigma}^s(\phi, \phi) 
         \geq  \lim_{s\to 0} \lambda_s\cdot \int_{\Sigma}\phi^2\cdot \alpha 
          = & \lambda_0 \cdot \int_{\Sigma}\phi^2\cdot \alpha \\
         \geq & \frac{\lambda_0}{\lambda_0 \epsilon}\cdot \int_{\cV_{\epsilon}} \phi^2 
         = \frac{1}{\epsilon}\int_{\cV_{\epsilon}} \phi^2
      \end{align*}
      Hence, $\Sigma$ satisfies the $L^2$-nonconcentration property in $B_{r_2}$.\\      
      \textbf{Proof of the claim}: We shall argue by contradiction. Suppose $\lambda_0 = 0$. By (\ref{Linear, Sol u^(s) vanish outside, pos inside}), (\ref{Linear, Sol u^(s) converges to u_0}) and (\ref{Linear, Sol u_0 of L = -lambda_0}),
      \begin{align}
        L_{\Sigma}u_0 = 0\ \ \text{ on } B_{r_2};\ \ \ \ 
        u_0 (r, \cdot) \begin{cases}    = 0\ & \text{ if } r>r_2 \\
           \geq 0\ & \text{ if } 0<r<r_2     \end{cases} \label{Linear, Sol u_0 of L = 0}
      \end{align}  
      Also notice that for $s\in (0, 1/2)$, \[
        Q_{\Sigma}(u^{(s)}, u^{(s)}) \leq Q_{\Sigma}^s(u^{(s)}, u^{(s)}) = \lambda_s \leq \lambda_{1/2}    \]
      Hence by (\ref{Linear, Sol u_0 of L = 0}) and lemma \ref{Linear, Equi def of scB}, $u_0\in \scB_0(B_{2r_2})$ and $Q_{\Sigma}(u_0, u_0) = 0 $.

      On the other hand, recall that $\Sigma$ is stable in $B_{2r_2}$, i.e. \[
        Q_{\Sigma}(\phi, \phi) \geq 0 \ \ \ \ \forall \phi \in C_c^1 (B_{2r_2})   \]
      Therefore, $u_0 \neq 0$ achieves the infimum of $Q_{\Sigma}$ in $\scB_0(B_{2r_2})$. The Euler-Lagrangian equation gives \[
       L_{\Sigma}u_0 = 0 \ \  \text{ on } B_{2r_2}    \]
       However, (\ref{Linear, Sol u_0 of L = 0}) contradict to the strong maximum principle of $u_0$ in $B_{2r_2}$. This completes the proof of the claim.  
      \end{proof}
      \begin{Rem}
       The proof of the Claim above also tells that $-L_{\Sigma}$ is strictly positive on $\scB_0(B_{r_2})$.
      \end{Rem}

%%%%%%%%%%%%%%%%%%%%%%%%%%%%%%%%%%%%%%%%%%%%%%%%%%%%%%%%%%%%%%%%%%%%%%%%%%%%
%%%%%%%%%%%%%%%%%%%%%%%%%%%%%%%%%%%%%%%%%%%%%%%%%%%%%%%%%%%%%%%%%%%%%%%%%%%%
%%%%%%%%%%%%%%%%%%%%%%%%%%%%%%%%%%%%%%%%%%%%%%%%%%%%%%%%%%%%%%%%%%%%%%%%%%%%
%%%%%%%%%%%%%%%%%%%%%%%%%%%%%%%%%%%%%%%%%%%%%%%%%%%%%%%%%%%%%%%%%%%%%%%%%%%%
%%%%%%%%%%%% Subsection: Asymptotic analysis near singularities %%%%%%%%%%%%
%%%%%%%%%%%%%%%%%%%%%%%%%%%%%%%%%%%%%%%%%%%%%%%%%%%%%%%%%%%%%%%%%%%%%%%%%%%%
%%%%%%%%%%%%%%%%%%%%%%%%%%%%%%%%%%%%%%%%%%%%%%%%%%%%%%%%%%%%%%%%%%%%%%%%%%%%
%%%%%%%%%%%%%%%%%%%%%%%%%%%%%%%%%%%%%%%%%%%%%%%%%%%%%%%%%%%%%%%%%%%%%%%%%%%%
%%%%%%%%%%%%%%%%%%%%%%%%%%%%%%%%%%%%%%%%%%%%%%%%%%%%%%%%%%%%%%%%%%%%%%%%%%%%
  
     \subsection{Asymptotic analysis near singularities} \label{Subsec, Linear asymp near sing}
      The aim of this subsection is to study the asymptotic behavior of special functions near strongly isolated singularities. They will be used in  section \ref{Sec, Asymp & Asso Jac field} to describe associated Jacobi fields and in section \ref{Sec, One-sided Perturb} to construct desired one-sided perturbations. We keep the assumptions on $\Sigma$, $(M,g)$, $U$, $p\in Sing(\Sigma)\cap U$ and $r_1>r_2>0$ as in section \ref{Subsec, L^2 noncon for iso sing}. Also keep parametrizing $\Sigma\cap B^M_{r_1}(p)$ by the tangent cone $C_p$ and using the notation that $B_r:= \cB_r(p)$. Recall that $-L_{\Sigma}$ is strictly positive on $\scB_0(B_{r_2})$. 
      For later applications, we shall consider operators of a more general form: $L^h:= L_{\Sigma} - h$ and $Q^h(\phi, \phi):= Q_{\Sigma}(\phi, \phi) + \int_{\Sigma}h\cdot\phi^2$, where $h\in L^{\infty}(\Sigma)$ be fixed throughout this subsection. Assume that $-L^h$ is also strictly positive on $\scB_0(B_{r_2})$.
      
      We first slightly extend the function space we are dealing with.
      
      \begin{Lem} \label{Lem_Linear, Loc Green's func unique}
       Up to a normalization, there exists a unique $G \in C_{loc}^{\infty}(Clos(B_{r_2})$ such that $L^h G = 0$ on $B_{r_2}(p)$, $G>0$ on $B_{r_2}$, $G = 0$ on $\partial B_{r_2}$. 
      \end{Lem}
      \begin{proof}
       We first prove the existence. Consider $(r_2/2>)t_j\to 0_+$ and the solution $u_j$ of $L^h u_j = 0$ on $A_{t_j, r_2}$, $u_j = 0$ on $\partial B_{r_2}$, $u_j = 1$ on $\partial B_{t_j}$. By weak maximum principle, $u_j >0$ on $A_{t_j, r_2}$. By Harnack inequality and standard elliptic estimate, $u_j/\sup_{\partial B_{r_2/2}} u_j \to G$ in $C^{\infty}_{loc}(Clos(B_{r_2}))$ and $G>0$ on $B_{r_2}$, $G=0$ on $\partial B_{r_2}$.
       
       To see the uniqueness, if $G'$ is another positive smooth function vanishing on $\partial B_{r_2}$ and satisfies the equation above, first note that by weak maximum principle, $\forall t\in (0, r_2)$, 
       \begin{align}
        \sup_{A_{t, r_2}} G/G' = \sup_{\partial B_t}G/G'\ \ \text{ and }\ \   \inf_{A_{t, r_2}} G/G' = \inf_{\partial B_t}G/G'  \label{Linear, extrema G/G' achieved on bdy}
       \end{align}
       and by Harnack inequality, $\sup_{\partial B_t}G/G' \leq C(\Sigma)\inf_{\partial B_t} G/G'$, $\forall t\in (0, r_2/2)$. Hence, \[
        0< a_-:= \inf_{B_{r_2}} G/G' \leq a_+:=\sup_{B_{r_2}} G/G' <+\infty   \]
       We now show that $a_- = a_+$, with which the lemma is proved.
       
       By (\ref{Linear, extrema G/G' achieved on bdy}), $a_- = \liminf_{t\to 0_+}\inf_{\partial B_t} G/G'$ and $a_+ = \liminf_{t\to 0_+}\inf_{\partial B_t} G/G'$. Thus, up to a subsequence of $t\to 0_+$, $G_t(x):= G(tx)/\inf_{\partial B_t} G \to G_0$, $G'_t(x):= G'(tx)/\inf_{\partial B_t} G \to G_0'$ and $G_0, G_0'$ are positive Jacobi fields on $C_p$ with $a_- = \inf_{\partial B_1}G_0 /G_0'$, $a_+ = \sup_{\partial B_1}G_0 /G_0'$. By lemma \ref{Lem_Pre, pos Jacob field}, $a_- = a_+$.
      \end{proof}
      
      Call such $G$ in lemma \ref{Lem_Linear, Loc Green's func unique} a \textbf{Green's function} of $L^h$ at $p$ in $B_{r_2}$.
      Since $-L^h$ is strictly positive on $\scB_0(B_{r_2})$, $G$ doesn't belong to $\scB(B_{r_2})$. Just like Green's functions in a smooth domain, $G$ dominates certain functions in $\scB(B_{r_2})$.
      \begin{Lem} \label{Lem_Linear, Green's function dominates}
       Let $\varphi_1\in \scB_0(B_{r_2})$ be the first eigenfunction of $-L^h$ in $\scB_0(B_{r_2})$ with $\|\varphi_1\|_{L^2}=1$; $\lambda_1>0$ to be the corresponding eigenvalue, i.e. \[
        \lambda_1 = \inf\{Q^h(\phi, \phi): \phi\in \scB_0(B_{r_2}), \|\phi\|_{L^2}=1\} = Q^h(\varphi_1, \varphi_1) > 0    \] 
       Then $\lim_{r\to 0} \sup_{\partial B_r} (\varphi_1/G) = 0$.
       
       In particular, if $u\in \scB(B_{r_2})$ satisfies $L^h u \in L^{\infty}(B_{r_2})$, then $\lim_{r\to 0} \sup_{\partial B_r} (|u|/G) = 0$
      \end{Lem}
      \begin{proof}
       By weak maximum principle and Harnack inequality, for each $r\in (0, r_2)$, \[
        \sup_{\partial B_r} \varphi_1/G \leq C(\Sigma)\inf_{\partial B_r} \varphi_1/G = C(\Sigma)\inf_{A_{r, r_2}} \varphi_1/G    \]
       Hence, if assume for contradiction that $\limsup_{r\to 0} \sup_{\partial B_r} (\varphi_1/G) >0$, then by a renormalization in $G$, WLOG $G\leq \varphi_1$ on $B_{r_2}$.
       
       For each $r<r_2$, consider the minimizer $v_r$ of $Q^h(\phi, \phi)$ among \[
        \{\phi\in \scB_0(B_{r_2}): \phi = \varphi_1 \text{ in }B_r\}    \]
       The existence of such minimizer follows from proposition \ref{Prop_Linear, Basic prop for L^2-noncon} and lemma \ref{Linear, L^2 noncon for iso}. Moreover, $L^h v_r = 0$ on $A_{r, r_2}$; $G\leq v_r\leq \varphi_1$ on $B_{r_2}$ by weak maximum principle; $\|v_r\|_{\scB(B_{r_2})} \leq (1+ \|h\|_{L^{\infty}})\|\varphi_1\|_{\scB(B_{r_2})}$ by definition. 
       Let $r\to 0_+$, up to a subsequence, $v_r \to v_0$ in $C^{\infty}_{loc}(Clos(B_{r_2}))$ for some $v_0\in \scB_0(B_{r_2})$ satisfying $L^h v_0 = 0$ on $B_{r_2}$ and $G\leq v_0$. But since $L^h$ is nondegenerate in $\scB_0(B_{r_2})$, $v_0 \equiv 0$, contradicts to that $G>0$ on $B_{r_2}$.
       
       To see that near $p$ $G$ also dominates $u\in \scB(B_{r_2})$ with $L^h u \in L^{\infty}(B_{r_2})$, first observe that by Harnack inequality, lemma \ref{Lem_Pre, pos Jacob field} and a blow up argument, $\varphi_1(x)\to +\infty$ as $x\to p$. Hence there exists $C>>1$ depending on $u$ and $\varphi_1$ such that $|L^h u| \leq -C\cdot L^h \varphi_1$ in $B_{r_2/2}$ and $|u|\leq C\varphi_1$ on $\partial B_{r_2/2}$. Then by proposition \ref{Prop_Linear, Basic prop for L^2-noncon} (4), $|u|\leq C\varphi_1$ on $B_{r_2/2}$ and the lemma is proved.
      \end{proof}
      
%      \begin{Cor} \label{Cor_Linear, Decomp of pos Jac by G+scB}
%       Let $G$ be a Green's function of $L^h$ at $p$ as above. If $0\leq u\in W^{2,2}_{loc}(B_{r_2})$ satisfies $L^h u \in L^{\infty}(B_{r_2})$, then $\exists\ c\geq 0$ such that $u - cG \in \scB(B_{r_2/2})$.
%      \end{Cor}
%      \begin{Pf}
%       Harnack ineq + barrier argument by Green's function + lemma \ref{Lem_Linear, Loc Green's func unique}.
%      \end{Pf}

      Now we introduce the notion of \textbf{asymptotic rate} at $p$. Recall for $\sigma\in \RR\setminus \Gamma_{C_p}$, \[
       J^{\sigma}_u(r, s):= \int_{A_{r, s}} u^2(x)|x|^{-n-2\sigma}    \]
      is introduced in (\ref{Pre, J^sigma_u(r, s)}). Here, the integration is with respect to volume density on $C_p\subset T_p M= \RR^{n+1}$. Define the asymptotic rate of $u\in L^2_{loc}(B_{r_2})$ at $p$ to be
      \begin{align}
       \cA\cR_p(u):=\sup \{\sigma : \lim_{t\to 0+} J^{\sigma}_u(t, 2t) = 0\}
      \end{align}
      where we use the convention that $\sup \emptyset = -\infty$ and $\cA\cR_p(0) = +\infty$.
      \begin{Lem} \label{Lem_Linear, Growth rate charact}
       Suppose $0\neq u\in W^{2,2}_{loc}(B_{r_2})$ such that $L^h u\in L^{\infty}$. Also suppose that $\gamma:=\cA\cR_p(u) \in (-\infty, 1)$, then 
       \begin{enumerate}
       \item[(1)] $\forall \sigma\in (\gamma, 1)\setminus \Gamma_{C_p}$, let $K_0(\sigma, C_p)>2$ be determined in lemma \ref{Lem_Pre, growth rate mon for Jac}. For every $K\geq K_0(\sigma, C_p)$, there exists some $t_0=t_0(K, u, \sigma, C_p, L^h, \Sigma, M, g)\in (0, r_2)$ such that $l\mapsto J^{\sigma}_u(t_0 K^{-l-1}, t_0 K^{-l})$ is strictly increasing in $l\in \NN$;
       \item[(2)] $\exists t_j\to 0_+$ and $c_j\to 0_+$ such that $u_j(x):= u(t_j x)/c_j \to u_{\infty}$ in $W^{1,2}_{loc}(C_p)$ for some $u_{\infty}(r, \omega)= r^{\gamma}(c+ c'\log r)\cdot w$, where $0\neq w\in W_{\gamma}(C_p)$ is the corresponding eigenfunction of $\gamma$ defined in section \ref{Subsec, Geom of cone}; $(c, c')\neq (0, 0)$ and $c' = 0$ unless $\gamma = -(n-2)/2$. 
       
       In particular, $\cA\cR_p(u)\in \Gamma_{C_p}$.
       \end{enumerate}
      \end{Lem}
      \begin{proof}
       To prove (1), first notice that for every $\sigma'<\sigma$ we have $J^{\sigma}_u(r,s)\geq J^{\sigma'}_u(r, s)\cdot s^{2(\sigma'-\sigma)}$. Hence by definition of $\cA\cR_p(u)$, $\forall \sigma\in (\cA\cR_p(u), 1)\setminus \Gamma_{C_p}$, $\forall K \geq K_0(\sigma, C_p)>2$ as in lemma \ref{Lem_Pre, growth rate mon for Jac} fixed, $\limsup_{t\to 0_+}J^{\sigma}_u(t, Kt)\geq \limsup_{t\to 0_+}J^{\sigma}_u(t, 2t) = +\infty$. Thus, there exists $s_j\to 0_+$ such that for every $j$, \[
        J^{\sigma}_u(Ks_j, K^2s_j)\geq J^{\sigma}_u(t, Kt)\ \ \ \ \forall t\in (Ks_j, r_2/K)    \]
       
       Also note that $u$ satisfies the equation $L^h u = f\in L^\infty (\Sigma)$. For each $r\in (0, r_2)$, let $u^{(r)}(x):= u(r x)/r$ defined on $B_{r_2/r}\subset C_p$, then by lemma \ref{Lem_Pre, geom of asymp at 0}, $u^{(r)}$ satisfies an equation of form 
       \begin{align}
        & div_C(\nabla_C u^{(r)} + b_0^{(r)}(x, u^{(r)}, \nabla_C u^{(r)})) + |A_C|^2u^{(r)} + b_1^{(r)} (x, u^{(r)}, \nabla_C u^{(r)}) = 0  \label{Linear_An ell equ satisfied by u^{(r)}}
       \end{align}
       where $|b_0^{(r)}(x, z, p)| + |b_1^{(r)}(x, z, p)| \leq \epsilon(r)(|z|/|x|+ |p|) + r\cdot \|f\|_{L^{\infty}(\Sigma)}$ on $A_{1, r_2/r}$ and $\epsilon(r)\to 0_+$ as $r\to 0_+$. 
       Also by changing variable, $J^{\sigma}_{u^{(r)}}(s, t) = J^{\sigma}_u(rs, rt)\cdot r^{2(\sigma -1)}$. Therefore, we conclude that there exists $J>>1$, $u^{(s_{J})}$ satisfies the assumptions in corollary \ref{Cor_Pre, quant growth rate mon for Jac}. Then by applying corollary \ref{Cor_Pre, quant growth rate mon for Jac} inductively to $u^{(s_J\cdot K^{-l})}$, we have $l \mapsto J^{\sigma}_u(s_J\cdot K^{-l-1}, s_J\cdot K^{-l})$ is  strictly increasing. Take $t_0:= s_J$\\
       
       To prove (2), consider $\sigma'\in (-\infty, \cA\cR_p(u))$ and $\sigma\in (\cA\cR_p(u), 1)\setminus \Gamma_{C_p}$ TBD; Let $K\geq K_0(C_p, \sigma)> 2$ be fixed as in (1). By definition of $\cA\cR_p(u)$, $\lim_{t\to 0_+} J^{\sigma'}_u(t, Kt) = 0$. Thus, there exists $t_j\to 0_+$ such that \[
        J^{\sigma'}_u(t_j, K t_j)\geq J^{\sigma}_u(t, Kt)\ \ \ \ \forall t\in (0, t_j)    \]
       Therefore, together with (1) we see that if denote $u_j:= u^{(t_j)}/\|u^{(t_j)}\|_{L^2(A_{1, K})}$, then $u_j$ satisfies an elliptic equation of form (\ref{Linear_An ell equ satisfied by u^{(r)}}) and a local-bound-global estimate for $j>>1$: \[
        \|u_j\|_{L^2(A_{R^{-1}, R})} \leq C(R, \sigma, \sigma')\|u_j\|_{L^2(A_{1, K})}\ \ \ \forall 1<R<t_0/t_j   \] 
       By standard elliptic estimate \cite{GilbargTrudinger01}, up to a subsequence, $u_j \to u_\infty$ in $W^{1,2}_{loc}(C_p)$ and $u_\infty$ satisfies the Jacobi field equation $L_{C_p}u_\infty = 0$ and the growth estimate near $0$ and $\infty$: 
       \begin{align}
        \limsup_{r\to 0_+} J^{\sigma'}_{u_\infty}(r, Kr) + \limsup_{r\to +\infty}J^{\sigma}_{u_{\infty}}(r, Kr) < +\infty  \label{Linear_Growth est at 0 and infty}       
       \end{align}
       Now choose $\sigma, \sigma'$ such that $ \Gamma_{C_p} \cap [\sigma', \sigma]\setminus \{\cA\cR_p(u)\} = \emptyset$. Then by lemma \ref{Lem_Pre, Bdness of Jac field near 0/infty} and (\ref{Linear_Growth est at 0 and infty}), $\cA\cR_p(u)\in \Gamma_{C_p}$ and $u_{\infty}(r, \omega)= r^{\gamma}(c + c'\log r)\cdot w(\omega)$ where $w\in W_{\gamma}$, $(c, c')\neq (0, 0)$, $c' = 0$ unless $\gamma = -(n-2)/2$. This completes the proof.
      \end{proof}
      
      \begin{Cor} \label{Cor_Linear, Growth est for G and vaiphi_1}
       Let $G$ be a Green's function of $L^h$ in $B_{r_2}$ at $p$; 
       $\varphi_1\in \scB_0(B_{r_2})$ be the first eigenfunction of $-L^h$ in $\scB_0(B_{r_2})$;
       $\xi \in \scB(B_{r_2})$ be the solution of $ - L^h \xi = 0$ on $B_{r_2}$ and $\xi = 1$ on $\partial B_{r_2}$. 
       
       Then $\cA\cR_p(G) = \gamma_1^-$, $\cA\cR_p(\varphi_1)= \cA\cR_p(\xi)= \gamma_1^+$.
      \end{Cor}
      \begin{proof}
       By weak maximum principle, $\varphi_1, \xi>0$; Also by definition, $G>0$. By lemma \ref{Lem_Linear, Growth rate charact} (2), lemma \ref{Lem_Pre, pos Jacob field} and Harnack inequality, $\cA\cR_p(G), \cA\cR_p(\varphi_1), \cA\cR_p(\xi) \in \{\gamma_1^{\pm}\}$. And to distinguish them, it suffices to assume in addition that $C_p$ is strictly stable, in which case $\gamma_1^+ > -(n-2)/2 > \gamma_1^-$. 
       
       With this assumption, by lemma \ref{Lem_Pre, Hardy-typed inequ} and \ref{Lem_Pre, geom of asymp at 0}, $\|\cdot\|_{\scB} \sim \|\cdot\|_{W^{1,2}}$ and in particular $\scB_0(B_{r_2}) = W^{1,2}_0(B_{r_2})$. Therefore, combined with lemma \ref{Lem_Linear, Growth rate charact} (1) and the standard elliptic estimate, we have for every $u\in \{\varphi_1, G, \xi\}$, $\cA\cR_p(u) = \gamma_1^+$ if and only if $u\in W_0^{1,2}(B_{r_2})$. Thus proves the corollary. 
      \end{proof}
      
      \begin{Cor} \label{Cor_Linear, Unique Asymp for Green's func}
       Let $G$ be a Green's function of $L^h$ in $B_{r_2}$ at $p$. Then \[
        \lim_{r\to 0_+} G(rx)/\inf_{\partial B_r}G = G_0(x)   \]
       where $G_0(r, \omega)=r^{\gamma^-_1}w_1(\omega)/\inf_{\partial B_1} w_1$ is a homogeneous positive Jacobi field on $C_p$ with fast asymptotic rate. 
       Moreover, the convergence is in $C^{\infty}_{loc}(C_p)$. 
      \end{Cor}
      \begin{proof}
       First note that by Harnack inequality, when $r\to 0_+$, $G_{(r)}(x):= G(rx)/\inf_\partial B_r$ always subconverges in $C_{loc}^{\infty}(C_p)$ to some positive Jacobi field on $C_p$ with infimum $1$ on $\partial B_1$. The goal now is to show that the limit is unique. 
       
       By lemma \ref{Lem_Pre, pos Jacob field}, if $C_p$ is not strictly stable, then there's only one such positive Jacobi field on $C_p$. Hence, it suffices to prove in case when $C_p$ is strictly stable, i.e. $\gamma_1^-<\gamma_1^+$. 
       By corollary \ref{Cor_Linear, Growth est for G and vaiphi_1}, $\cA\cR_p(G) = \gamma_1^-$; By applying lemma \ref{Lem_Linear, Growth rate charact} (1) to $\sigma\in (\gamma_1^-, \gamma_1^+)$, it's easy to see that for every sequence $r_j\to 0_+$ such that $G_{(r_j)}\to \hat{G}_0=(c^+ r^{\gamma_1^+} + c^-r^{\gamma_1^-})w_1$ we have, $l \mapsto J^{\sigma}_{\hat{G}_0}(\hat{t}_0 K^{-l-1}, \hat{t}_0 K^{-l})$ is increasing in $l\in \ZZ$ for some $\hat{t}_0\in [1, K)$. Hence $c^+ = 0$ and thus $\hat{G}_0$ is of the form in corollary \ref{Cor_Linear, Unique Asymp for Green's func}.
      \end{proof}

      \begin{Cor} \label{Cor_Linear, Loc Growth Bd and function decomp}
       Suppose $u\in W^{2,2}_{loc}(B_{r_2})$ with $L^h u \in L^{\infty}(B_{r_2})$. $G$ be a Green's function of $L^h$ in $B_{r_2}$ at $p$. Then,
       \begin{enumerate}
       \item[(1)] If $u\in \scB(B_{r_2})$, then $\cA\cR_p(u)\geq \gamma_1^+$;
       \item[(2)] If $\cA\cR_p(u)\geq \gamma_1^-$, then there exists a unique $c\in \RR$ such that $u - c\cdot G \in \scB(B_{r_2/2})$.
%       \item[(3)] $\cA\cR_p(u) \in \{\gamma_1^{\pm}\}$ if and only if $u \neq 0$ has a fixed sign in some neighborhood of $p$.
       \end{enumerate}
      \end{Cor}
      \begin{proof}
       (1) follows directly from weak maximum principle and lemma \ref{Cor_Linear, Growth est for G and vaiphi_1}, which assert that $\cA\cR_p(\varphi_1) = \gamma_1^+$.
       
       To prove (2), consider $v\in \scB(B_{r_2/2})$ be the solution of $L^h v = L^h u$ on $B_{r_2/2}$, $v = u$ on $A_{r_2/2, r_2}$. Let $w:= u-v$. By (1) and definition we see that $\cA\cR_p(w) \geq \gamma_1^-$ and $L^h(u-v) = 0$ on $B_{r_2/2}$, $w = 0$ outside $B_{r_2/2}$. 
       
       If $\cA\cR_p(w) >\gamma_1^-$, then by corollary \ref{Cor_Linear, Growth est for G and vaiphi_1}, $\limsup_{t\to 0_+} \sup_{\partial B_t} |w|/G' = 0$, where $G'$ denotes the Green's function of $L^h$ at $p$ in $B_{r_2/2}$. Then by weak maximum principle, $w = 0$.
       
       If $\cA\cR_p(w) = \gamma_1^-$, then by lemma \ref{Lem_Linear, Growth rate charact} and lemma \ref{Lem_Pre, Bdness of Jac field near 0/infty}, WLOG there exists $t_j \to 0_+$ such that $w>0$ on $\partial B_{t_j}$. Then by classical weak and strong maximum principle for elliptic PDE, $w> 0$ on $B_{r_2/2}$. Hence by lemma \ref{Lem_Linear, Loc Green's func unique}, $w = c'G'$ for some $c'\in \RR\setminus \{0\}$. 
       To go back to a Green's function on $B_{r_2}$, consider $\eta\in C_c^{\infty}(B^M_{r_2/4}(p))$ which equals to $1$ near $p$ and the solution $\psi\in \scB_0(B_{r_2})$ of $L^h \psi = L^h(\eta\cdot w)$ on $B_{r_2}$. By lemma \ref{Lem_Linear, Green's function dominates} and weak maximum principle, $-\psi+\eta\cdot w$ is a Green's function at $p$ in $B_{r_2}$, combined with lemma \ref{Lem_Linear, Loc Green's func unique}, $u - cG \in \scB(B_{r_2/2})$ for some $c\in \RR$, which completes the proof.
      \end{proof}

      Now we turn to the globally defined functions on $\Sigma$. Recall that $\cU := U\cap \Sigma$. Denote $G_{r, p}$ to be a Green's function of $L^h$ in $\cB_r(p)$, where $p\in Sing(\Sigma)$, $r\in (0, r_2]$.
      \begin{Def}
      Define $\scB_0(\cU)\bigoplus \RR_{L^h}\langle Sing(\Sigma)\cap U\rangle$ to be the function space linearly generated by $\scB_0(\cU)$ and $G_{r, p}$, $p\in Sing(\Sigma)\cap U$.       
      \end{Def}
      The proof of corollary \ref{Cor_Linear, Loc Growth Bd and function decomp} (2) actually guarantees that this space depends only on $\Sigma, U$ and $p$, but not on the choice of $r$. 

      \begin{Cor} \label{Cor_Linear, Global Green's func exists}
       Suppose $L^h$ be nondegenerate on $\scB_0(\cU)$. Then, for each $p\in Sing(\Sigma)\cap U$, there is a unique (up to normalization) $0\neq G_p\in \scB_0(\cU)\bigoplus \RR_{L^h}\langle Sing(\Sigma)\cap U\rangle$ such that $L^h G_p= 0$ on $\cU$, $G_p \in \scB(\cU\setminus Clos(B^M_{r_2}(p))$. 
       Moreover, $\scB_0(\cU)\bigoplus \RR_{L^h}\langle Sing(\Sigma)\cap U\rangle$ is linearly generated by $\scB_0(\cU)$ and $\{G_p : p\in Sing(\Sigma)\cap U\}$.
       
       If further $-L^h$ is strictly positive on $\scB_0(\cU)$, then $G_p$ can be chosen positive on $\cU$.
      \end{Cor}
      Call such $G_p$ a \textbf{Green's function of $L^h$ on $\cU$ at $p$}.
      \begin{proof}
       The existence of $G_p$ follows the same construction in the proof of corollary \ref{Cor_Linear, Loc Growth Bd and function decomp} (2) from $G'$ to $G$;
       The uniqueness of such $G_p$ up to normalization follows directly from lemma \ref{Lem_Linear, Loc Green's func unique}, \ref{Lem_Linear, Green's function dominates} and corollary \ref{Cor_Linear, Growth est for G and vaiphi_1}, \ref{Cor_Linear, Loc Growth Bd and function decomp}; 
       The positivity of $G_p$ when $-L^h$ is strictly positive follows from weak maximum principle.
      \end{proof}

      \begin{Cor} \label{Cor_Linear, Global Growth Bd function decomp}
       Suppose $u\in W^{2,2}_{loc}(\cU)$, $L^h u \in L^{\infty}(\Sigma)$ and $u|_{\partial \cU} = 0$. If $\cA\cR_p(u)\geq \gamma_1^-$ for every $p\in Sing(\Sigma)\cap U$, then $u\in \scB_0(\cU)\bigoplus \RR_{L^h}\langle Sing(\Sigma)\cap U\rangle$.
      \end{Cor}
      \begin{proof}
       This directly follows from corollary \ref{Cor_Linear, Loc Growth Bd and function decomp}.
      \end{proof}
      
      The following lemma shows that having top asymptotic growth rate near each singularity is a generic property for solutions of $L^h u = f$.
      \begin{Lem} \label{Lem_Genericness of Sol with Top Growth Rate}
       Suppose $L^h$ be nondegenerate on $\scB_0(\cU)$, $p\in Sing(\Sigma)\cap U$. Then there exists an open subset $\scV_p \subset L^{\infty}(\cU)$ satisfying
       \begin{enumerate}
       \item[(1)] $f\in \scV_p$ implies $cf\in \scV_p$, $\forall c\in \RR\setminus \{0\}$;
       \item[(2)] If $f\in \scV_p$ and $g\in L^{\infty}(\cU)\setminus \scV_p$, then $f+g \in \scV_p$;
       \item[(3)] $\scV_p \cap C^{\infty}_c(\bar{\cU})$ is dense in $C^{\infty}_c(\bar{\cU})$;
       \item[(4)] If $u \in \scB_0(\cU)$ is the solution of $L^h u = f$ for some $f\in \scV_p$, then $\cA\cR_p(u) = \gamma_1^+(C_p)$. 
       \end{enumerate}
       In particular, for every $k\in \NN\cup \{0\}$ and $\alpha\in [0, 1]$, $\{f\in C^{k, \alpha}(Clos(U)): f|_{\Sigma} \in \cap_{p\in Sing(\Sigma)\cap U}\scV_p\}$ is open and dense in $C^{k, \alpha}(Clos(U))$.
      \end{Lem}
      \begin{proof}
       For each $p\in Sing(\Sigma)\cap U$, define $\xi_p\in \scB(\cB_{r_2}(p))$ to be the solution of $L^h \xi_p = 0$ with $\xi_p|_{\partial \cB_{r_2}(p)} = 1$; And let $\check{\xi}_p\in \cB_{r_2}(p)$ be the solution of $(-L^h + 1)\check{\xi}_p = 0$ with $\check{\xi}_p|_{\partial \cB_{r_2}(p)} = 1$. Then by weak maximum principle, $\xi_p\geq \check{\xi}_p > 0$ on $\cB_{r_2}(p)$ for every $p\in Sing(\Sigma)\cap U$, and by corollary \ref{Cor_Linear, Growth est for G and vaiphi_1} $\check{\xi}_p(x)\to +\infty$ as $x\to p$.
       
       Now fix p, define \[
        \scV_p:= \{f\in L^\infty (\cU): u:= (L^h)^{-1}f \in \scB_0(\cU) \text{ satisfies }\limsup_{x\to p} |u|(x)/\xi_p(x) >0 \}   \]
       (1), (2) directly follows from this definition; (4) follows from lemma \ref{Lem_Linear, Growth rate charact} and corollary \ref{Cor_Linear, Growth est for G and vaiphi_1}. If (1)-(4) and openness of $\scV_p$ is verified, then by Baire category theorem, \[
        \bigcap_{p\in Sing(\Sigma)\cap U} \{f\in C^{k, \alpha}(Clos(U)): f|_{\Sigma} \in \scV_p\}\subset C^{k, \alpha}(Clos(U))\  \text{ is open and dense}    \]
               
       Now we proceed to prove (3) and openness of $\scV_p$ in $L^\infty$. \\
       \textbf{Proof of openness}: Clearly, it suffices to show that there exists a $C=C(h, \Sigma, U, M, g)>0$ such that for every $\|f\|_{L^{\infty}(\cU)}\leq 1$ we have 
       \begin{align}
        \limsup_{x\to p}|(L^h)^{-1}f|(x)/\xi_p(x) \leq C  \label{Linear_Asymp bded by xi_p}
       \end{align} 
       To prove this bound, denote $(L^h)^{-1}f$ by $v$. First recall that by proposition \ref{Prop_Linear, Basic prop for L^2-noncon} (3) and interior $C^0$-estimate, $|v|\leq C_1 = C_1(h, \Sigma, U, M, g, r_2)$ on $\cA_{r_2/2, r_2}$. Let $\kappa:= \inf_{\cB_{r_2}(p)} \check{\xi}_p (>0)$, $C:= max\{1, C_1\}/\kappa$. Then by $|L^h v| \leq 1 \leq -C\cdot L^h(2\xi_p - \check{\xi}_p)$ on $\cB_{r_2}(p)$, $|v|\leq C\kappa \leq C(2\xi_p - \check{\xi}_p)$ on $\partial \cB_{r_2}(p)$ and weak maximum principle, we see that $|v|\leq C(2\xi_p - \check{\xi}_p)$ on $\cB_{r_2}(p)$. Hence (\ref{Linear_Asymp bded by xi_p}) is proved. \\
       \textbf{Proof of (3)}: By (1) and (2), it suffices to show that $\scV_p\cap C^\infty_c(\bar{\cU})\neq \emptyset$. To see this, first recall that by proposition \ref{Prop_Linear, Basic prop for L^2-noncon} (5), there exists $\vartheta\in C_c^\infty (\bar{\cU}; \RR_+)$ such that $-L^{h + \vartheta}$ is strictly positive. 

       Let $w:= (-L^{h + \vartheta})^{-1}(\vartheta) \in \scB_0(\cU)$, then $w>0$ on $\cU$ by weak and strong maximum principle; Let $r_3\in (0, r_2)$ such that $\cB_{r_3}(p)\cap spt(\vartheta) =\emptyset$. Then $w$ satisfies the equation $-L^h w = \vartheta - \vartheta w = 0$ on $\cB_{r_3}(p)$, and by weak maximum principle, $w \geq \varepsilon \xi_p$ for some $\varepsilon>0$. Therefore, $\vartheta -\vartheta w\in \scV_p \cap C_c^{\infty}(\bar{\cU})$.
      \end{proof}

    \section{Asymptotics and associated Jacobi fields} \label{Sec, Asymp & Asso Jac field}

     Let $(\Sigma,\nu) \subset (M,g)$ be a minimal hypersurface with normal field $\nu$; $U\subset \subset M$ be a smooth domain as in section \ref{Sec, Prelim}. 
     Suppose $\{g_j\}$ be a family of metric smoothly converges to $g_0$, $V_j$ be a family of stationary integral varifolds in $(U, g_j)$ \textbf{different from} $|\Sigma|$ and converging to $|\Sigma|$. Then by Allard regularity theorem \cite{Allard72, Simon83_GMT}, the convergent is in $C^{\infty}_{loc}(U\cap\Sigma)$

     \begin{Def}
      Call $0\neq \phi \in C^{\infty}(\Sigma\cap U)$ a \textbf{generalized Jacobi field associated to} $\{(V_j, g_j)\}_{j\geq 1}$, if $\exists\ t_j \to 0_+$ and $v_j \in C^{\infty} (\Sigma\cap U)$ such that $v_j/t_j \to \phi $ in $C^{\infty}_{loc}(U\cap \Sigma)$ and that $\forall \ U' \subset \subset U\setminus Sing(\Sigma)$, \[
      |graph_{\Sigma}(v_j)|\llcorner U' = V_j\llcorner U' \ \ \ \text{ for sufficiently large }j   \]
     \end{Def}
     
     Clearly, if $\phi$ is an associated Jacobi operator to $\{(V_j, g_j)\}$, then so is $c\phi$ for $c>0$. Also, if $g_j = g$ and $\phi$ is a Jacobi field associated to some $(V_j, g)$, then $\phi$ satisfies the Jacobi field equation $L_{\Sigma} \phi = 0 $.
     
     When $g \equiv g_j$, $U=M$ and $\Sigma$ is regular, after passing to subsequences, the existence of associated Jacobi field directly follows from Allard regularity and elliptic estimates, see \cite{Sharp17}.  However, if $Sing(\Sigma) \neq \emptyset$, even when $g_j=g$, it's unclear whether for any such family of stationary integral varifold $\{V_j\}$, there exists an associated Jacobi field.\\
     
     From now on, suppose that $\Sigma$ is closed, locally stable and has only strongly isolated singularities, and $U = M$. Let $\beta$ be a symmetric 2-tensor on $M$ such that \[
      q(\beta):= div_{\Sigma} (\beta(\nu, \cdot)) -\frac{1}{2}tr_{\Sigma}(\nabla^M_{\nu}\beta)\  \text{ is not identically }0 \text{ on }\Sigma    \] 
     The goal of this section is to prove the following 
     \begin{Thm} \label{Thm_Ass Jac, Main thm}
      Let $\Sigma, \beta$ be as above, $\{g_j\}$ be a family of Riemannian metric. Suppose either $g_j\equiv g$ for all $j$, or $\exists\ c_j\searrow 0$ and $\beta_j\to \beta$ in $C^4$ such that $g_j = g+c_j\beta_j$. 

      Let $V_j\neq |\Sigma|$ be a stationary integral varifold in $(M, g_j)$, $V_j\to |\Sigma|$. Suppose for every $p\in Sing(\Sigma)$, one of the following holds
      \begin{enumerate}
       \item[(A)] the tangent cone $C_p$ of $\Sigma$ at $p$ is area-minimizing at least in one side, see \cite{Lin87, LiuZH19_OneSided}.
       \item[(B)] there exists some neighborhood $U_p\subset M$ of $p$ such that $V_j\llcorner U_p$ are stable minimal hypersurfaces (possibly with multiplicity) in $(U_p, g_j)$, $j>>1$.
      \end{enumerate}
      Then after passing to a subsequence, there exists a (generalized) associated Jacobi field $\phi \in \scB(\Sigma) \bigoplus \RR_{L_{\Sigma}} \langle Sing(\Sigma) \rangle$ to $\{(V_j, g_j)\}_{j\geq 1}$; Moreover, $\phi$ satisfies $L_{\Sigma}\phi = 0$ provided $g_j = g$; Or $ L_{\Sigma}\phi = cq(\beta)$ for some $c\geq 0$ provided $g_j = g+ c_j\beta_j$.
     \end{Thm}
     
     \begin{Rem}
     \begin{enumerate}
      \item[(1)] The assumption (B) is always satisfied when $n=7$ and $V_j$ are locally stable minimal hypersurfaces with the same index as $\Sigma$. Hence, the theorem applies in this case.
      \item[(2)] In section \ref{Subsec, Ass Jac_Finiteness of ass Jac fields} we study in more details the possible asymptotic rate of $\phi$ above near each $p\in Sing(\Sigma)$. Roughly speaking, they are modeled on the asymptotic rate of stable minimal hypersurface in $\RR^{n+1}$ near infinity to the tangent cone $C_p$. See corollary \ref{Cor_Ass Jac, Max growth of ass Jac => smooth}.
     \end{enumerate}
     \end{Rem}

%%%%%%%%%%%%%%%%%%%%%%%%%%%%%%%%%%%%%%%%%%%%%%%%%%%%%%%%%%%%%%%%%%%%%%%%%%%%
%%%%%%%%%%%%%%%%%%%%%%%%%%%%%%%%%%%%%%%%%%%%%%%%%%%%%%%%%%%%%%%%%%%%%%%%%%%%
%%%%%%%%%%%%%%%%%%%%%%%%%%%%%%%%%%%%%%%%%%%%%%%%%%%%%%%%%%%%%%%%%%%%%%%%%%%%
%%%%%%%%%%%%%%%%%%%%%%%%%%%%%%%%%%%%%%%%%%%%%%%%%%%%%%%%%%%%%%%%%%%%%%%%%%%%
%%%%%% Subsection: Asymptotics of minimal hypersurfaces near infinity %%%%%%
%%%%%%%%%%%%%%%%%%%%%%%%%%%%%%%%%%%%%%%%%%%%%%%%%%%%%%%%%%%%%%%%%%%%%%%%%%%%
%%%%%%%%%%%%%%%%%%%%%%%%%%%%%%%%%%%%%%%%%%%%%%%%%%%%%%%%%%%%%%%%%%%%%%%%%%%%
%%%%%%%%%%%%%%%%%%%%%%%%%%%%%%%%%%%%%%%%%%%%%%%%%%%%%%%%%%%%%%%%%%%%%%%%%%%%
%%%%%%%%%%%%%%%%%%%%%%%%%%%%%%%%%%%%%%%%%%%%%%%%%%%%%%%%%%%%%%%%%%%%%%%%%%%%

     \subsection{Asymptotics of minimal hypersurfaces near infinity} \label{Subsec, Ass Jac_Asymp of min surf near infty}
      Let $C\subset \RR^{n+1}$ be a regular hypercone (not necessarily stable) throughout this subsection, with cross section $S$ and normal field $\nu$. The notions of geometric quantities on $C$ and $S$ will be the same as section \ref{Subsec, Geom of cone}. Recall that $E_{\pm}$ are the two connected components of $\RR^{n+1}\setminus C$, with $\nu$ pointing into $E_+$.
      
      Suppose that $V$ is an integral varifold in $\RR^{n+1}\setminus \BB_1^{n+1}$ (not necessarily closed), with finite density at infinity, i.e. \[
        \limsup_{R\to +\infty}\frac{1}{R^n}\|V\|(\BB_R^{n+1}\setminus \BB_1^{n+1}) < +\infty     \]
      \begin{Def}
       Call $V$ \textbf{asymptotic to} C near $\infty$, if there's a function $h\in C^2(C)$ and $R_0 > 1$ such that 
       \begin{align}   
        \frac{1}{R}|h|(R,\cdot) + |\nabla_C h|(R, \cdot) + R|\nabla_C^2 h|(R, \cdot) \to 0\ \ \ \text{ in }C^0(S)\text{ as }R\to +\infty   \label{Ass Jac, Asymp graph tend to 0}
       \end{align}
       and that 
       \begin{align}
        |graph_C(h)|\llcorner (\RR^{n+1}\setminus \BB^{n+1}_{R_0}) = V\llcorner (\RR^{n+1} \setminus \BB^{n+1}_{R_0})  \label{Ass Jac, graphical near infty of asymp min surf}
       \end{align}
       For such $V$, define \[
        \cA\cR_{\infty}(V) = \cA\cR_{\infty}(V; C) := \inf\{\sigma: \limsup_{R\to +\infty} R^{-\sigma}h(R, \cdot) = 0\ \text{ in }C^0(S) \}     \]
       Called the \textbf{asymptotic rate} of $V$ to $C$ at infinity. We use the convention that $\inf \emptyset := +\infty$ and $\cA\cR_\infty(|C|; C) = -\infty$.
       
       If $\Sigma\subset \RR^{n+1}\setminus \BB_1^{n+1}$ be a hypersurface, denote $\cA\cR_{\infty}(\Sigma):=\cA\cR_{\infty}(|\Sigma|)$.
      \end{Def}
      For stationary integral varifolds with finite density at infinity, by Allard regularity \cite{Allard72, Simon83_GMT}, being asymptotic to $C$ is equivalent to having $|C|$ to be the tangent varifold at $\infty$. Clearly, by (\ref{Ass Jac, Asymp graph tend to 0}), $\cA\cR_{\infty}(V) = \cA\cR_\infty((\eta_{0,R})_{\sharp}V ) \leq 1$, $\forall R>0$.  \\

      The goal of this subsection is to prove the following asymptotic rate estimate of minimal hypersurface towards it's tangent cone at infinity.
      \begin{Thm} \label{Thm_Ass Jac, Asymp Thm of ext. min. graph} 
       Suppose $V\in \cI\cV_n(\RR^{n+1})$ be a stationary integral varifold asymptotic to a regular minimal hypercone $C$ near infinity, $V \neq |C|$. Then
       \begin{enumerate}
        \item[(1)] If $C$ is one-sided minimizing (resp. strictly minimizing), then $\cA\cR_{\infty}(V) \geq \gamma_1^- (\text{resp. }\gamma_1^+)$.
         Moreover, if $C$ is area minimizing and $\cA\cR_{\infty}(\Sigma) = \gamma_1^{\pm}$, then $\Sigma$ has no singularity and lies on one side of $C$.
        \item[(2)] If $V=|\Sigma|$ for some stable minimal hypersurface $\Sigma\subset \RR^{n+1}$, then $\cA\cR_{\infty}(\Sigma) \geq \gamma_1^-$. 
         If further, $\cA\cR_{\infty}(\Sigma) = \gamma_1^{\pm}$, then $\Sigma$ has no singularity and lies on one-side of $C$, inside which $\Sigma$ and $C$ are minimizing.
       \end{enumerate}
      \end{Thm}
      We start with some preparation.      
      
      \begin{Lem} \label{Lem_Ass Jac, asymp char} 
       Suppose $\Sigma \subset \RR^{n+1}\setminus\BB^{n+1}_1$ be a minimal hypersurface asymptotic to $C$ near infinity. Then
       \begin{enumerate}
        \item[(1)] $\cA\cR_{\infty}(\Sigma) \in \Gamma_C\cup \{-\infty\}$, where $\Gamma_C$ is defined below (\ref{Pre, power of cone Jac}).
        \item[(2)] If $\gamma := \cA\cR_{\infty}(\Sigma) \in (-\infty, 1)$, $h$ be the graphical function as in (\ref{Ass Jac, graphical near infty of asymp min surf}), then \[
          h(r,\omega) = r^{\gamma}(c + c'\log r)w_j(\omega) + O_2(r^{\gamma - \epsilon})     \]
         for some $(c, c')\neq (0, 0)$, $c'=0$ unless $\gamma = \gamma_1^+ = \gamma_1^-$ (in which case $C$ is not strictly stable); $\epsilon = \epsilon(\Sigma) > 0$; $w_j(\omega)$ is a unit eigenfunction of $\cL_S$ as in section \ref{Subsec, Geom of cone}, and $\gamma \in \{\gamma_j^{\pm}\}$.
       \end{enumerate}
      \end{Lem}
      This characterization is well-known, at least when $\gamma = \gamma_1^{\pm}$, in \cite[1.9]{HardtSimon85}. For sake of completeness, we sketch the proof here. In what follows, every constant depends on the cone $C$.
      \begin{proof}
       Since $1\in \Gamma_C$, (1) follows from (2). 
       To prove (2), since $\Sigma$ is minimal, $h$ satisfies the minimal surface equation \[
        \scM_C h = -L_C h + \scR_C h = 0     \]
       where by \cite[(2.2)]{CaffHardtSimon84}, 
       \begin{align*}
        & \scR_C h\ (r, \omega) := N(r\omega, h/r, \nabla h)\cdot \nabla^2 h + r^{-1}P(r\omega, h/r,\nabla h ) \\
        & |P(x, z, p)|\leq C(|z|+|p|)^2 \\
        & |N(x, z, p)|+ |P_z(x, z, p)| + |P_p(x, z, p)|\leq C(|z|+|p|) \\
        r(& |N_x(x, z, p)+P_x(x, z, p)|) + |N_z(x, z, p)| + |N_p(x, z, p)| \leq C
       \end{align*}
       By elliptic estimate, for $R>>2R_0$, \[
        \|\nabla_C h\|_{C^0(A_{R, 2R})} + R\|\nabla_C^2 h\|_{C^0(A_{R, 2R})} \leq \frac{C}{R}\|h\|_{C^0(A_{R/2, 4R})}     \]
       Hence by definition, for any $\delta>0$, since $R^{-\gamma - \delta}h(R, \cdot)\to 0$ as $R\to +\infty$, we have\[
        |\scR_C h|(R, \omega) \leq C_{\delta}R^{\gamma-2 + (\gamma - 1 + 2\delta)} \ \ \ \ \forall R>> 2R_0    \]
       Fix $\delta>0$ such that $\gamma - 1 + 3\delta < 0 $ and $\sigma := \gamma + (\gamma -1 +3\delta) \notin \Gamma_C$. By lemma \ref{Lem_Pre, Ext sol of Jac, prescribed asymp}, $\exists u \in C^{\infty}(C\setminus B_1)$ solving $L_C u = \scR_C h$ and \[
        \|u(R, \cdot)\|_{L^2(S)}\leq CR^{\sigma}    \]
       But since $L_C(h-u) = 0$ and $\forall \delta'>0$, $\limsup_{R\to \infty}R^{-\gamma + \delta'}\|h(R, \cdot)\|_{C^0(S)}>0$, by (\ref{Pre, hom Sol Jacob}), $h+u = r^{\gamma}(c+ c'\log r) + l.o.t.$ and so does $h$.
      \end{proof}

      The following lemma is an analogue of \cite[Theorem 2.1]{HardtSimon85} for stable minimal cones.
      \begin{Lem} \label{Lem_Ass Jac, Reg of infty one-side perturb} 
       Suppose $\Sigma\subset\RR^{n+1}$ be a connected stable minimal hypersurface with  finite density at infinity; Also suppose that $\Sigma$ is contained in $E_+$ outside a large ball, i.e. $\exists\ R_0>1$ such that $\Sigma\subset \BB^{n+1}_{R_0} \cup E_+$. Then 
       \begin{enumerate}
        \item[(1)] $Sing(\Sigma)=\emptyset$, $\Sigma$ is contained entirely in $E_+$ and area-minimizing in $E_+$. 
        \item[(2)] $C$ is area-minimizing in $Clos(E_+)$.
        \item[(3)] $\{r\cdot \Sigma\}_{r>0}$ foliates $E_+$.
        \item[(4)] Any stationary integral $n$-varifold $V$ with $spt V\subset E_+$ and having the same density at $\infty$ as $C$ is a rescaling of $|\Sigma|$.
       \end{enumerate}
      \end{Lem}
      \begin{Rem}
      \begin{enumerate}
       \item[(1)] When $C$ is a priori assumed to be area-minimizing, (1), (3), (4) of lemma \ref{Lem_Ass Jac, Reg of infty one-side perturb} is a corollary of \cite[Theorem 2.1]{HardtSimon85} by a barrier argument. However, in the proof we present here, (1)(3) are established first based on the linear analysis on $\Sigma$ in section \ref{Subsec, Linear ana of Jac oper} and \ref{Sec, Linear Ana}, and the one-sided minimizing property of $C$ is part of the conclusion. With (1) and (2), the proof (4) is similar to \cite{HardtSimon85}. For sake of completeness, we also sketch the idea of (4).
       \item[(2)] By solving Plateau problem with one-sided perturbed boundary, [H-S] shows that such $\Sigma$ in lemma \ref{Lem_Ass Jac, Reg of infty one-side perturb} do exist if $C$ is area-minimizing. The argument also works for one-sided area-minimizers. 
       \item[(3)] We assert here that the stability assumption on $\Sigma$ can't be dropped. See (P1.2) for further discussion in section \ref{Sec, Apps and Discussion}.
      \end{enumerate}
      \end{Rem}
      
      \begin{proof}
       Recall that we use the notion $B_R := \BB_R^{n+1}\cap \Sigma$. Since $\Sigma$ is stable with finite density at infinity, by \cite{SchoenSimon81}, any tangent varifold $V_{\infty}$ of $|\Sigma|$ at infinity is a stable minimal cone. And since $\Sigma$ is contained in $E_+$ outside a large ball, $spt(V_{\infty})\subset Clos(E_+)$. Then by $0\in sptV_{\infty}\cap C$ and strong maximum principle \cite{SolomonWhite89_StrongMax, Ilmanen96}, $spt(V_{\infty})=C$; Since $C$ is regular, by \cite{SchoenSimon81} and \cite{CaoShenZhu97_Stab_One_End}, $V = |C|$, i.e. $\Sigma$ asymptotic to $C$ near infinity. Take $R_0>1$ and $h\in C^2(C)$ such that (\ref{Ass Jac, Asymp graph tend to 0}) and (\ref{Ass Jac, graphical near infty of asymp min surf}) holds.

       Since $\Sigma$ is contained in $E_+$ near $\infty$, by a blow down argument and combining lemma \ref{Lem_Pre, pos Jacob field} and \ref{Lem_Ass Jac, asymp char}, see also \cite[1.9]{HardtSimon85}, when $R\to +\infty$
       \begin{align}
        \begin{cases}
         \text{ either }\ & h(R,\omega) = (c_1+c_2\log R)R^{\gamma_1^+}w_1(\omega) + O_2(R^{\gamma_1^+ -\epsilon}) \\
         \text{ or }\ & h(R,\omega) = c_1 R^{\gamma_1^-}w_1(\omega) + O_2(R^{\gamma_1^- -\epsilon})
        \end{cases} \label{Ass Jac, Asymp near infty of one-side perturb}
       \end{align}
       for some $\epsilon>0$; $c_1+c_2\log R> 0$ for $R>R_0$; $c_2 = 0$ unless $\gamma_1^+ = \gamma_1^- = -(n-2)/2$.
       
       Let $\nu_{\Sigma}$ be the normal field of $\Sigma$, pointing away from $C$. Consider the function $\psi = (\nabla_{\RR^{n+1}}|x|^2)\cdot \nu_{\Sigma}$ on $\Sigma$. Note that $|\nabla \psi|(x) \leq |x|\cdot|A_{\Sigma}|_x $ and $L_{\Sigma}\psi = 0$; And by \cite{CheegerNaber13_Stra_Min_Surf}, $|A_{\Sigma}| \in L^7_{weak}(B_R)$, $|\nabla A_{\Sigma}|\in L^{7/2}_{weak}(B_R)$, thus $\psi \in W^{1,2}(B_R)$ for each $R>0$.
       By (\ref{Ass Jac, Asymp near infty of one-side perturb}), $\psi>0$ outside some large ball $B_{R_1}$, $R_1>R_0$; Hence by stability of $\Sigma$ and weak maximum principle, $\psi>0$ on the whole $\Sigma$. Hence by lemma \ref{Lem_Pre, Diverge of pos Jac field}, $\psi(x)\to \infty$ as $x\to Sing(\Sigma)$. 
       But notice that by its definition, $\psi$ is bounded on every $B_R$. This means $Sing(\Sigma) = \emptyset$.  \\

       To show that $\Sigma\subset E_+$, let $\Omega_+ := E_+\cap \SSp^n$ and consider the smooth map\[
        p:\Sigma\setminus \{0^{n+1}\} \to \SSp^n\ \ \ x\mapsto \frac{x}{|x|}     \]
       By $\psi> 0$, $p$ is a diffeomorphism locally; By (\ref{Ass Jac, Asymp near infty of one-side perturb}), $p$ is diffeomorphism between $\Sigma\setminus B_{R_1}$ and a collar neighborhood of $S = \partial \Omega_+$ in $\Omega_+$. By the following topological lemma \ref{Lem_Ass Jac, Top lemma of diff}, $p$ is a glocal diffeomorphism onto $\Omega_+$. Hence, $\Sigma\subset E_+$ and $\{r\cdot \Sigma\}_{r>0}$ foliates $E_+$.\\
       
       To show that $\Sigma$ and $C$ are area-minimizing in $Clos(E_+)$, it suffices to observe that $\{r\cdot \Sigma\}_{r>0}$ determines a calibration $\Theta := \iota_{\nu_+}dVol_{\RR^{n+1}}$ on $Clos(E_+)\setminus \{0\}$, where $\nu_+$ be the normal vector field of $\{r\cdot \Sigma\}_{r>0}$, $C^1$ extended to $Clos(E_+)\setminus \{0\}$. Hence, for any $R>0$, and any integral cycle $T \in \cZ_n(\RR^{n+1})$,
       \begin{align*}
        \textbf{M}_{\BB^{n+1}_R}(T) & \geq \langle T, \Theta \rangle = \langle [C], \Theta \rangle = \textbf{M}_{\BB^{n+1}_R}([C])\ \ \ \text{ if }spt([C]-T)\subset \BB_R^{n+1} \\ 
        \textbf{M}_{\BB^{n+1}_R}(T) & \geq \langle T, \Theta \rangle = \langle [\Sigma], \Theta \rangle = \textbf{M}_{\BB^{n+1}_R}([\Sigma])\ \ \ \text{ if }spt([\Sigma]-T)\subset \BB_R^{n+1}
       \end{align*}

       We finally sketch the idea of (4). If $V$ is a stationary integral $n$-varifold with the same density at $\infty$ as $C$ and $sptV\subset E_+$, then by a similar argument as the first part of this proof, $|C|$ is the tangent varifold of $V$ at infinity, and hence $V$ could be parametrized as a graph over $C$ near infinity as in (\ref{Ass Jac, graphical near infty of asymp min surf}) and (\ref{Ass Jac, Asymp near infty of one-side perturb}). Let 
       \begin{align*}
        r_+ &:= \inf\{r: \text{for any }R>r, sptV\cap R\cdot \Sigma =\emptyset\} \\
        r_- &:= \sup\{r: \text{for any }R<r, sptV\cap R\cdot \Sigma =\emptyset\}
       \end{align*}
       By (\ref{Ass Jac, Asymp near infty of one-side perturb}) and strong maximum principle, $0< r_- = r_+ <\infty$ and $V = |r_+\cdot\Sigma|$.
      \end{proof}      
      
      \begin{Lem} \label{Lem_Ass Jac, Top lemma of diff}
       Let $N$ be a connected compact $n$ manifold with nonempty boundary, $M$ be a closed simply connected $n$ manifold, $p:N\to M$ be a local diffeomorphism onto its image, and restricted to a bijection near $\partial N$. If $M\setminus p(\partial N) = M_+\sqcup M_-$ has 2 connected component, with $p(N)\cap M_+ \neq \emptyset$, then $p$ is a diffeomorphism onto $M_+$. 
      \end{Lem}
      \begin{proof}
       Notice that the glued map $p\cup_{\partial N}id_{M_-}: N\cup_{\partial N} M_- \to M $ is a local homeomorphism between closed manifolds, hence a covering map. Since $N$ is connected and $M$ is simply connected, this is a bijection, so is $p$. 
      \end{proof}
      
      \begin{Lem} \label{Lem_Ass Jac, Pos Jac field on Sigma with AR bd}
       Let $\Sigma\subset \RR^{n+1}$ be a stable minimal hypersurface. Then the space of positive Jacobi field on $\Sigma$ is nonempty, i.e. \[
        \cG(\Sigma):= \{u\in C^{\infty}(\Sigma): L_{\Sigma}u = 0,\ u>0\} \neq \emptyset     \]
       And if $Sing(\Sigma)$ is bounded, then $\forall u\in \cG(\Sigma)$, we have $u(x)\to +\infty$ if $x\to x_{\infty}\in Sing(\Sigma)$. 
       
       Moreover, if $\Sigma$ is asymptotic to $C$ near infinity for some regular minimal hypercone $C$, assume WLOG (by a rescaling) $Sing(\Sigma)\subset B_1$, then for every $\epsilon>0$, \[
        u(x)\geq C(\Sigma, \epsilon, u)|x|^{\gamma_1^- - \epsilon}>0\ \ \ \ \forall x\in \Sigma\setminus B_2   \]
      \end{Lem}
      \begin{proof}
       To see $\cG(\Sigma)\neq \emptyset$, consider an exhaustion $\{\cU_i\}_{i\geq 1}$ of $\Sigma$ and $B_i:=B_{r_i}(p_i)\subset\subset \cU_i$, where $p_i\to \infty$, $r_i\in (0, 1)$. Let $u_i\in W_0^{1,2}(\cU_i)$ be the solution of $L_{\Sigma} u_i = 0$ on $\cU_i\setminus B_i$ and $u_i = 1$ on $B_i$. By the stability of $\Sigma$ and maximum principle, $u_i> 0$ on $\cU_i$; Also let $q\in \Sigma$ be fixed. Then by Harnack inequality, up to a subsequence, \[
        u_i/u_i(q) \to u_{\infty}\ \ \ \text{ in }C_{loc}^{\infty}(\Sigma)     \]
       for some positive Jacobi field $u_{\infty}\in \cG(\Sigma)$.
       
       The unboundedness of any $u\in \cG(\Sigma)$ towards singular set follows from lemma \ref{Lem_Pre, Diverge of pos Jac field}.
       
       To deduce the lower bound of $u$ near infinity, suppose WLOG that $h\in C^{\infty}(C)$ and $R_0 = 1$ such that (\ref{Ass Jac, Asymp graph tend to 0}) and (\ref{Ass Jac, graphical near infty of asymp min surf}) holds for $V = |\Sigma|$. Parametrize $\Sigma\setminus B_1$ by \[
        C \to \Sigma\ \ \ x\mapsto x+ h(x)\nu(x)     \]
       By (\ref{Ass Jac, Asymp graph tend to 0}), under this parametrization and the normal coordinates of $C$, $L_{\Sigma} = a_0\cdot\partial_i(a^{ij}\partial_j u) + |A_{\Sigma}|^2 u $, where \[
        \sup_{\omega\in S}|a_0 - 1|(R, \omega) + |a^{ij}-\delta^{ij}|(R,\omega) + R^2\big| |A_{\Sigma}|^2- |A_C|^2\big|(R,\omega) \to 0\ \ \ \text{ as } R\to \infty   \]
       Consider for each $R>2$, $c_R:=\inf_{\partial B_R}u$ and $v_R\in C^{\infty}(C\setminus \BB_{2/R})$ defined by \[
        v_R(r, \omega):= u(rR, \omega)/c_R     \]
       By elliptic estimate, up to subsequences, $v_R \to v_{\infty}$ in $C^0_{loc}(C)$ for some positive Jacobi fields $v_{\infty}$ on $C$ with $\inf_{\SSp^n \cap C} v_{\infty} = 1$. By lemma \ref{Lem_Pre, pos Jacob field}, 
       \begin{align}
        v_{\infty}(r, \omega) \geq C(S)r^{\gamma_1^-}  \label{Ass Jac, decay bd on pos Jac fields}       
       \end{align}
       The lower bound of decaying rate for $u$ then follows from (\ref{Ass Jac, decay bd on pos Jac fields}) and an approximation argument.
      \end{proof}
      
      \begin{proof}[Proof of theorem \ref{Thm_Ass Jac, Asymp Thm of ext. min. graph} (1).] 
      Let $\Sigma_+ = \partial P_+$ be the leaf of Hardt-Simon foliation in theorem \ref{Thm_Pre, H-S foliation} lying on the side $E$ where $C$ is minimizing (resp. strictly minimizing), $P_+\subset E$; Let $h_+$, $h$ be the garphical functions of $\Sigma_+$ and $V$ correspondingly over $C$. Suppose otherwise that $\cA\cR_{\infty}(V)<\gamma_1^-$ (resp. $\gamma_1^+$). Then by lemma \ref{Lem_Ass Jac, asymp char}, (\ref{Pre, asymp of one-sided, strict}) and (\ref{Pre, asymp of one-sided, nonstrict}), $h_+>h$ near infinity. Hence for $R>>1$, $R\cdot P_+ \cap spt(V) = \emptyset$. By considering \[
       \inf \{R: s\cdot P_+ \cap spt(V) = \emptyset, \ \forall s>R\}   \]
      and using the strong maximum principle \cite{SolomonWhite89_StrongMax}, we see that $spt(V)\subset \RR^{n+1}\setminus Clos(E)$, i.e. lying on the other side of $C$. Then by lemma \ref{Lem_Ass Jac, Reg of infty one-side perturb} and theorem \ref{Thm_Pre, H-S foliation}, $V$ is a rescaling of Hardt-Simon foliation on the compliment of $Clos(E)$ and thus $\cA\cR_{\infty}(V) \geq \gamma_1^-$ (resp. $=\gamma_1^+$), from which we get a contradiction.
      
      If $\cA\cR_{\infty}(V) \in \{\gamma_1^{\pm}\}$, then by lemma \ref{Lem_Ass Jac, asymp char}, $spt(V)$ lies on one side of $C$ near infinity. By using the Hardt-Simon foliation on the other side and strong maximum principle as above, one see that $spt(V)$ lies entirely on one side of $C$, hence becomes a leaf of Hardt-Simon foliation by theorem \ref{Thm_Pre, H-S foliation}. \\
      
\noindent  \textit{Proof of theorem \ref{Thm_Ass Jac, Asymp Thm of ext. min. graph} (2)} Consider $\psi := \nabla_{\RR^{n+1}}|x|^2\cdot \nu_{\Sigma}$. By lemma \ref{Lem_Ass Jac, asymp char}, if $\cA\cR_{\infty}(\Sigma)<\gamma_1^-$, then for each $\epsilon>0$, $\exists\ R(\Sigma, \epsilon)>1$ such that \[
       |\psi(x)|\leq |x|^{\gamma_2^- + \epsilon} \ \ \ \ \forall |x|>R(\Sigma, \epsilon)    \] 
       On the other hand, let $v\in \cG(\Sigma)$ be a positive Jacobi field on $\Sigma$. By lemma \ref{Lem_Pre, Diverge of pos Jac field}, \ref{Lem_Ass Jac, Pos Jac field on Sigma with AR bd} and weak maximum principle, $\exists\ c>0$ such that $cv> \psi$ on $\Sigma$. Consider the smallest one among all such $c$ and using strong maximum principle, we conclude that $\psi = 0$ and thus $\Sigma = C$, which contradicts to our assumption. This proves the lower bound on $\cA\cR_{\infty}(\Sigma)$.
       
       When $\cA\cR_{\infty}(\Sigma)\in \{\gamma_1^{\pm}\}$, the conclusion follows from lemma \ref{Lem_Ass Jac, asymp char} and \ref{Lem_Ass Jac, Reg of infty one-side perturb}.       
      \end{proof}

%%%%%%%%%%%%%%%%%%%%%%%%%%%%%%%%%%%%%%%%%%%%%%%%%%%%%%%%%%%%%%%%%%%%%%%%%%%%
%%%%%%%%%%%%%%%%%%%%%%%%%%%%%%%%%%%%%%%%%%%%%%%%%%%%%%%%%%%%%%%%%%%%%%%%%%%%
%%%%%%%%%%%%%%%%%%%%%%%%%%%%%%%%%%%%%%%%%%%%%%%%%%%%%%%%%%%%%%%%%%%%%%%%%%%%
%%%%%%%%%%%%%%%%%%%%%%%%%%%%%%%%%%%%%%%%%%%%%%%%%%%%%%%%%%%%%%%%%%%%%%%%%%%%
%%%%%%%%%%%% Subsection: Finiteness of associated Jacobi fields %%%%%%%%%%%%
%%%%%%%%%%%%%%%%%%%%%%%%%%%%%%%%%%%%%%%%%%%%%%%%%%%%%%%%%%%%%%%%%%%%%%%%%%%%
%%%%%%%%%%%%%%%%%%%%%%%%%%%%%%%%%%%%%%%%%%%%%%%%%%%%%%%%%%%%%%%%%%%%%%%%%%%%
%%%%%%%%%%%%%%%%%%%%%%%%%%%%%%%%%%%%%%%%%%%%%%%%%%%%%%%%%%%%%%%%%%%%%%%%%%%%
%%%%%%%%%%%%%%%%%%%%%%%%%%%%%%%%%%%%%%%%%%%%%%%%%%%%%%%%%%%%%%%%%%%%%%%%%%%%
      
      \subsection{Finiteness of associated Jacobi fields} \label{Subsec, Ass Jac_Finiteness of ass Jac fields}
      
       The first goal of this subsection is to prove theorem \ref{Thm_Ass Jac, Main thm}. The key is a growth rate estimate near each singularity. We begin with some notations. Recall $(\Sigma, \nu) \subset (M, g)$ is a two sided, locally stable minimal hypersurfaces with strongly isolated singularities. The geometric quantities of $\Sigma$ are defined in section \ref{Subsec, Min surface near sing}. 
       
       Let $\tau_{\Sigma}(1)$ be defined below (\ref{Pre, conic rad of sing}); $0<\tau<\tau_{\Sigma}(1)/2$. When working in $B^M_{2\tau}(p)$ for some singular point $p$, we identify $\Sigma \cap B_{2\tau} ^M(p) \hookrightarrow T_pM$ as in section \ref{Subsec, Min surface near sing} and parametrize it by the tangent cone $C_p$ of $\Sigma$ at $p$ as in (\ref{Pre, param of asymp conic at 0}). 
       Denote $\Gamma_{C_p}=\{\gamma_k^{\pm}(C_p)\}_{k\geq 1}$ to be the asymptotic spectrum defined in (\ref{Pre, power of cone Jac}) for $C = C_p$. We may simply write $\gamma_k^{\pm}(C_p)$ as $\gamma_k^{\pm}$ if there's no confusion. \\
       
       We shall first prove a general local dichotomy near every fixed $p\in Sing(\Sigma)$. Recall for a domain $\Omega\subset \Sigma$ and $k\geq 0$, $\|\cdot \|^*_{k;\Omega}$ norm is introduced in (\ref{Pre_C^k_* norm}) and will be used frequently in this subsection.
       
       Consider in general $\{g_j\}_{j\geq 1}$ a family of smooth metrics $C^4$ converges to $g$ on $M$; $\{V_j\}_{j\geq 1}$ are stationary integral varifolds in $(B^M_{2\tau}(p), g_j)$ which converges to $|\Sigma|\llcorner B_{2\tau}^M(p)$ in $B_{2\tau}^M(p)$ in varifold sense. Suppose either of the following holds,
       \begin{enumerate}
       \item[(A')] $C_p$ is area-minimizing at least in one side.
       \item[(B')] $spt(V_j)$ are stable minimal hypersurfaces in $B^M_{2\tau}(p)$. 
       \end{enumerate}
       For any $\delta\in (0,1)$, call $\{(\eta_{p, \tau_j})_{\sharp} V_j\}$  \textbf{$\delta$-minimal blow up sequence} of $\{V_j\}_{j\geq 1}$ at $p$ 
       if \[
         \tau_j:= \inf\{t>0 : V_j\llcorner A^M_{t, \tau}(p) = |graph_{\Sigma}(v_j)|\llcorner A^M_{t, \tau} \text{ for some }\|v_j\|^*_{2; \cA_{t, \tau}}\leq \delta\}>0   \]
       Note that by Allard compactness \cite{Allard72, Simon83_GMT} and the minimality of $\tau_j$, for every $0<\delta<<1$, the $\delta$-minimal blow up sequence sub-converges to some stationary integral varifold $|C_p|\neq \tilde{V}_{\infty}\subset \RR^{n+1}$ which is a exterior graph over $C_p$. Hence by (A'), (B') and theorem \ref{Thm_Ass Jac, Asymp Thm of ext. min. graph},
       \begin{align}
        \cA\cR_{\infty}(\tilde{V}_{\infty}; C_p)\geq \gamma_1^- (C_p)  \label{Ass Jac_Asymp Lower Bd for delta-minimal blow up lim}
       \end{align}
       Call such $\tilde{V}_{\infty}$ a \textbf{$\delta$-minimal blow-up limit} of $\{V_j\}_{j\geq 1}$.

       For every $\sigma\in (-\infty, 1)\setminus \Gamma_{C_p}$ fixed, let $K= K_0(C_p, \sigma)$ be determined in lemma \ref{Lem_Pre, growth rate mon for Jac}; Let $\delta_0, N$ be determined in corollary \ref{Cor_Pre, quant growth rate mon for Jac}, both depend on $C_p, \sigma, K$. 
       For each $Q\in \RR_+$, $\phi\in L^2(C_p)$ and $r>0$, define 
       \begin{align}
        \hat{J}^{\sigma}_{p, \phi}(r; Q):= \sup\big\{Q, \int_{A^{C_p}_{r, Kr}(p)} \phi^2(x)\cdot |x|^{-n-2\sigma}\ dvol_{C_p}(x) \big\}
       \end{align}
       
       \begin{Lem} \label{Lem_Ass Jac, Dichotomy Growth Rate Bd}
        Let $\Sigma, M, g, \tau, \sigma, K, N$ be described above; $\kappa>0$. Then there exists $\delta_3>0$ and $s_0\in (0, \tau)$ such that,
        if $\{(V_j, g_j)\}_{j\geq 1}$ satisfying either (A') or (B') are given as above; then one of the following holds,
        \begin{enumerate}
        \item[(i)] $\exists\ \delta_3$-minimal blow-up limit $\tilde{V}_\infty$ of $\{V_j\}_{j\geq 1}$ such that $\cA\cR_{\infty}(\tilde{V}_\infty; C_p)\leq \sigma$;
        \item[(ii)] If $\tau_j\to 0_+$ and $v_j\in C^2(\cB_{2\tau}(p))$ be the $\delta_3$-graphical function of $V_j$ over $\cA_{\tau_j, \tau}$, i.e $v_j$ satisfies $\|v_j\|^*_{2, \cA_{\tau_j/2, 2\tau}(p)}\leq \delta_3$ and $V_j\llcorner A^M_{\tau_j, \tau} = |graph_{\Sigma}(v_j)|\llcorner A^M_{\tau_j, \tau}$, then for $j>> 1$, \[
         \hat{J}_j(l) \leq \sup_{1\leq i\leq l-1}\hat{J}_j(i) \ \ \ \text{ for all integers } N+1\leq l\leq (\log s_0 - \log \tau_j)/\log K    \]
         where $\hat{J}_j(l):= \hat{J}^{\sigma}_{p, v_j}(s_0K^{-l}; \kappa\|g_j - g\|_{C^3(M)})$.
        \end{enumerate}
       \end{Lem}
       \begin{proof}
        First observe that by lemma \ref{Lem_Pre, geom of asymp at 0} and \ref{Lem_Pre, Vol element and mean curv of graph}, $\exists s_1\in (0, \tau)$ and $\delta_3>0$ such that if $\cU\subset \cB_{s_1}(p)$ is an open subset, $\|g'-g\|_{C^3(M)}\leq \delta_3$ and $\|v\|^*_{2, \cU}\leq \delta_3$ such that $graph_{\Sigma}(v)$ is minimal in $(M, g')$ in its interior, then $v$ satisfies the equation \[
         -div_C (\nabla_C v + b_0) + |A_C|^2v + b_1 = 0    \]
        on $\cU$, where \[
         |b_0(x, z, p)| + |x|\cdot|b_1(x, z, p)| \leq \delta_0(|z|/|x|+ |p|+ \|g'-g\|_{C^2(M)})/2 \ \ \text{ on }\cU\times \RR\times T_*\Sigma  \]
        and $\delta_0$ is specified in corollary \ref{Cor_Pre, quant growth rate mon for Jac}. Moreover, $v^{(r)}(y):= v(ry)/r$ also satisfies the equation of the same form with the same estimate on its domain. Fix the choice of $\delta_3$ from now on.
        Also notice that 
        \begin{align}
         \hat{J}^{\sigma}_{p, v}(s; Q) = r^{2(1-\sigma)}\hat{J}^{\sigma}_{p, v^{(r)}}(s/r; r^{2(\sigma - 1)}Q) \label{Ass Jac, J(r) change under rescaling}        
        \end{align}
        Take $s_0\in (0, s_1)$ such that $s_0^{2(\sigma-1)}\cdot\kappa \geq 1$. 
        
        Suppose (ii) fails, then after passing to a subsequence, $\exists\ l_j\geq N+1$, $\tau'_j:= s_0\cdot K^{-l_j}\in (\tau_j, s_0)$ such that 
        \begin{align}
         \hat{J}^{\sigma}_{p, v_j} (\tau'_j; \kappa\|g_j-g\|_{C^3(M)}) > \sup_{1\leq i\leq l_j-1}\hat{J}^{\sigma}_{p, v_j} (\tau'_j K^i; \kappa\|g_j-g\|_{C^3(M)})  \label{Ass Jac, hat(J)^sigma strictly increase at l_j}        
        \end{align}          
        Then by inductively considering $v_j^{(\tau'_jK^{-d})}$, $d\geq 1$, by the choice of $s_0$, (\ref{Ass Jac, J(r) change under rescaling}) and corollary \ref{Cor_Pre, quant growth rate mon for Jac}, 
        \begin{align}
         \hat{J}^{\sigma}_{p, v_j} (s_0K^{-l}; 0)\geq 
         \sup_{1\leq i\leq l-1}\hat{J}^{\sigma}_{p, v_j} (s_0K^i; \kappa\|g_j-g\|_{C^3(M)})\ \ \ \forall l_j\leq l\leq (\log s_0 -\log \tau_j)/\log K   \label{Ass Jac, J^sigma growth mon in l}        
        \end{align}
        WLOG $\tau_j\to 0_+$ are such that $(\eta_{p, \tau_j})_{\sharp}V_j$ is the $\delta_3$-minimal blow up sequence for $\{V_j\}_{j\geq 1}$, note that it also follows from (\ref{Ass Jac, J^sigma growth mon in l}) and (\ref{Ass Jac, hat(J)^sigma strictly increase at l_j}) that $\tau_j>0$. Then by (\ref{Ass Jac, J^sigma growth mon in l}), the varifold limit $\tilde{V}_{\infty}$ of this blow up sequence satisfies $\cA\cR_{\infty}(\tilde{V}_{\infty})\leq \sigma$. This completes the proof.         
       \end{proof}

       \begin{proof}[Proof of theorem \ref{Thm_Ass Jac, Main thm}.]
        Recall that $V_j$ is stationary in $(G, g_j)$ and $V_j \to |\Sigma|$, $V_j\neq |\Sigma|$. 

        We consider first the simpler setting where $g_j \equiv g$. By assumption (A) and (B) in theorem \ref{Thm_Ass Jac, Main thm}, suppose $\tau\in (0, \tau_{\Sigma}(1)/2)$ such that for each $p\in Sing(\Sigma)$, either $C_p$ is minimizing in one side, or $V_j$ is a stable minimal hypersurface in $B^M_{\tau}(p)$. Let $\sigma_p \in (\gamma_2^-(C_p), \gamma_1^-(C_p))\setminus \Gamma_{C_p}$; $K_p\geq 2$, $N_p\in \NN$ be specified in corollary \ref{Cor_Pre, quant growth rate mon for Jac} as above; $s_0(p)\in (0,\tau)$ and $\delta_3(p)$ be specified in lemma \ref{Lem_Ass Jac, Dichotomy Growth Rate Bd} for each $p\in Sing(\Sigma)$, where $\kappa = 1$ and $g_j=g$; 
        And let $\delta_4:= \inf_{p\in Sing(\Sigma)}\delta_3(p)$; $\bar{s}_0:=\inf_{p\in Sing(\Sigma)} s_0(p)\cdot K_p^{-N_p-1}$.

        Note that by (\ref{Ass Jac_Asymp Lower Bd for delta-minimal blow up lim}) and the choice of $\sigma_p$, only (ii) in lemma \ref{Lem_Ass Jac, Dichotomy Growth Rate Bd} could happen for every singularity $p$.
        
        Since $V_j\to |\Sigma|$, by Allard regularity, $\exists s_j\to 0_+$ and $v_j\in C^2(\Sigma)$ such that $\|v_j\|^*_{2,\Sigma\setminus B_{s_j}(Sing(\Sigma))}\leq \delta_4$ and that \[
         V_j\llcorner M\setminus B^M_{s_j}(Sing(\Sigma)) = |graph_{\Sigma}(v_j)|\llcorner M\setminus B^M_{s_j}(Sing(\Sigma))   \]
        Then by lemma \ref{Lem_Ass Jac, Dichotomy Growth Rate Bd} (ii), for each $p\in Sing(\Sigma)$,
        \begin{align}
         \int_{A_{r, K_pr}(p)} v_j^2 \rho^{-n-2\sigma_p}\ dvol_{\Sigma}(x) \leq C(\Sigma, M, g)\int_{A_{\bar{s}_0, \tau}(p)} v_j^2 \rho^{-n-2\sigma_p}\ dvol_{\Sigma}(x)\ \ \ \ \ \forall r\in (s_j, \bar{s}_0) \label{Ass Jac, growth rate bd for v_j}
        \end{align} 
        In particular, $\forall j>>1$ and $\forall \Omega\subset \Sigma\setminus B_{s_j}(Sing(\Sigma))$, we have 
        \begin{align}
         \int_{\Omega} v_j^2\ dx \leq C(\Sigma, M, g, \{\sigma_p\}_{p\in Sing(\Sigma)}, \Omega)\int_{\Sigma\setminus B_{\bar{s}_0}(Sing(\Sigma))} v_j^2\ dx  \label{Ass Jac, global bded by loc}        
        \end{align}
        The key point of this estimate is that the constant on RHS does \textbf{NOT} depend on $j$. 
        
        Take $a_j:= \|v_j\|_{L^2(\Sigma\setminus B_{\bar{s}_0})}>0$ for $j>>1$. Since $V_j \neq |\Sigma|$, by Allard regularity, unique continuation property of minimal surface equation and assumption (A) (B) on $V_j$, $a_j>0$ for $j>>1$.
        
        Take $\hat{v}_j:= v_j/a_j$. By (\ref{Ass Jac, global bded by loc}) and standard elliptic estimate, $\hat{v}_j \to v_{\infty}$ in $C^{\infty}_{loc}(\Sigma)$ and by the choice of $a_j$, $v_{\infty}$ is not identically $0$. Also by (\ref{Ass Jac, growth rate bd for v_j}) and the choice of $\sigma_p$, $\cA\cR_{p}(v_{\infty})\geq \gamma_1^-(C_p)$ for all $p\in Sing(\Sigma)$. Hence corollary \ref{Cor_Linear, Global Growth Bd function decomp} yields $v_{\infty}\in \scB(\Sigma)\oplus \RR_{L_{\Sigma}}\langle Sing(\Sigma)\rangle$.\\
        
        The proof for the case where $g_j = g+ c_j\beta_j$ as in the statement of theorem is similar but more delicate in choosing the renormalizing constant. Since $q(\beta)$ is not identically $0$, let $\Omega_0\subset \subset \Sigma$ such that $spt(q(\beta))\cap \Omega_0 \neq \emptyset$. Let $v_j\in C^2(\Omega_0)$ be the graphical function of $V_j$ over $\Omega$, $\|v_j\|_{C^2}\to 0$.\\
        \textbf{Claim}: $\liminf_{j\to \infty} \|v_j\|_{L^2(\Omega_0)} / c_j >0$.\\
        \textbf{Proof of the claim}: Otherwise, up to a subsequence, $v_j/c_j \to 0$ in $L^2(\Omega_0)$. Since $V_j$ is stationary in $(M, g_j)$, by lemma \ref{Lem_Pre, Vol element and mean curv of graph I}, $\forall \phi\in C_c^2(\Omega_0)$, 
        \begin{align}
        \begin{split}
         0 &= \int_{\Omega_0}\scM^{g_j}v_j\cdot \phi \\
           &= \int_{\Omega_0} a_{g_j}^{ik}(x, v_j, \nabla v_j)\partial_i(v_j)\partial_k\phi + b_{g_j}^k(x, v_j, \nabla v_j)\partial_k\phi + \partial_z F^{g_j}(x, v_j ,\nabla v_j)\phi \\
           &= \int_{\Omega_0} -v_j\cdot \partial_i\big(a_{g_j}^{ik}(x, v_j, \nabla v_j)\partial_k \phi \big) + b_{g_j}^k(x, v_j, \nabla v_j)\partial_k\phi + \partial_z F^{g_j}(x, v_j ,\nabla v_j)\phi 
        \end{split} \label{Ass Jac, test function in min surf equ}
        \end{align}
        multiply both side of (\ref{Ass Jac, test function in min surf equ}) by $1/c_j$ and let $j\to \infty$, by lemma \ref{Lem_Pre, Vol element and mean curv of graph I} and the assumption that $(g_j - g)/c_j\to \beta$ in $C^4(M)$ and $v_j/c_j \to 0$ in $L^2(\Omega_0)$, we get \[
         \int_{\Omega_0} \beta(\nu, \nabla_{\Sigma} \phi) + \frac{1}{2}tr_{\Sigma}(\nabla^M_{\nu}\beta)\cdot \phi  = 0\ \ \ \ \forall \phi\in C^2_c(\Omega_0)    \]
        This contradicts to that $spt(q(\beta))\cap \Omega_0 \neq \emptyset$. Thus the claim is proved.\\
        
        With this claim, we see that $\exists \kappa>0$ such that $\kappa\|g_j - g\|_{C^4} \leq \|v_j\|_{L^2(\Omega)}$. Now repeat the previous proof by applying lemma \ref{Lem_Ass Jac, Dichotomy Growth Rate Bd} with also taking $\tau < dist(Sing(\Sigma), \Omega_0)$ and $\kappa>0$ specified above. We still get a local-control-global typed $L^2$-estimate as (\ref{Ass Jac, global bded by loc}), and then the conclusion holds immediately by elliptic estimate.
       \end{proof}
       
       We close this section by pointing out the following improved characterization of nearby minimal hypersurfaces using their associated Jacobi fields. 
       
       \begin{Cor}  \label{Cor_Ass Jac, Max growth of ass Jac => smooth}
        Suppose $\sigma <1$; $\{(V_j, g_j)\}_{j\geq 1}$ is a sequence of varifold-metric pair described in theorem \ref{Thm_Ass Jac, Main thm}. 
        Suppose $\bar{v}\in C_{loc}^2(\Sigma)$ is an associated generalized Jacobi field of $\{(V_j, g_j)\}_{j\geq 1}$; $p\in Sing(\Sigma)$ such that $\cA\cR_{p}(\bar{v})\leq \sigma$. Then up to a subsequence, there exists a $\delta_3(p)$-minimal blow up sequence of $\{V_j\}_{j\geq 1}$ near $p$, every blow up limit $\tilde{V}_{\infty}$ of which is asymptotic to $C_p$ near infinity with $\cA\cR_{\infty}(\tilde{V}_\infty) \leq \sigma$.
        
        In particular, suppose $\sigma\leq \gamma_1^+$, and either (B) of theorem \ref{Thm_Ass Jac, Main thm} hold or $C_p$ is area-minimizing in both side. Then $Sing(V_j)\cap B^M_{\tau_{\Sigma}(1)}(p) = \emptyset$ for $j>>1$.
       \end{Cor}
       \begin{proof}
        Directly apply lemma \ref{Lem_Ass Jac, Dichotomy Growth Rate Bd}. By lemma \ref{Lem_Linear, Growth rate charact} and lemma \ref{Lem_Ass Jac, asymp char}, take WLOG $\sigma\in (-\infty, 1)\setminus \Gamma_{C_p}$. 

        By lemma \ref{Lem_Linear, Growth rate charact}, since $\cA\cR_p(\bar{v})\leq \sigma$, (ii) of lemma \ref{Lem_Ass Jac, Dichotomy Growth Rate Bd} can not happen. Hence there exists a blow up limit with asymptotic rate to $C_p$ near infinity $\leq \sigma$.
        
        When $\sigma\leq \gamma_1^+(C_p)<\gamma_2^+(C_p)$, by theorem \ref{Thm_Ass Jac, Asymp Thm of ext. min. graph}, $\tilde{V}_{\infty}$ is a closed embedded smooth minimal hypersurface in $\RR^{n+1}$ of multiplicity $1$, hence by Allard regularity, for those $V_j$ whose blow-up sequence converges to $\tilde{V}_{\infty}$, $V_j$ has no singularity near $p$.
       \end{proof}

       \begin{Cor} \label{Cor_Ass Jac, Growth rate bd near strictly minimizing singularity}
        Suppose $p\in Sing(\Sigma)$ such that $C_p$ is strictly minimizing. Let $\{V_j\}_{j\geq 1}$ be a family of stationary integral varifold in $B^M_{2\tau}(p)$ which converges in varifold sense to $|\Sigma|\llcorner B^M_{2\tau}(p)$. 

        Let $\sigma\in (-\infty, \gamma_1^+)\setminus \Gamma_{C_p}$ be fixed; $N, s_0, \delta_3$ be the same as in lemma \ref{Lem_Ass Jac, Dichotomy Growth Rate Bd}; $s_0' := s_0\cdot K^{-N-1}$.
        
        Suppose $r_j\to 0_+$ and $\|v_j\|^*_{2, \cA_{r_j, 2\tau}}\leq \delta_3$ be the graphical function of $V_j$ over $\Sigma$, i.e. $V_j\llcorner A^M_{r_j, 2\tau} = |graph_{\Sigma}(v_j)|\llcorner A^M_{r_j, 2\tau}$. Then for $j>> 1$, the following pointwise bound holds on $\cA_{Kr_j, \tau/K}$ \[
         |v_j|/\rho + |\nabla v_j| \leq C(\sigma, \Sigma, M, g)\rho^{\sigma - 1}\cdot \big(\int_{\cA_{s_0', \tau}} v_j^2\ \big)^{1/2}    \]
       \end{Cor}
       \begin{proof}
        Directly apply lemma \ref{Lem_Ass Jac, Dichotomy Growth Rate Bd} one get an $L^2$ bound of $v_j$ on each $\cA_{r, Kr}$; then since $v_j$ solves the minimal surface equation, classical interior Schauder estimate yields the corollary.
       \end{proof}
       
       \begin{proof}[Proof of theorem \ref{Thm_Intro, Induced Jacobi field}.]
        The existence of associated Jacobi field $u$ follows from theorem \ref{Thm_Ass Jac, Main thm}; If assuming $u>0$ near some $p\in Sing(\Sigma)$, then by a blow up argument and lemma \ref{Lem_Pre, pos Jacob field}, $\cA\cR_p(u)\in \{\gamma_1^{\pm}(C_p)\}$. The regularity of $\Sigma_j$ near $p$ then follows from corollary \ref{Cor_Ass Jac, Max growth of ass Jac => smooth}. 
       \end{proof}

    \section{One-sided deformations of minimal hypersurfaces} \label{Sec, One-sided Perturb}
     
     Let $(\Sigma,\nu)\subset (M,g)$ be a strictly stable, closed minimal hypersurface with strongly isolated singularities, i.e. $-L_{\Sigma}$ is strictly positive on $\scB_0(\Sigma)$ as in proposition \ref{Prop_Linear, Basic prop for L^2-noncon} (4). Recall "closed" means $Clos(\Sigma)$ is compact.
     Suppose $M\setminus Clos(\Sigma)$ has 2 connected component $M_{\pm}$. This is always possible if we replace $M$ by a small neighborhood of $\Sigma$. Suppose $M_+$ be the one which $\nu$ points in.
     
     Let $\lambda_1>0$ be the first $L^2$-eigenvalue of $-L_{\Sigma}$ on $\scB_0(\Sigma)$; $\Lambda \in (0, \lambda_1/2]$ be fixed throughout this section. Note that by definition we have 
     \begin{align}
      Q_{\Sigma}(\phi, \phi)  \geq 2\Lambda\int_{\Sigma} \phi^2  \ \ \ \ \forall \phi \in C_c^1(\Sigma)   \label{One-sided Perturb, Q^h strictly positive} 
     \end{align}
     
     The aim of this section is to prove theorem \ref{Thm_Intro, loc minimiz}. More precisely, we shall prove the following
     
     \begin{Thm} \label{Thm_One-sided Perturb, Main Thm, mean convex foliation}
      Let $\Sigma, M$ be as above; $\tau \in (0, \tau_{\Sigma}(1)/4)$, where $\tau_{\Sigma}$ is defined in (\ref{Pre, conic rad of sing}). Suppose each tangent cone of $\Sigma$ is strictly minimizing. 
      Then there exists a family of piecewise smooth neighborhood $\{U_t\}_{t\in (0, 1)}$ of $Clos(\Sigma)$ such that
      \begin{enumerate}
       \item[(1)] $U_{t_1}\subset\subset U_{t_2}$ for any $t_1< t_2$;
       \item[(2)] $\bigcap_{t\in (0,1)} U_t = Clos(\Sigma)$; $\bigcap_{s\in (t, 1)}U_s = Clos(U_t)$ and $\bigcup_{s\in(0, t)}U_s = U_t$, $\forall t\in (0,1)$;
       \item[(3)] $\exists$ smooth domains $\{\Omega_t\}_{t\in (0,1)}$, $\Sigma\setminus B_{\tau}(Sing(\Sigma)) \subset \Omega_t\subset \subset \Sigma$ and decomposition \[
         \partial U_t \cap M_{\pm} = \bar{\Sigma}_t^{\pm, ext} \cup \bigcup_{p\in Sing(\Sigma)}\bar{\Sigma}_{t, p}^{\pm, int}    \]
       where $\bar{\Sigma}_t^{\pm, ext}$ are smooth minimal graphs over $\Omega_t$ with smooth boundaries; $\bar{\Sigma}_{t, p}^{\pm, int}$ are smooth minimal hypersurfaces in $B^M_{\tau}(p)$ with boundary $\partial \bar{\Sigma}_{t, p}^{\pm, int} = \partial \bar{\Sigma}_t^{\pm, ext} \cap B^M_{\tau}(p)$.
       \item[(4)] Along their intersection, $\bar{\Sigma}_t^{\pm, ext}$ and $\bar{\Sigma}_{t, p}^{\pm, int}$ forms a convex wedge.
      \end{enumerate}
     \end{Thm}
     
     \begin{Rem} \label{Rem_One-sided Perturb_Mean conv dom can be smooth approx}
      Note that $\partial U_t$ are uniform local Lipschitz hypersurfaces. Hence by \cite{EckerHuisken91_Int_Reg}, by running mean curvature flow starting from $\partial U_t$, we see from (3) and (4) above that $U_t$ can be approximated by smooth mean convex domains.
     \end{Rem}
     
     We first sketch the construction of $U_t$.      
     Let \[
      \rho(x) := min\ \{dist_{\Sigma}(x, p): p\in Sing(\Sigma)\}\cup \{2\tau\}    \]
     When working near a certain singularity $p$, we shall parametrize $\cB_{2\tau}(p)$ by the tangent cone $C_p$ of $\Sigma$ near $p$. We may abuse the notations and simply write any point $x\in \cB_{2\tau}(p)$ as $x = (r, \omega)$ as in (\ref{Pre, param of asymp conic at 0}). 
     Then near each $p\in Sing(\Sigma)$, $\rho(r, \omega) = r$ if $r\leq 2\tau$. 
     
     Denote for simplicity $L = L_{\Sigma}$ and $L_{\Lambda}:= L + \Lambda$.
     By (\ref{One-sided Perturb, Q^h strictly positive}), $-L_{\Lambda}$ is strictly positive, with the first $L^2$-eigenvalue $\geq \Lambda >0 $. Fixed $p_0\in \Sigma\setminus \cB_{4\tau}(Sing(\Sigma))$. For each $p\in Sing(\Sigma)$, let $G_p$ be the Green's function of $L_{\Lambda} $ at $p$ constructed in corollary \ref{Cor_Linear, Global Green's func exists} with $G_p(p_0)=1$. Let $G := \sum_{p\in Sing(\Sigma)}G_p$. By corollary \ref{Cor_Linear, Global Green's func exists} and \ref{Cor_Linear, Growth est for G and vaiphi_1}, $G>0$ satisfies the equation $L G = -\Lambda\cdot G$ on $\Sigma$ and growth estimate \[
      G(x)/\rho(x)^{\gamma_1^-(C_p) + \epsilon}\to + \infty\ \ \ \text{ as }x\to p    \]
     for every $\epsilon > 0$, where recall $\gamma_1^-(C_p)$ is the spectrum of $C_p$ defined in (\ref{Pre, power of cone Jac}). Therefore, there's a sufficiently large $T_0 > 0$ such that \[
      \{x\in \Sigma : G(x)\rho(x)^{-1} > T_0\} \subset \cB_{\tau}(Sing(\Sigma))    \]

     \textbf{From now on in this section, every constant will depend on $M, g, \Sigma, \Lambda, \tau, p_0$.}

     For each $T > T_0$ and $p\in Sing(\Sigma)$, let 
     \begin{align}
      r_p(T):= \sup\{R: G(x)\rho(x)^{-1} >T \text{ on }\cB_R(p) \}   \label{One-sided Perturb, Def r_p(T)}     
     \end{align}
     and \[ 
      \Sigma_T := \Sigma \setminus \bigcup_{p\in Sing(\Sigma)}\cB_{r_p(T)}(p)      \]
     By corollary \ref{Cor_Linear, Unique Asymp for Green's func}, there's a $T_1 > T_0$ such that on $[T_1, +\infty)$, $r_p(T)$ varies continuously in $T$, $\Sigma_T$ are smooth manifolds with boundary; Also by Harnack inequality, for every $p\in Sing(\Sigma)$,  
     \begin{align}
      T\cdot r_p(T) = \inf_{ \partial \Sigma_T \cap \cB_{\tau}(p)} G \leq \sup_{ \partial \Sigma_T \cap \cB_{\tau}(p)} G \leq C(\Sigma, \Lambda, \tau) T\cdot r_p(T)  \label{One-sided Perturb_G/T r_p(T) euqiv 1}
     \end{align}
     $\Sigma_T$ will be the spaces on which we shall build graph.

     For each $T>>T_1$, we shall choose specific small positive function $\varphi^+ \in C^{\infty}(\partial \Sigma_T)$ and construct a graph of some positive function $u^+\in C^{\infty}(\Sigma_T)$ with mean curvature $h$ and boundary $graph_{\Sigma}\varphi^+$. Such a graph will be called an \textbf{exterior hypersurface} with boundary $\varphi^+$; 
     
     Then near each singularity $p$, solve Plateau problem with boundary $graph_{\Sigma}(\varphi^+)$ in $\cB_{\tau}(p)$ to find a minimizing hypersurface, called an \textbf{interior hypersurface} with boundary $\varphi+$ near $p$. The union of these exterior and interior hypersurfaces, denoted by $\Sigma_T^+$, is piecewise smooth and will lie in $M_+$ for sufficiently large $T$. Repeat the same process to construct $\Sigma_T^-$ lying in $M_-$. Then for each $t\in (0, 1)$, a reparametrization $T= T(t)\geq T_1$ will be specified below such that regions $U_t$ bounded by $\Sigma_{T(t)}^{\pm}$ are what we want.
     
     The mean convexity of $\partial U_t$ will be derived by a blow up argument letting $T\to +\infty$ and using the Hardt-Simon foliation as barriers. 
     To make it precise, we need an existence result of exterior graphs over $\Sigma_T$ with uniform growth estimate near each singularity. 
%     In the following theorem, we shall parametrize a sufficiently small neighborhood of $\Sigma$ in M by Fermi coordinate.
     We assert here that this existence result does \textbf{NOT} require the tangent cones of $\Sigma$ to be strictly stable or minimizing.

     \begin{Thm} \label{Thm_One-sided Perturb, Ext graph with growth est}
      Let $\alpha\in (0,1)$. Then there exists $C=C(\alpha)>1$, $T_2(\alpha)>T_1$ and $\delta_5 = \delta_5(\alpha) >0$ sufficiently small such that, for each $T\in (T_2, +\infty)$, each $0< \delta \leq \delta_5$ and each $\varphi \in C^{2,\alpha}(\partial \Sigma_T)$ with estimate near each singularity $p$ 
      \begin{align}
       \frac{1}{r_p(T)}|\varphi| + |\nabla \varphi| + r_p(T)|\nabla^2 \varphi|  \leq \delta   \label{One-sided Perturb, Bdy value of ext graph}  
      \end{align}
      there's a unique $u\in C^{2,\alpha}(\Sigma_T)$ solving 
      \begin{align}
       \begin{cases}
        \scM_{\Sigma} u(x) = 0\ & \text{ on }\Sigma_T \\
         u = \varphi\ & \text{ on }\partial \Sigma_T
       \end{cases} \label{One-sided Perturb, Ext graph sol of min surf equ}
      \end{align}
      and satisfying the \textbf{pointwise estimate}: for every $x\in \Sigma_T$,
      \begin{align}
       \frac{|u(x)|}{\rho(x)}+|\nabla u(x)|+ \sup\Big\{\rho(x)^{\alpha}\cdot \frac{|\nabla u(x)- \nabla u(y)|}{dist_{\Sigma}(x, y)^{\alpha}}: y\in \Sigma_T, \frac{1}{10}\leq \frac{\rho(x)}{\rho(y)}\leq 10 \Big\} \leq C\delta\cdot \frac{G(x)}{T\rho(x)}  \label{One-sided Perturb, Asymp est near bdy of ext graph}
      \end{align}
      where $\scM_{\Sigma}:=\scM^g$ is the minimal surface operator on $\Sigma$ specified in lemma \ref{Lem_Pre, Vol element and mean curv of graph I}. 
     \end{Thm}

     \begin{Rem} \label{Rem_One-sided Perturb_Conti depend of ext graph sol wrt bdy and T}
      By the uniqueness of solution above and the estimate (\ref{One-sided Perturb, Asymp est near bdy of ext graph}), we know further that $u$ depends continuously with respect to $\varphi$ in $C^2$-norm and in $T$.
     \end{Rem}

     We first finish the proof of theorem \ref{Thm_One-sided Perturb, Main Thm, mean convex foliation} assuming theorem \ref{Thm_One-sided Perturb, Ext graph with growth est} to be true:
     \begin{proof}[Proof of theorem \ref{Thm_One-sided Perturb, Main Thm, mean convex foliation}.]
     Let $\delta_5$ be determined in theorem \ref{Thm_One-sided Perturb, Ext graph with growth est}. For each singularity $p$ with strictly minimizing tangent cone $(C_p, \nu_p)$ at $p$, let $S_p$ be the cross section of $C_p$. For sufficiently small $\delta_6< \delta_5/2$, there's a leaf of Hardt-Simon foliation $\Sigma_p^+\subset \RR^{n+1}$ and a positive function $u_p^+ \in C^{\infty}(C_p)$ satisfying 
      \begin{align}
       \begin{split}
        \Sigma_p^+ \setminus \BB^{n+1}_{1/4} & = graph_C (u_p^+) \setminus \BB^{n+1}_{1/4}   \\ 
       \|u_p^+(1,\cdot )\|_{C^2(S_p)} & = \delta_6 \leq C\inf_{S_p} u_p^+(1,\cdot)    
       \end{split} \label{One-sided Perturb, Int graph model near p}
      \end{align}
      We fix such a $\delta_6$ from now on so that (\ref{One-sided Perturb, Int graph model near p}) is true with certain $u_p^+$ for every $p\in Sing(\Sigma)$.
      
      Now for each $T > T_2$, define $\varphi_T^+\in C^2(\partial \Sigma_T)$ to be
      \begin{align}
       \varphi_T^+(r_p(T), \omega) := r_p(T)u_p^+(1,\omega) \ \ \ \forall \omega \in S_p  \label{One-sided Perturb, Prescribe bdy of ext graph foliation} 
      \end{align}
      By a standard compactness argument, for sufficiently large $T$, $\varphi_T^+$ satisfies (\ref{One-sided Perturb, Bdy value of ext graph}) near each singularity $p$ with $\delta = 2\delta_6$, and \[
       \inf_{\partial \Sigma_T} \varphi_T^+/r_p(T) \geq \delta_6/C     \]
      Let $u_T^+$ be the solution to (\ref{One-sided Perturb, Ext graph sol of min surf equ}) with $\varphi = \varphi_T^+$ satisfying the pointwise estimate (\ref{One-sided Perturb, Asymp est near bdy of ext graph}). Then $\Sigma_T^{+,ext} := graph_{\Sigma_T}(u_T^+) $ is a smooth minimal graph over $\Sigma_T$ with boundary $graph_{\Sigma}(\varphi_T^+)$. \\
      \textbf{Claim 1}: $u_T^+>0$ for every $T\in (T_2,+\infty)$.\\
      \textbf{Proof of claim 1}: For fixed $T>T_1$ and every $s\in (0, 1]$, consider the solution $u^+_{T, s}$ of (\ref{One-sided Perturb, Ext graph sol of min surf equ}) on $\Sigma_T$ with $\varphi = s\varphi_T^+$. Then by theorem \ref{Thm_One-sided Perturb, Ext graph with growth est}, $u_{T, s}^+$ is unique and varies continuously in $s$. Moreover, by (\ref{One-sided Perturb, Asymp est near bdy of ext graph}) and the definition of $\varphi_T^+$ we have, $u_{T, s}^+>0$ in some neighborhood of $\partial \Sigma_T$ independent of $s$.

      When $s\to 0_+$, by (\ref{One-sided Perturb, Asymp est near bdy of ext graph}) and classical elliptic estimate \cite{GilbargTrudinger01}, $u^+_{T, s}/s \to u_{T,0}^+$ in $C^2(\Sigma_T)$ for some $u_{T, 0}^+$ satisfying $L \hat{u}^+_{T, 0} = 0$ on $\Sigma_T$ and $\hat{u}^+_{T, 0} = \varphi_T^+ > 0$ on $\partial \Sigma_T$. Then by the strict positivity of $-L$, $\hat{u}^+_{T, 0}>0$ on $\Sigma_T$. Hence $u^+_{T, s}>0$ on $\Sigma_T$ for $s<<1$.
      
      Suppose for contradiction that $\{u^+_T\leq 0\}\neq \emptyset$. Consider $\inf\{s\in (0, 1]: \{u^+_{T, s}\leq 0\}\neq \emptyset\}=:s_0 >0$. By continuous dependence of $u^+_{T, s}$ in $s$, $u^+_{T, s_0}\geq 0$, nonzero on $\partial\Sigma_T$ and vanishes somewhere in $Int(\Sigma_T)$. This contradicts to strong maximum principle since $0$ is also a solution to (\ref{One-sided Perturb, Ext graph sol of min surf equ}) with $\varphi = 0$.\\

      Now for each $p \in Sing(\Sigma)$, let $\Sigma_{T, p}^{+, int}$ be the minimizing hypersurface with boundary $\partial \Sigma_T^{+,ext} \cap \cB_{\tau}(p)$. Such hypersurface is unique and smooth for $T\in (T_3, +\infty)$ for some $T_3>T_2$ by the uniqueness of Hardt-Simon's foliation $\Sigma_p^+$, nondegeneracy of $L_{\Sigma_p^+}$ on bounded subsets, Allard regularity theorem \cite{Allard72, Simon83_GMT} and implicit function theorem. Let \[
       \Sigma^+_T := \Sigma_T^{+,ext} \cup (\bigcup_{p\in Sing(\Sigma)}\Sigma_{T, p}^{+, int}\ )     \]
      with normal field $\nu_T^+$ in $M$ pointing away from $\Sigma$, defined in the interior of each smooth component. \\
      \textbf{Claim 2}: For each $p\in Sing(\Sigma)$, let $\nu_T^{+,ext}:= \nu_T^+|_{\Sigma_T^{+, ext}}$ be the normal field of $\Sigma_T^{+, ext}$ in $M$ (pointing away from $\Sigma$), and $\xi^{+, int}_{T,p}$ be the inward normal field of $\partial \Sigma_{T, p}^{+, int}$ in $\Sigma_{T, p}^{+, int}$. Then there exists $T_4>T_3$ such that for every $T\in (T_4, +\infty)$, 
       \[  \nu_T^{+, ext}\cdot \xi^{+, int}_{T,p} < 0\ \ \ \ \text{ on }\partial \Sigma_{T, p}^{+, int}     \]
      \textbf{Proof of claim 2}: Identify $\Sigma_T^{+, ext}$ as a hypersurface with boundary in $T_p M$ using $exp^M_p$. Suppose up to a subsequence, \[  
        \frac{1}{r_p(T)}\Sigma_T^{+,ext} \to \Sigma_p^{+, ext}\ \ \ \text{ as }T \to +\infty    \]
      By (\ref{One-sided Perturb, Asymp est near bdy of ext graph}), the convergence is in $C^1$ sense as an exterior graph over $C_p$, and $\Sigma_p^{+, ext}$ is a smooth graph with vanishing mean curvature of some $u_p^{+, ext} > 0$ over $C_p\setminus \BB^{n+1}_1$. Moreover, by the boundary condition (\ref{One-sided Perturb, Int graph model near p}) \& (\ref{One-sided Perturb, Prescribe bdy of ext graph foliation}) in (\ref{One-sided Perturb, Ext graph sol of min surf equ}), \[
        u_p^{+, ext}(1, \omega) = u_p^+(1, \omega) \ \ \ \text{ for every }\omega\in S_p    \]
      and by (\ref{One-sided Perturb, Asymp est near bdy of ext graph}), corollary \ref{Cor_Linear, Unique Asymp for Green's func} and the definition of strictly minimizing, we have $u_p^{+, ext} < u_p^+$ near $\infty$. Hence, by considering $\inf\{R>0 : R\cdot \Sigma_p^+ \text{ lies above }\Sigma_p^{+, ext}\}$ and using strong maximum principle for minimal hypersurfaes, 
       \begin{align}
        u_p^{+, ext} < u_p^+\ \ \ \text{ on }C_p \setminus Clos(\BB_1^{n+1})       \label{One-sided Perturb, Blow-up ext graph lie below HS foliation}
       \end{align}
       Since both $u_p^{+,ext}$ and $u_p^+$ satisfies the minimal surface equation on $C_p\setminus Clos(\BB_1^{n+1})$, by (\ref{One-sided Perturb, Blow-up ext graph lie below HS foliation}) and Hopf boundary lemma, \[
        \partial_r u_p^{+, ext}(1, \omega) < \partial_r u_p^+ (1, \omega)\ \ \ \text{ for every }\omega\in S_p    \]
       And since the convergence of $\frac{1}{r_p(T)}\Sigma_T^{+, ext}$ is in $C^1$ sense, the strict inequality in claim 2 is true for sufficiently large $T$.\\

      With claim 2, $(\Sigma^+_T,\nu^+_T)$ is mean convex in $B^M_{2\tau}(p)$ for each $p\in Sing(\Sigma)$ and all $T > T_4$. Let $U_T^+$ be the open domain bounded by $\Sigma^+_T$ and $Clos(\Sigma)$. Using the same blow up argument as claim 2, one can show that there's a $T_5>T_4$ such that for every $T\in (T_5, +\infty)$, \[
       \partial\Sigma^+_{T+s}\subset U_T^+ \ \ \ \forall s>0   \]
      Hence by the uniqueness of solutions to (\ref{One-sided Perturb, Ext graph sol of min surf equ}), the uniqueness of $\Sigma_{T, p}^{+, int}$ and a barrier argument similar to claim 1, we see that for every $T>T'>T_5$, $\Sigma_T^+\subset U_{T'}^+$. In particular, combined with the uniqueness of solutions to (\ref{One-sided Perturb, Ext graph sol of min surf equ}) and $\Sigma^{+,int}_{T, p}$, we have $\{\Sigma^+_{T'}\}_{T>T'}$ forms a foliation of $U_{T'}^+$. 
      
      Similarly, one can construct $\Sigma^-_T$ on the other side of $\Sigma$ satisfying the similar properties for $T>T_6$.
       
      Let $\bar{T} = max\{T_5, T_6\}$ and $U_t$ be the domain bounded by $\Sigma^+_{\bar{T}+1/t -1}$ and $\Sigma^-_{\bar{T} + 1/t -1}$, $t\in (0,1)$. (1)-(4) in Theorem \ref{Thm_One-sided Perturb, Main Thm, mean convex foliation} are easily verified by the constructions above.
     \end{proof}

%     \begin{Rem}
%      If we assume that the tangent cone at each singularity is strictly stable and strictly minimizing, then the proof of Theorem \ref{Thm_One-sided Perturb, Main Thm, mean convex foliation} is much simpler, given by combining [Hardt-Simon '85, Thm 4] and corollary \ref{Cor_Linear, Growth est for G and vaiphi_1} instead of using Theorem \ref{Thm_One-sided Perturb, Ext graph with growth est}.
%     \end{Rem}

     Now we start to prove Theorem \ref{Thm_One-sided Perturb, Ext graph with growth est}. We first introduce the following weighted norms for functions $u$ defined on $E\subset \Sigma$. Let $\alpha\in (0,1)$, $q\geq 1$, $k\in \NN$.
     \begin{align}
     \begin{split}
      \|u\|_{C^{k, \alpha}_G(E)} & := \sum_{j=0}^k \|\nabla^j u\cdot \frac{\rho^j}{G}\|_{C^0(E)} + \sup\{\frac{|\nabla^k u(x) -\nabla^k u(y)|}{dist_{\Sigma}(x, y)^{\alpha}}\cdot \frac{\rho(x)^{k+\alpha}}{G(x)}: x\neq y\in E, \frac{1}{10} < \frac{\rho(x)}{\rho(y)} < 10 \} \\  
      \|u\|_{L^q_G(E)} & := \sup_{t>0}\ \big(t^{-n}\cdot\int_{E\cap \{t\leq \rho\leq 100t\}} (u/G)^q\ d\scH^n(x)\big)^{1/q} \\
      \|u\|_{W^{k,q}_G(E)} & :=\sum_{j=0}^k \||\nabla^j u|\rho^j\|_{L^q_G(E)}      
     \end{split} \label{One-sided Perturb_Def C^k,a_G and L^q_G norm} 
     \end{align} 
     By applying Morrey's inequality in every annuli $\cA_{t, 100t}(p)\subset \Sigma_T$, for $q\geq n/(1-\alpha)$, 
     \begin{align}
      \|u\|_{C^{1,\alpha}_G(\Sigma_T)} \leq C(\alpha, q)\|u\|_{W^{2,q}_G(\Sigma_T)}  \label{One-sided Perturb_Morrey est C^(1, alpha)< W^(2, q)}     
     \end{align}
     for every $u$ and $T> T_1$. We emphasis that by Harnack inequality for $G$, the constant here doesn't depend on $T$.
     
     Fix $q=n/(1-\alpha)$ from now on, and define for each $T>T_1$, $\|u\|_{X_T}:= T\cdot\|u\|_{W^{2,q}_G(\Sigma_T)}$. Note that (\ref{One-sided Perturb, Bdy value of ext graph}) hold for all $p\in Sing(\Sigma)$ is equivalent to that $T\cdot \|\varphi\|_{C^2_G(\partial \Sigma_T)} \leq \delta$, and (\ref{One-sided Perturb, Asymp est near bdy of ext graph}) is equivalent to $T\cdot\|u\|_{C^{1,\alpha}_G}\leq C\delta$, hence is implied by that $\|u\|_{X_T}\leq C\delta$ by (\ref{One-sided Perturb_Morrey est C^(1, alpha)< W^(2, q)}).
     
     Let $\delta_2'$ be determined in lemma \ref{Lem_Pre, Vol element and mean curv of graph II}. Let $\scR$ be the error term operator given by \[
      \scM_{\Sigma}u (x)=:  -L u + \scR u   \] 
      By lemma \ref{Lem_Pre, Vol element and mean curv of graph II}, if $\|u_1\|_{X_T}, \|u_2\|_{X_T}\leq \theta\leq \delta_2'$, then the following pointwise inequality holds,
     \begin{align}
      |\scR u_1 - \scR u_2| \leq C_1\theta\cdot \frac{G}{\rho^3 T}\Big[ \big(|u_1-u_2| + \rho|\nabla(u_1 - u_2)| \big)(1+ \frac{\rho^2|\nabla^2 u_2|\cdot T}{G\cdot \theta}) + \rho^2|\nabla^2(u_1 -u_2)|  \Big]  \label{One-sided Perturb_|Ru_1-Ru_2| bound}
     \end{align}
     Note that if write $f:= \big[ \big(|u_1-u_2| + \rho|\nabla(u_1 - u_2)| \big)(1+ \frac{\rho^2|\nabla^2 u_2|\cdot T}{G\cdot \theta}) + \rho^2|\nabla^2(u_1 -u_2)|  \big]$, then by definition and (\ref{One-sided Perturb_Morrey est C^(1, alpha)< W^(2, q)}),
     \begin{align}
     \begin{split}
       \|f\|_{L^q_G(\Sigma_T)} & \leq C  \Big[ \big(\| \frac{|u_1 - u_2|}{G}\|_{C^0} + \|\frac{\rho |\nabla(u_1 - u_2)|}{G}\|_{C^0} \big)\cdot (1+ \|u_2\|_{W_G^{2,q}}\cdot \frac{T}{\theta}) 
         + \|u_1 - u_2\|_{W_G^{2,q}} \Big] \\
       & \leq C_1'\|u_1 - u_2\|_{X_T}\cdot T^{-1}
     \end{split} \label{One-sided Perturb_|f|_(L^q) < |u_1 - u_2|_X / T}
     \end{align}
     
     The following linear estimate is the key part of the proof.
     \begin{Lem} \label{Lem_One-sided Perturb_Linear est for Lv=Gf/rho^3}
      $\exists\ C_2 = C_2(q)>1$, $T_2(q)> T_1$ such that, if $T>T_2$, $v\in W^{2,q}(\Sigma_T)$ is the solution to 
      \begin{align}
      \begin{cases}
       L v = \frac{G}{\rho^3 T}\cdot f\ & \text{ on }\Sigma_T \\
        v = \phi\ & \text{ on }\partial \Sigma_T
      \end{cases}  \label{One-sided Perturb_Equ Lv = f G/rho^3T, v = phi on bdy}
      \end{align}
      Then \[
       \|v\|_{X_T} \leq C_2 T (\|f\|_{L^q_G(\Sigma_T)} + \|\phi\|_{C^2_G(\partial \Sigma_T)})    \]
     \end{Lem}
     \begin{proof}
      The key is that $C_2$ doesn't depend on $T$. 
      Note that $G/{\rho^3 T} \sim \rho^{-2}$ near $\partial \Sigma_T$. Hence, we need to carefully deal with the estimate near $\partial \Sigma_T$.
      
      \textbf{Step 1 (Local estimate) }
      We shall work near any fixed singularity $p$ of $\Sigma$ and parametrize $\Sigma\cap \cB_{2\tau}(p)$ by the tangent cone $C_p$ or by $(0,+\infty)\times S_p$ as in section \ref{Subsec, Geom of cone}.  By corollary \ref{Cor_Linear, Unique Asymp for Green's func}, $\rho(x)\cdot\partial_r G/G (x)\to \gamma_1^-(C_p)$ as $x\to p$. Hence by Harnack inequality and the definition of $r_p(T)$ (\ref{One-sided Perturb, Def r_p(T)}), for every $\sigma < \gamma_1^-(C_p) <\sigma'$ and every $T>T_1$, 
      \begin{align}
       C(\Sigma, \sigma)(\frac{\rho}{r_p(T)})^{\sigma}\leq \frac{G}{r_p(T) T} \leq C(\Sigma, \sigma') (\frac{\rho}{r_p(T)})^{\sigma'}  \label{One-sided Perturb_rho^sigma < G< rho^sigma'}
      \end{align}
      on $\Sigma_T \cap \cB_{4\tau}(p)$. 
      We now fix a choice of $\sigma_p\in (-\infty, \gamma_1^-(C_p))\setminus \Gamma_{C_p}$ and $\sigma_p' \in (\gamma_1^-(C_p), (1+\sigma_p)/2)$.
      
      Let $t_p = t_p(\sigma_p, \Sigma, M)\in (0, \tau)$ such that $L = L_{C_p} + \scR_p$ satisfies the condition in corollary \ref{Cor_Pre, sol perturbed Jac equ, prescribed asymp}. The existence of such $t_p$ is guaranteed by lemma \ref{Lem_Pre, geom of asymp at 0}. Also let $T_{2, p}>T_1$ such that $r_p(T_{2,p})<t_p/10$. By corollary \ref{Cor_Pre, sol perturbed Jac equ, prescribed asymp}, for every $T>T_{2,p}$, there exists a solution $v_p\in W^{2,2}_{\sigma_p}(\cB_{t_p}(p))$ to the equation $L v_p = \chi_{\Sigma_T}fG/(\rho^3 T)$ on $\cB_{t_p}(p)$ satisfying the weighted estimate 
      \begin{align}
       \sup_{t\in (0, t_p)} \|v_p(t, \cdot)\|_{L^2(S_p)}\cdot t^{-\sigma_p} & \leq C(\sigma_p, t_p)\|\frac{G}{\rho^3 T}\cdot\chi_{\Sigma_T}f\|_{L^2_{\sigma_p-2}(\cB_{t_p}(p))}   \label{One-sided Perturb_|v_p| leq t^-sigma_p ||fG/rho^3T||}
      \end{align}
      And if denote for simplicity $\cA_j:=\cA_{100^{-j}t_p, 100^{-j+1}t_p}(p)$, then by definition,
      \begin{align*}
       \|\frac{G}{\rho^3 T}\cdot\chi_{\Sigma_T}f\|^2_{L^2_{\sigma_p-2}(\cB_{t_p}(p))} 
       & = \int_{\cB_{t_p}(p)} \chi_{\Sigma_T}\cdot(\frac{G}{\rho^3 T}f)^2\rho^{-2\sigma_p + 4 - n}\ dx \\
       & = T^2\cdot\sum_{j\geq 1}\int_{\cA_j} (\frac{f}{G})^2\rho^{-2n/q}\cdot \chi_{\Sigma_T}\cdot (\frac{G}{T})^4 \rho^{-2\sigma_p-2-n+2n/q}\ dx \\
       & \leq C(q)T^2 \|f\|_{L^q_G(\Sigma_T)}^2\cdot\sum_{j\geq 1}\big(\int_{\cA_j}\chi_{\Sigma_T}(\frac{G}{T})^{4q'}\rho^{(-2\sigma_p-2)q' - n}\ dx\ \big)^{1/q'} \\
       & \leq C(q, \sigma_p, \sigma_p', t_p)T^2 \|f\|_{L^q_G(\Sigma_T)}^2\cdot\int_{\cA_{r_p(T),t_p}(p)} r_p(T)^4(\frac{\rho}{r_p(T)})^{4\sigma'_p}\cdot \rho^{-2\sigma_p - 2 - n}\ dx \\
       & \leq C(q, \sigma_p, \sigma_p', t_p)r_p(T)^{2-2\sigma_p}T^2 \|f\|_{L^q_G(\Sigma_T)}^2 
      \end{align*}
      where the first inequality is by Holder inequality with $q' = q/(q-2)$; the second inequality is by RHS of (\ref{One-sided Perturb_rho^sigma < G< rho^sigma'}) and the last inequality is true since we choose $\sigma_p'<(1+\sigma_p)/2$. Combined with (\ref{One-sided Perturb_|v_p| leq t^-sigma_p ||fG/rho^3T||}), the interior $L^q$ estimate for elliptic equations and the Morrey's inequality we get the pointwise estimate on $\cB_{t_p/2}(p)$,
      \begin{align}
       |v_p| + \rho|\nabla v_p|  \leq C(q, \sigma_p, \sigma_p', t_p) (\frac{\rho}{r_p(T)})^{\sigma_p}\cdot r_p(T)T \cdot\|f\|_{L^q_G(\Sigma_T)}
        \leq C \|f\|_{L^q_G(\Sigma_T)}\cdot G \label{One-sided Perturb_|v_p|leq G|f|_L^q ptwsly}
      \end{align}
      where the last inequality above follow from LHS of (\ref{One-sided Perturb_rho^sigma < G< rho^sigma'}). Since we have fixed the choice of $\sigma_p, \sigma_p', t_p$, from now on every constant might depend on them in this proof.\\
      
      \textbf{Step 2 (Global $C^0$ estimate) }
      Now let $T>T_2:=\sup_{p\in Sing(\Sigma)} T_{2,p}$; Let $v$ be the solution to the equation (\ref{One-sided Perturb_Equ Lv = f G/rho^3T, v = phi on bdy}). Let $\bar{v}:= v - \sum_{p\in Sing(\Sigma)} \eta_p v_p$, where $v_p$ is given by step 1 satisfying the estimate (\ref{One-sided Perturb_|v_p|leq G|f|_L^q ptwsly}) in $\cB_{t_p/2}(p)$; $\eta_p \in C_c^\infty(\cB_{t_p/2}(p))$ is a cut-off function which equals to $1$ on $\cB_{t_p/4}(p)$ and has $|\nabla \eta_p|+|\nabla^2 \eta_p|\leq C(t_p)$. 
      
      By definition of $v_p$ we have \[
       L\bar{v} = \frac{G}{\rho^3T}\cdot f(1-\sum_{p\in Sing(\Sigma)}\eta_p) - \sum_{p\in Sing(\Sigma)}(2\nabla \eta_p\cdot \nabla v_p + v_p\Delta \eta_p )  \]
      In particular, $spt(L\bar{v})\subset \Sigma\setminus \bigcup_{p\in Sing(\Sigma)}\cB_{t_p/4}(p)$ and by (\ref{One-sided Perturb_|v_p|leq G|f|_L^q ptwsly}), $\|L \bar{v}\|_{L^q(\Sigma)}\leq C\|f\|_{L^q_G(\Sigma_T)}$, where $C$ is independent of $T$.
      
      Let $\cU_0:= \Sigma\setminus \bigcup_{p\in Sing(\Sigma)}\cB_{t_p/8}(p)$ and $\bar{v}_0\in W_0^{1,2}(\cU_0)$ be the solution of $L\bar{v}_0 = L\bar{v}$ on $\cU_0$. Then by global $L^q$ estimate and Morrey's inequality, 
      \begin{align}
       \|\bar{v}_0\|_{C_G^1(\cU_0)}\leq C\|f\|_{L^q_G(\Sigma_T)}  \label{One-sided Perturb_|bar(v)_0|leq G|f| ptwsly} 
      \end{align}
      
      Let $\eta_0\in C_c^\infty(\cU_0)$ be a cut-off function with value $1$ on $\Sigma\setminus \bigcup_{p\in Sing(\Sigma)}\cB_{t_p/4}(p)$ and $|\nabla \eta_0|+|\nabla^2\eta_0|\leq C(\cU_0)$. Then 
      \begin{align*}
      \begin{cases}
       |L(\bar{v}-\eta_0\bar{v}_0)| = |2\nabla \eta_0\cdot \nabla \bar{v}_0 + \bar{v}_0\Delta\eta_0| \leq  C\|f\|_{L^q_G(\Sigma_T)}\cdot G\  &\text{ on }\Sigma_T \\
       |\bar{v}-\eta_0\bar{v}_0| = |v -\sum_{p\in Sing(\Sigma)}v_p| \leq  C(\|f\|_{L^q_G(\Sigma_T)}+\|\phi\|_{C^0_G(\partial \Sigma_T)})\cdot G\  &\text{ on }\partial \Sigma_T
      \end{cases}
      \end{align*}
      But recall that by definition, $-LG=\Lambda\cdot G$ and $-L$ is strictly positive on $\Sigma$, hence by weak maximum principle for $\bar{v}-\eta_0\bar{v}_0$ and the estimate (\ref{One-sided Perturb_|v_p|leq G|f|_L^q ptwsly}) (\ref{One-sided Perturb_|bar(v)_0|leq G|f| ptwsly}), we have weighted $C^0$ estimate for $v$
      \begin{align}
       |v|\leq C(\|f\|_{L^q_G(\Sigma_T)} + \|\phi\|_{C^0_G(\partial \Sigma_T)})\cdot G \ \ \ \text{ on }\Sigma_T  \label{One-sided Perturb_|v|leq (|f|+|phi|)G}
      \end{align}
      where the constant $C$ is independent of $f, \phi, T$.\\
      
      \textbf{Step 3 (Global $W^{2,q}_G$ estimate) }
      Note that since $0<G/{\rho T}\leq C(\Sigma)$ on $\Sigma_T$, we have $|L v|\leq C(\Sigma) f/\rho^2$. Hence the desired estimate for $\|v\|_{X_T}$ follows from (\ref{One-sided Perturb_|v|leq (|f|+|phi|)G}) and interior and boundary $L^q$ estimate on every domain $\{t<\rho<100t\}\cap \Sigma_T$, $\forall t>0$. This completes the proof of lemma \ref{Lem_One-sided Perturb_Linear est for Lv=Gf/rho^3}.
     \end{proof}

     \begin{proof}[Proof of theorem \ref{Thm_One-sided Perturb, Ext graph with growth est}.]
      Recall that $\delta_2'$ is determined in lemma \ref{Lem_Pre, Vol element and mean curv of graph II}. 
      Let $\bar{C}>1$ and $\delta_5\in (0, \delta_2'/2\bar{C})$ TBD subsequently. For every $T>T_1$, $\delta\in (0, \delta_5]$ and $\varphi\in C^2(\partial \Sigma_T)$ with $T\cdot\|\varphi\|_{C^2_G(\partial\Sigma_T)}\leq \delta$, define $X_{T, \bar{C}\delta}:= \{u\in W^{2,q}(\Sigma_T): \|u\|_{X_T}\leq \bar{C}\delta \}$. 
     
      Consider $\scT: X_{T, \bar{C}\delta}\to W^{2,q}(\Sigma_T)$, $\scT u =:v$ solves 
      \begin{align*}
      \begin{cases}
       L v = \scR u\ & \text{ on }\Sigma_T \\
        v = \varphi\ & \text{ on }\partial \Sigma_T
      \end{cases}
      \end{align*}
      Note that by non-degeneracy of $L$, the equation above has a unique solution. We shall appropriately choose $\bar{C}$ and $\delta_5$ to make $\scT$ a contraction map from $X_{T, \bar{C}\delta}$ to itself. 
      
      For every $u\in X_{T, \bar{C}\delta}$, let $v:=\scT u$. By taking $u_1 = u$ and $u_2 = 0$ in (\ref{One-sided Perturb_|Ru_1-Ru_2| bound}) and applying (\ref{One-sided Perturb_|f|_(L^q) < |u_1 - u_2|_X / T}) and lemma \ref{Lem_One-sided Perturb_Linear est for Lv=Gf/rho^3}, 
      \begin{align}
       \|v\|_{X_T}\leq C_2(C_1C_1'\bar{C}\delta\cdot \|u\|_{X_T} + \delta )\leq C_2(C_1C_1' \bar{C}^2\delta^2 + \delta) \label{One-sided Perturb_ |v|_X_T < C_2 delta}
      \end{align}
      
      For every $u_1, u_2\in X_{T, \bar{C}\delta}$, let $v_i := \scT u_i$. Also by (\ref{One-sided Perturb_|Ru_1-Ru_2| bound}), (\ref{One-sided Perturb_|f|_(L^q) < |u_1 - u_2|_X / T}) and lemma \ref{Lem_One-sided Perturb_Linear est for Lv=Gf/rho^3}, 
      \begin{align}
       \|v_1 - v_2\|_{X_T}\leq C_2(C_1C_1'\bar{C}\delta\cdot \|u_1 - u_2\|_{X_T} + 0 ) \label{One-sided Perturb_ |v_1 - v_2|_X_T < C_2 delta |u_1 - u_2|}
      \end{align}
      Hence, take $\bar{C}:= 4C_2$ and $\delta_5:= 1/(C_1C_1'\bar{C}^2)$, we see from (\ref{One-sided Perturb_ |v|_X_T < C_2 delta}) and (\ref{One-sided Perturb_ |v_1 - v_2|_X_T < C_2 delta |u_1 - u_2|}) that $\scT$ is a contraction map from $X_{T, \bar{C}\delta}$ to itself. Therefore, the theorem follows from Banach fixed point theorem.
     \end{proof}
     \begin{Rem}
      The uniqueness of solution $u$ in the statement of theorem \ref{Thm_One-sided Perturb, Ext graph with growth est} seems slightly stronger since we only assume $\|u\|_{C^{1,\alpha}_G}$ bound. However, by similar argument as step 3 in the proof of lemma \ref{Lem_One-sided Perturb_Linear est for Lv=Gf/rho^3}, for sufficiently small $\delta'$ (not depending on $T$), (\ref{One-sided Perturb, Asymp est near bdy of ext graph}) hold for $u$ with $\delta = \delta'$ implies $\|u\|_{X_T}\leq C'\delta'$. Thus reduce the uniqueness of solutions with (\ref{One-sided Perturb, Asymp est near bdy of ext graph}) estimate to the uniqueness of solutions in $X_{T, \bar{C\delta}}$ above, which comes from Banach fixed point theorem. 
     \end{Rem}
     \begin{Rem} \label{Rem_One-sided Perturb_MainThm Deform true in stable case}
      By theorem \ref{Thm_One-sided Perturb, Main Thm, mean convex foliation} and strong maximum principle \cite{SolomonWhite89_StrongMax}, $\Sigma$ is homologically area minimizing in $Clos(U_1)$ and is the unique stationary integral varifold supported in $Clos(U_1)$. Hence theorem \ref{Thm_Intro, loc minimiz} is proved.
      
      Also notice that,directly by a compactness argument, theorem \ref{Thm_Intro, main thm, deform} is proved in case when $\Sigma$ is nondegenerate stable, even without assuming (3).
     \end{Rem}
     \begin{Rem} \label{Rem_One-sided Perturb_Loc Mininizing near Sing}
      Theorem \ref{Thm_One-sided Perturb, Main Thm, mean convex foliation} also guarantees that a locally stable minimal hypersurface $\Sigma$ must be uniquely area-minimizing near the singular point $p$ at which the tangent cone is regular and strictly minimizing. In fact, by proposition \ref{Prop_Linear, Basic prop for L^2-noncon} (5), one can change the metric outside a small neighborhood of a $p$ to make $\Sigma$ strictly stable. Then theorem \ref{Thm_One-sided Perturb, Main Thm, mean convex foliation} provides a mean convex foliation of $Clos(\Sigma)$, hence contains a small neighborhood of $p$, inside which $[\Sigma]$ is uniquely area minimizing.
     \end{Rem}

    \section{Strong min-max property for singular minimal hypersurfaces} \label{Sec, Strong min-max prop}
     
     The goal of this section is to prove a strong min-max property for certain non-degenerate singular minimal hypersurface. The proof closely follows the argument in \cite{White94}, The only difference is the treatment of singularities, which is based on the results in section \ref{Sec, Asymp & Asso Jac field} and \ref{Sec, One-sided Perturb}.
     
     Let's begin with the basic set up:
     \begin{enumerate}
      \item[(S1)] $(\Sigma, \nu)\subset (M^{n+1}, g)$ be a two-sided closed minimal hypersurface with only strongly isolated singularities; $M\setminus Clos(\Sigma) = M_- \sqcup M_+$ where $\nu$ points into $M_+$. 
      \item[(S2)] Assume the Jacobi operator $L_{\Sigma}$ is nondegenerate in $\scB_0(\Sigma)$ with $I:= ind(\Sigma; M)\geq 1$;
      \item[(S3)] Assume that $\forall p\in Sing(\Sigma)$, the tangent cone $C_p$ of $\Sigma$ at $p$ is strictly minimizing and strictly stable.
      \item[(S4)] Let $\tau \in (0, \tau_{\Sigma}(1)/4)$, $K>1$; Let $\{U_j\}_{j\geq 1}$ be a decreasing family of smooth domains satisfying 
       \begin{enumerate}
        \item[•] $\bigcap_{j\geq 1} U_j = Clos(\Sigma)$.
        \item[•] $\partial U_j$ are mean convex in $B^M_{4\tau}(Sing(\Sigma))$
        \item[•] $\partial U_j\cap M_{\pm}$ are graphs of smooth functions $u_j^{\pm}\in C^2(\Sigma)$ in $M\setminus B_{\tau}^M(Sing(\Sigma))$ such that 
         \begin{align*}
           |\nabla u_j^{\pm}| + |\nabla^2 u_j^{\pm}| & \leq K |u_j^{\pm}| \ \ \ \text{ on }\Sigma\setminus B_{\tau}(Sing(\Sigma)) \\
           \sup_{\Sigma\setminus B_{\tau}(Sing(\Sigma))} |u^{\pm}_j| & \leq K\cdot \inf_{\Sigma\setminus B_{\tau}(Sing(\Sigma))} |u^{\pm}_j| \to 0_+\ \ \ \text{ as }j\to \infty   
         \end{align*}
       \end{enumerate}
     \end{enumerate}
     
     The existence of such $K, \tau$ and $\{U_j\}_{j\geq 1}$ can be seen from the following argument. 
     By remark \ref{Rem_One-sided Perturb_Loc Mininizing near Sing}, for each $p\in Sing(\Sigma)$, let $O_p\subset M$ be a smooth convex neighborhood of $p$ in which $[\Sigma]$ is uniquely area-minimizing; $\Sigma_{p, j}^{\pm}\subset M_{\pm}\cap O_p$ be a family of area-minimizing hypersurfaces in $O_p$ which converges to $\Sigma\cap O_p$ from both side; Such $\Sigma_{p, j}^{\pm}$ exists by solving Plateau problem with boundary to be the one-sided perturbation of $\partial O_p\cap \Sigma$; For every compact subset $Z\subset O_p$, $Sing(\Sigma_{p, j})\cap Z = \emptyset$ for $j>>1$ by Allard regularity \cite{Allard72, Simon83_GMT}, theorem \ref{Thm_Pre, H-S foliation} and a blow-up argument similar to \cite{Smale93}. Hence $U_j$ in (S4) can be constructed by (passing to a subsequence if necessary) gluing these $\Sigma_{p, j}^{\pm}$ with a constant function over a neighborhood of $\Sigma\setminus \bigcup_{p\in Sing(\Sigma)} O_p$ by partition of unity.
     
%     An alternative approach is the following.
%     Notice that if $h\in C_c^{\infty}(M\setminus Sing(\Sigma))$ is a nonnegative function such that $h = |\nabla_M h| = 0$ on $\Sigma$, then under $\tilde{g}:= (1+h)g$, $\Sigma$ is also a minimal hypersurface, with second variation \[
%      Q_{\Sigma, \tilde{g}}(\phi, \phi) = Q_{\Sigma, g}(\phi, \phi) + \frac{n}{2}\int_{\Sigma} \nabla_M^2 h (\nu, \nu)\cdot \phi^2     \]
%     Hence by proposition \ref{Prop_Linear, Basic prop for L^2-noncon}, there exists some $h$ as above such that $\Sigma$ is strictly stable in $(M, \tilde{g})$. We fix a choice of such $h$.

%     Let $h$ be fixed above; $\tau\in (0,\tau_{\Sigma}(1)/4)$ such that $B_{4\tau}^M(Sing(\Sigma))\cap spt(h) = \emptyset$; $\rho(x):= min\{dist_{\Sigma}(Sing(\Sigma), x), \tau\}$.

%     Let $\{\tilde{U}_t\}_{t\in (0,1)}$ be the mean convex neighborhoods of $Clos(\Sigma)$ in $(M, \tilde{g})$ constructed in theorem \ref{Thm_One-sided Perturb, Main Thm, mean convex foliation}. Note that in $spt(h)$, $\partial \tilde{U}_t$ have vanishing mean curvature under metric $\tilde{g}$. Hence for every $t$, by looking at the mean curvature flow of $\tilde{U}_t$ in $(M, \tilde{g})$, we can find a smooth domain $\tilde{U}_{t/2}\subset U_t \subset \tilde{U}_t$ which has mean convex boundary inside $\{h=0\}$ and have mean curvature under metric metric $g$ satisfying \[
%      |H_{\partial U_t}|_x \leq C dist_M(x, \Sigma)\ \ \ \ \forall t\in (0, 1),\ %\forall x\in spt(h)\cap \partial U_t   \]

     Throughout this and next section every constant will depend on the quantities in (S1)-(S4). 
     
     \begin{Thm} \label{Thm_Strong min-max, Main}
      There exists $k_1 \geq 1$, $\Phi_0: Clos(\BB^I_1) \to \cZ_n(Clos(U_{k_1}))$ $\bfF$-continuous and $\epsilon:(0,1]\to (0,1)$ non-decreasing such that 
      \begin{enumerate}
      \item[(i)] $\Phi_0(0^I) = [\Sigma]$;
      \item[(ii)] $\bfM(\Phi_0(z)) \leq \bfM(\Sigma) - \epsilon(r)$ for every $z\in \partial \BB^I_r$;
      \item[(iii)] If $r\in (0, 1]$ and $\{\Phi_i\}_{i\geq 1}: Clos(\BB^I_r) \to \cZ_n(Clos(U_{k_1}))$  is a family of $\bfF$-continuous map such that \[
           \limsup_{i\to \infty} \sup\{\bfF(\Phi_0(z), \Phi_i(z)): z\in \partial \BB_r^I\} = 0   \]
         Then $\liminf_{i\to \infty}\sup\{\bfM(\Phi_i(z)): z\in \BB^I_r\} \geq \bfM(\Sigma)$. 
      \end{enumerate}
     \end{Thm}
     
     We start with some analytic preparation.
     \begin{Lem} \label{Lem_Strong min-max, I direction of deform}
      There exists $\phi_1, \phi_2, ... \phi_I\in C^\infty_c(\Sigma)$ and $\lambda, \mu>0$ such that
      \begin{enumerate}
      \item[(i)] For every $u\in span_{1\leq i\leq I}\{\phi_i\}$, \[
       Q_{\Sigma}(u, u) \leq -\lambda\|u\|^2_{L^2(\Sigma)}   \]
      \item[(ii)] For every $v\in C^1_c(\Sigma)$, \[
       Q_{\Sigma}(v, v)+ \mu^2\sum_{i=1}^I (\int_{\Sigma}v\cdot \phi_i\ )^2 \geq \lambda\int_{\Sigma} |\nabla v|^2 + v^2/\rho^2   \]
      \item[(iii)] The matrix $[\langle \phi_i, \phi_j \rangle_{L^2(\Sigma)}]_{(i,j)\in I\times I}$ is nondegenerate.
      \end{enumerate}       
     \end{Lem}
     \begin{proof}
      Let $\psi_1, ..., \psi_I \in \scB_0(\Sigma)$ be the first $I$ pairwise orthogonal unit $L^2$-eigenfunctions of $-L_{\Sigma}$. Since $C^\infty_c(\Sigma)$ is dense in $\scB_0(\Sigma)$, take $\phi_i\in C_c^\infty(\Sigma)$ such that $\sup_{1\leq i\leq I}\|\phi_i - \psi_i\|_{\scB(\Sigma)} \leq \varepsilon$, where $\varepsilon \in (0,1)$ TBD.
      
      To prove (i), notice that for every $a\in \RR^I$, \[
        Q_{\Sigma}(\sum_i a_i \psi_i, \sum_i a_i\psi_i) \leq -|\lambda'|\sum_i a_i^2   \]
      where $\lambda'$ be the largest negative eigenvalue of $-L_{\Sigma}$; Also notice that 
      \begin{align*}
       & |Q_{\Sigma}(\sum_i a_i \phi_i, \sum_i a_i\phi_i) -Q_{\Sigma}(\sum_i a_i \psi_i, \sum_i a_i\psi_i)|  \\
       =\ & |\sum_{1\leq i,j\leq I}a_ia_j(Q_{\Sigma}(\phi_i, \phi_j) - Q_{\Sigma}(\psi_i, \psi_j))|\leq C(\Sigma, I)\varepsilon\cdot \sum_i a_i^2 
      \end{align*}
      Hence by taking $\varepsilon<|\lambda'|/2C(\Sigma, I)$, we have (i) for $\lambda \leq |\lambda'|/2$.
      
      To prove (ii), first notice that for every $v\in \scB_0(\Sigma)$ \[
       Q_{\Sigma}(v, v) + |\mu'|\sum_{i=1}^I (\int_{\Sigma}v\cdot \psi_i\ )^2 \geq \lambda'' \|v\|_{L^2(\Sigma)}^2       \]
      where $\lambda''$ be the smallest positive eigenvalue of $-L_{\Sigma}$ and $\mu'$ be the difference between $\lambda''$ and the first eigenvalue of $-L_{\Sigma}$.
      Also note that \[
       (\int_{\Sigma} v\cdot \psi_i\ )^2 - (\int_{\Sigma} v\cdot \phi_i\ )^2 \leq C'(\Sigma)\varepsilon\cdot \|v\|_{L^2}^2     \]
      Hence if take $\varepsilon < \lambda''/(2I|\mu'|C'(\Sigma))$, $\mu^2\geq |\mu'|$ and $\lambda \leq \lambda''/2$, then we have 
      \begin{align}
       Q_{\Sigma}(v, v)+ \mu^2\sum_{i=1}^I (\int_{\Sigma}v\cdot \phi_i\ )^2 \geq \lambda\|v\|^2_{L^2(\Sigma)}   \label{Strong min-max_Q + <phi, >^2 geq L^2}       
      \end{align}
      
      Also note that since each tangent cone of $\Sigma$ is strictly stable, by lemma \ref{Lem_Pre, Hardy-typed inequ} and lemma \ref{Lem_Pre, geom of asymp at 0}, there is some small neighborhood $\cV$ of $Sing(\Sigma)$ and $\vartheta(\Sigma)>0$ such that \[
       Q_{\Sigma}(v, v) - \vartheta\int_{\cV} v^2/\rho^2 \geq 0\ \ \ \forall v\in C_c^1(\cV)   \]
      Hence by a similar argument as lemma \ref{Linear, essential positivity of Q} we have, for some $C''(\Sigma)>0$,
      \begin{align}
       Q_{\Sigma}(v, v) - \vartheta\int_{\Sigma} v^2/\rho^2 + C''\|v\|^2_{L^2(\Sigma)} \geq 0\ \ \ \forall v\in C_c^1(\Sigma)  \label{Strong min-max_Q + L^2 geq L^2_w}
      \end{align}
      
      Also by lemma \ref{Lem_Pre, geom of asymp at 0}, $\big||A_{\Sigma}|^2+ Ric_M(\nu, \nu)\big| \leq C'''(\Sigma)/\rho^2$ on $\Sigma$, hence $\forall v\in C_c^1(\Sigma)$
      \begin{align}
      \begin{split}
       & (1+a)\big(Q_{\Sigma}(v, v)+ C''\|v\|^2_{L^2(\Sigma)}\big) \\
       \geq\ & a\int_{\Sigma}|\nabla v|^2 + \Big(Q_{\Sigma}(v, v) - \vartheta\int_{\Sigma} v^2/\rho^2 + C''\|v\|^2_{L^2(\Sigma)}\Big) \geq a\int_{\Sigma}|\nabla v|^2
       \end{split}   \label{Strong min-max_Q + L^2 geq |nabla|^2}
      \end{align}
      where $0<a\leq \vartheta/C'''$. Hence lemma \ref{Lem_Strong min-max, I direction of deform} (ii) follows by combining (\ref{Strong min-max_Q + <phi, >^2 geq L^2}), (\ref{Strong min-max_Q + L^2 geq L^2_w}) and (\ref{Strong min-max_Q + L^2 geq |nabla|^2}) and take $0<\lambda \leq min\{1, \lambda''/(2C'')\}\cdot min\{\vartheta, a/(1+a)\}/2$ .
      
      To prove (iii), it suffices to notice that $|\langle \phi_i, \phi_j \rangle_{L^2(\Sigma)} - \delta_{ij}| \leq 10\varepsilon $ and take $\varepsilon < 1/100I^2$.
     \end{proof}     

     \begin{Rem}
       We see from its proof that lemma \ref{Lem_Strong min-max, I direction of deform} (i) and (\ref{Strong min-max_Q + <phi, >^2 geq L^2}) holds for any $\Sigma$ satisfying $L^2$ nonconcentration property; while the proof of (ii) essentially use strictly stability of tangent cones of $\Sigma$.
     \end{Rem}

     Now take $\vec{f}=(f_1, ..., f_I)\in C^{\infty}(M\setminus Sing(\Sigma); \RR^I)$ such that $\vec{f} = 0$ on $\Sigma$ and $\nu(f_i)= \mu\phi_i$ on $\Sigma$; Let $\bar{\tau}\in (0, \tau)$ such that $B^M_{\bar{\tau}}(Sing(\Sigma))\cap spt(\vec{f}) = \emptyset$.
     
     Let $\zeta\in C^2(\RR^I; [0, +\infty))$ such that $\zeta(0) = 0$, $\nabla^2 \zeta(0) = id$ and $\sup_{\RR^I} |\nabla \zeta|\cdot \sup_M|\vec{f}| <1$. Define for $T\in \cZ_n(M)$, \[
      \bfA^*(T):= \bfM(T) + \zeta\big( \int_M \vec{f}(x)\ d\|T\|(x) \big)    \]
     \begin{Lem} \label{Lem_Strong min-max, L.s.c. for A^* in flat top}
      $\bfA^*$ is lower semi continuous in flat topology.
     \end{Lem}
     The proof is the same as \cite[Section 3, claim 2]{White94}.
     
     For a general parametric elliptic integrand $F$, the regularity problem of $(F, ct^{\alpha}, \delta)$-almost minimizing surfaces is studied in \cite{Almgren75, SchoenSimon82, Bombieri82}, and is used in \cite{White94} to gain better regularity of convergence of minimizers for $F^*$ in a smaller and smaller neighborhood of the given surface stationary with respect to $F$. $F\equiv 1$ corresponds to the mass functional.
     
     Observe that since $U_j$ are mean convex near $Sing(\Sigma)$ and $\vec{f}$ vanishes near $Sing(\Sigma)$, minimizers of $\bfA^*$ in $Clos(U_t)$ are minimal near each singularity. Hence the same argument in \cite[Section 3, claim 1]{White94} yields the following, 
     \begin{Lem} \label{Lem_Strong min-max, A^* minimizer is Mass almost minimizer}
      For every $\alpha\in (0, 1)$, there exists some constant $c, k_2>0$ such that for $j\geq k_2$, every minimizer $T\in \cZ_n(Clos(U_j))$ of $\bfA^*$ among cycles homologous to $[\Sigma]$ in $U_j$ is $(ct^{\alpha},\delta_7)$-almost minimizing in $M$.
      
      In particular, if $T_j$ is a minimizer of $\bfA^*$ in $Clos(U_j)$ homologous to $[\Sigma]$, then $T_j \to [\Sigma]$ in flat topology as $j\to \infty$ and 
      \begin{enumerate}
      \item[(i)] $T_j$ are homologically area minimizing in $B^M_{\bar{\tau}}(Sing(\Sigma))\cap Clos(U_j)$.
      \item[(ii)] There are $u_j\in C^1(\Sigma)$ such that for every $U\subset\subset M\setminus Sing(\Sigma)$, \[
        T_j\llcorner U = [graph_{\Sigma} (u_j)]\llcorner U\ \ \ \text{ for }j>>1    \]
      \end{enumerate}
     \end{Lem}
     
     \begin{Lem} \label{Lem_Strong min-max, Sigma minimize A^*}
      There exists $k_1\geq k_2$ such that $[\Sigma]$ is the only  $\bfA^*$-minimizer in its homology class in $\cZ_n(Clos(U)_{k_1})$.
     \end{Lem}
     \begin{proof}
      The proof closely follows \cite[Theorem 2]{White94}. The major extra effort is made to estimate error terms near singularities.
      
      Suppose for contradiction that after passing to a subsequence, there are $[\Sigma]\neq T_j\in \cZ_n(Clos(U_j))$ homologous to $[\Sigma]$, i.e. $T_j - [\Sigma] = \partial P_j$ for some $P_j\in \bfI_{n+1}(M)$ supported in $Clos(U_j)$, which minimize $\bfA^*$ among its homology class in $\cZ_n(Clos(U_j))$.

      Let $0\neq u_j$ be the graphical functions of $T_j$ over $\Sigma$ as described in lemma \ref{Lem_Strong min-max, A^* minimizer is Mass almost minimizer} (ii). Moreover, for each $p\in Sing(\Sigma)$, let \[
        r_j(p):= \inf\{t>0: T_j\llcorner A^M_{t, 2\tau} = [graph_{\Sigma}(u_j)]\llcorner A^M_{t, 2\tau}(p) \text{ and } \|u_j\|^*_{2, \cA_{t, 2\tau}(p)}\leq \delta_3 \}    \]
      where $\delta_3$ is specified in lemma \ref{Lem_Ass Jac, Dichotomy Growth Rate Bd}.
      Clearly, $r_j(p)\to 0_+$ as $j\to \infty$; WLOG $r_j(p)<\tau$.

      Since each $C_p$ is strictly stable, one can take $\sigma \in \big(-(n-2)/2, \inf_{p\in Sing(\Sigma)}\gamma_1^+(C_p) \big)$. By corollary \ref{Cor_Ass Jac, Growth rate bd near strictly minimizing singularity}, there exists $\Omega_0\subset\subset \Sigma$ such that for $j>>1$,
      \begin{align}
       |u_j|/\rho + |\nabla u_j| \leq C(\sigma)\rho^{\sigma -1}\cdot (\int_{\Omega_0} u_j^2\ )^{1/2} \ \ \ \text{ on }\cA_{2r_j(p), \tau}(p) \label{Strong min-max_Growth rate of u_j leq sigma}
      \end{align}
      \textbf{Claim}: $\liminf_{j\to \infty} r_j(p)^{2\sigma -2}\int_{\Omega_0} u_j^2 >0$. \\
      \textbf{Proof of claim}: Otherwise, consider $\hat{V}_j:= (\eta_{p, r_j(p)})_{\sharp} V_j$. By (\ref{Strong min-max_Growth rate of u_j leq sigma}), up to a subsequence, $\hat{V}_j \to \hat{V}_{\infty}$ for some stationary integral varifold $\hat{V}_{\infty}$ equal to $|C_p|$ outside the ball of radius $2$. Since $C_p$ is area minimizing, by a barrier argument using Hardt-Simon foliations, $\hat{V}_{\infty} = |C_p|$ in $\RR^{n+1}$. By Allard regularity, this contradict to the minimality of $r_j(p)$.\\
      
      Let $\tilde{\cU}_j := \Sigma\setminus \bigcup_{p\in Sing(\Sigma)} \cB_{2r_j(p)}(p)$; $\cU_j:= \Sigma\setminus \bigcup_{p\in Sing(\Sigma)} \cB_{4r_j(p)}(p)$. Take $j>>1$ such that $\cU_j \supset \Omega_0\cup (\Sigma \cap spt(\vec{f}) )$.
      
      Let $\xi_j\in C_c^\infty(\tilde{\cU_j})$ be a cut-off function which equal to $1$ on $\cU_j$ and $|\nabla \xi_j|\leq 10/r_j(p)$ on $\cB_{\tau}(p)$; 
      Let $\tilde{u}_j := u_j\cdot \xi_j$. Note that by (\ref{Strong min-max_Growth rate of u_j leq sigma}), 
      \begin{align}
       \int_{\cB_{4r_j(p)}(p)} |\nabla \tilde{u}_j|^2 + \tilde{u}_j^2/\rho^2 \leq C(\sigma) r_j(p)^{2\sigma - 2 + n}\cdot \int_{\Omega_0} u_j^2  \label{Strong min-max_int_(B_rj) u_j^2/rho^2 leq r_j^n-2+2sigma}
      \end{align}

      Now we are ready to estimate $\bfA^*(T_j)- \bfA^*(\Sigma)$. Work under Fermi coordinate near $\Sigma$, let $F(x, z, p)$ be the volume element for graphs over $\Sigma$ as in lemma \ref{Lem_Pre, Vol element and mean curv of graph II}, we have
      \begin{align}
      \begin{split}
       0 &\geq\ \bfA^*(T_j) - \bfA^*(\Sigma)  \\
       &\geq\ \int_{\cU_j} \big( F(x, u_j, \nabla u_j) - 1 \big) - \scH^n(\Sigma \setminus \cU_j)  + 
        \zeta \big(\int_{\Sigma} \vec{f}(x, u_j(x))F(x, u_j, \nabla u_j)\ d\scH^n(x) \big)  \\
       &\geq\ \underbrace{ \Big[\int_{\cU_j}\ dx\int_0^1 \partial_p F(x, su_j, s\nabla u_j)\cdot \nabla u_j + \partial_z F(x, su_j, s\nabla u_j)\cdot u_j\ ds \Big] }_{(I)}
        - \underbrace{\scH^n(\Sigma \setminus \cU_j) }_{(II)} \\  
        &\ +\ \underbrace{\zeta \big(\int_{\Sigma} \vec{f}(x, u_j(x))F(x, u_j, \nabla u_j)\ d\scH^n(x) \big) }_{(III)} 
      \end{split}  \label{Strong min-max_A^*(T_j) - A(Sigma) geq I + II + III}
      \end{align}
      
      To estimate $(I)$, notice that by lemma \ref{Lem_Pre, Vol element and mean curv of graph II}, \[
       |\partial_p F(x, z, p) - p| + \rho\big|\partial_z F(x, z, p) - (|A_{\Sigma}|^2 + Ric_M(\nu, \nu))z \big| \leq C(|z|/\rho + |p|)^2   \]
      Hence, combined with (\ref{Strong min-max_Growth rate of u_j leq sigma}) and (\ref{Strong min-max_int_(B_rj) u_j^2/rho^2 leq r_j^n-2+2sigma}) we have,
      \begin{align}
      \begin{split}
       (I) & \geq \frac{1}{2}\int_{\cU_j}|\nabla u_j|^2 + \big(|A_{\Sigma}|^2 + Ric_M(\nu, \nu)\big)u_j^2 - C(|\nabla u_j| + |u_j|/\rho)^3 \\
       & \geq \frac{1}{2}\int_{\Sigma} |\nabla \tilde{u}_j|^2 - \big(|A_{\Sigma}|^2 + Ric_M(\nu, \nu)\big)\tilde{u}_j^2\ d\scH^n \\
       &\ - C(\sigma, \Omega_0)\int_{\Sigma}(|\nabla \tilde{u}_j| + |\tilde{u}_j/\rho| + r_j(p)^{2\sigma - 2 + n})\cdot(|\nabla \tilde{u}_j|^2 + \tilde{u}_j^2/\rho^2) \\
       & = \frac{1}{2}\int_{\Sigma} |\nabla \tilde{u}_j|^2 - \big(|A_{\Sigma}|^2 + Ric_M(\nu, \nu) \big)\tilde{u}_j^2\ d\scH^n + o(\int_{\Sigma}|\nabla \tilde{u}_j|^2 + \tilde{u}_j^2/\rho^2 )
      \end{split} \label{Strong min-max_I geq Q_Sigma - o(u_j^2/rho^2)}
      \end{align}
      
      To estimate $(II)$, it suffices to see that by claim 1, $r_j(p)^n = o(\int_{\Omega_0} u_j^2)$ as $j\to \infty$. Hence,
      \begin{align}
       (II) \leq \sum_{p\in Sing(\Sigma)} \scH^n(\cB_{4r_j(p)}(p)) \leq C(\Sigma)\cdot \sum_{p} r_j(p)^n = o(\int_{\Sigma}|\nabla \tilde{u}_j|^2 + \tilde{u}_j^2/\rho^2 )   \label{Strong min-max_II leq o(u_j^2/rho^2)}
      \end{align}
      
      To estimate $(III)$, note that by definition, $\zeta(v) \geq |v|^2/2 - C_\zeta |v|^3$ for some $C_\zeta > 0$; and by the choice of $f$,
      \begin{align*}
        |\int_{\Sigma} \vec{f}(x, u_j(x))F(x, u_j, \nabla u_j) - \mu\vec{\phi}\cdot \tilde{u}_j\ d\scH^n(x) | \leq  C(f)\int_{\Sigma} |\nabla \tilde{u}_j|^2 + \tilde{u}_j^2/\rho^2 
      \end{align*}
      Hence, we have 
      \begin{align}
       (III) \geq \frac{\mu^2}{2} \big| \int_{\Sigma} \tilde{u}_j\cdot \vec{\phi}\ d\scH^n \big|^2 - o(\int_{\Sigma} |\nabla \tilde{u}_j|^2 + \tilde{u}_j^2/\rho^2 )  \label{Strong min-max_III geq |<u_j, phi>|^2 + o(u_j^2/rho^2)}
      \end{align}
      
      Now, plug (\ref{Strong min-max_I geq Q_Sigma - o(u_j^2/rho^2)}), (\ref{Strong min-max_II leq o(u_j^2/rho^2)}) and (\ref{Strong min-max_III geq |<u_j, phi>|^2 + o(u_j^2/rho^2)}) into (\ref{Strong min-max_A^*(T_j) - A(Sigma) geq I + II + III}) and apply lemma \ref{Lem_Strong min-max, I direction of deform} (ii), we get, 
      \begin{align*}
       0& \geq Q_{\Sigma}(\tilde{u}_j, \tilde{u}_j) + \mu^2\big| \int_{\Sigma} \tilde{u}_j\cdot \vec{\phi}\ d\scH^n \big|^2 - o(\int_{\Sigma} |\nabla\tilde{u}_j|^2 + \tilde{u}_j^2/\rho^2 ) \\
       & \geq \lambda \int_{\Sigma} |\nabla\tilde{u}_j|^2 + \tilde{u}_j^2/\rho^2\ d\scH^n - o(\int_{\Sigma} |\nabla\tilde{u}_j|^2 + \tilde{u}_j^2/\rho^2 )        
      \end{align*}
      This contradicts to that $\lambda>0$ in lemma \ref{Lem_Strong min-max, I direction of deform} and that $u_j\neq 0$. Thus the proof is completed.
     \end{proof}
     
     \begin{proof}[Proof of theorem \ref{Thm_Strong min-max, Main}.]
      Let $k_1$ be specified in lemma \ref{Lem_Strong min-max, Sigma minimize A^*}.
      Define $\Phi_0 :Clos(\BB^I_1) \to \cZ_n( Clos(U_{k_1}))$ by \[
       \Phi_0(z):= \big[graph_{\Sigma}\big(\bar{t}\cdot(z_1\phi_1+z_2\phi+... + z_I\phi_I) \big) \big]    \]
      where $\bar{t}\in (0, 1)$ small enough such that $Im(\Phi_0) \subset \cZ_n( Clos(U_{k_1}))$ and that (ii) in theorem \ref{Thm_Strong min-max, Main} hold for some $\epsilon(r)>0$. Such $\bar{t}$ and $\epsilon(r)$ exists by lemma \ref{Lem_Strong min-max, I direction of deform} (i).
      
      Also by lemma \ref{Lem_Strong min-max, I direction of deform} (iii), we can take $\bar{t}>0$ even smaller such that \[
       \vec{f}\Phi_0 : Clos(\BB^I_1) \to \RR^I \ \ \ z\mapsto \int_M \vec{f}\ d\|\Phi_0(z)\|   \]
      has non-zero degree at $0^I\in \RR^I$. Fix the choice of $\bar{t}$ and then $\Phi_0$ from now on.
      
      Now if $\{\Phi_i\}_{i\geq 1}$ be a family of $\bfF$-continuous sweepouts described in theorem \ref{Thm_Strong min-max, Main} (iii), then \[
       \vec{f}\Phi_i : Clos(\BB^I_1) \to \RR^I \ \ \ z\mapsto \int_M \vec{f}\ d\|\Phi_i(z)\|   \]
      $C^0$ converges to $\Phi_0$ restricted on $\partial \BB^I_1$, and hence also has non-zero degree at $0^I$ for $i>>1$. In particular, $\exists\ z_i\in \BB_1^I$ such that $\vec{f}\Phi_i (z) = 0^I$. Therefore, by lemma \ref{Lem_Strong min-max, Sigma minimize A^*}, \[
        \sup \{\bfM(\Phi_i(z)): z\in \BB^I_1\} \geq \bfM(\Phi(z_i)) = \bfA^*(\Phi(z_i)) \geq \bfA^*([\Sigma]) = \bfM(\Sigma)   \]
       for $i>>1$, thus proves (iii) of theorem \ref{Thm_Strong min-max, Main}.
     \end{proof}

    \section{A rigidity result for constrained minimal hypersurfaces} \label{Sec, Rigidity of Constraint Minimizer}

     Recall as introduced in \cite{WangZH2020_Obst}, for a smooth domain $U\subset (M, g)$, an integral varifold $V\in \cI\cV_n(Clos(U))$ is called constrained embedded minimal hypersurface in $Clos(U)$ if it's a locally stable minimal hypersurface with optimal regularity inside $U$, locally sum of $C^{1,1}$ graphs with multiplicity near $\partial U$, possibly with self-touching on $\partial U$, and satisfies the variational inequality \[
       \delta V(X) \geq 0    \]
     for every vector field $X\in \scX(M)$ which pointing inward $U$ along $\partial U$. Note that by \cite[Remark 2.9, (1)]{WangZH2020_Obst}, the mean curvature vectors of these $C^{1,1}$ graph comprising $V$ all points outward $U$ along $\partial U$ and equals to the mean curvature of $\partial U$ a.e. on $spt(V)\cap \partial U$. This is the only additional fact we shall use in the proof of lemma \ref{Lem_Rigidity c-Min}.
     
     \cite{WangZH2020_Obst} discusses the existence of such constrained embedded minimal hypersurface through min-max construction, as well as a local rigidity result \cite[Theorem 5.1]{WangZH2020_Obst} near a smooth nondegenerate minimal hypersurface. 
     The goal of this section is to prove an analogue of this local rigidity in the singular setting. The prove is almost the same, but more effort is made to deal with error terms near singularities. 

     We keep the same set up as section \ref{Sec, Strong min-max prop}, i.e. suppose (S1)-(S4) holds.
     Note that by strong maximum principle and mean convexity of $\partial U_j$ in (S4) inside $B^M_{4\tau}(Sing(\Sigma))$, any constrained embedded minimal hypersurface $V$ in $Clos(U_j)$ must be stationary (i.e have vanishing mean curvature) in $B^M_{4\tau}(Sing(\Sigma))$.
     
     \begin{Lem} \label{Lem_Rigidity c-Min}
      Let $\epsilon\in (0,1)$. There exists $k_2=k_2(\epsilon)\geq 1$ such that for every $j\geq k_2(\epsilon)$, if $V$ is a constrained embedded minimal hypersurface in $Clos(U_j)$ with \[
       \scH^n(\Sigma)\leq \|V\|(M)\leq (2-\epsilon)\scH^n(\Sigma)   \]
      then $V = |\Sigma|$. 
     \end{Lem}
     
     \begin{proof}
      We shall argue by contradiction as is done in \cite[Theorem 5.1]{WangZH2020_Obst}. 
      
      Suppose for sake of contradiction that after passing to a subsequence, there exist constrained embedded minimal hypersurfaces $|\Sigma|\neq V_j \in \cI\cV_n(Clos(U_j))$, $j\to \infty$. Since by the choice of $\{U_j\}_{j\geq 1}$ in (S4), the mean curvature of $\partial U_j$ on $spt(h)$ tends to $0$ as $j\to \infty$, we see that the mean curvature of $V_j$ converges to $0$ in $L^\infty$.
%     If $V = \sum_j V_j$, then $|H_{V_j}|\leq h_0$ implies $|H_V|\leq h_0$.
      Hence by Allard compactness theorem \cite{Allard72, Simon83_GMT}, up to a subsequence, $V_j \to V$ for some stationary integral varifold $V\in \cI\cV_n(M)$ supported in $Clos(\Sigma)$ and has $\scH^n(\Sigma)\leq \|V\|(M)\leq (2-\epsilon) \scH^n(\Sigma)$. This implies $V=|\Sigma|$. Thus by Allard regularity theorem \cite{Allard72, Simon83_GMT}, this convergence is in $C^{1,\alpha}$-graphical sense on each compact subset of $\Sigma$, $\forall \alpha\in (0,1)$.
      
      Let $\alpha\in (0,1)$ be fixed from now on; Let $\delta_3 > 0$ be specified in lemma \ref{Lem_Ass Jac, Dichotomy Growth Rate Bd}, $v_j\in C^{1,\alpha}(\Sigma)$ be the graphical function of $V_j$ over $\Sigma$, $r_j(p)\to 0_+$ for each singularity $p$ be such that \[
       r_j(p)=\inf \{t>0: V_j\llcorner A^M_{t, 2\tau}(p) = |graph_{\Sigma}(v_j)|\llcorner A^M_{t, 2\tau}(p),\ \|v_j\|^*_{2; \cA_{t, 2\tau}(p)}\leq \delta_3 \}    \]
      Take $j>>1$ such that $r_j(p)<\tau/2$. 
      Let $\sigma\in \big(-(n-2)/2, \inf_{p\in Sing(\Sigma)}\gamma_1^+(C_p)\big)$ be fixed. By corollary \ref{Cor_Ass Jac, Growth rate bd near strictly minimizing singularity}, there exists $\Sigma\setminus \cB_{\tau}(Sing(\Sigma)) \subset \Omega_0\subset\subset \Sigma$ such that for $j>>1$, 
      \begin{align}
       |v_j|/\rho + |\nabla v_j| \leq C(\sigma)\rho^{\sigma -1}\cdot (\int_{\Omega_0} v_j^2\ )^{1/2} \ \ \ \text{ on }\cA_{2r_j, \tau}(p)  \label{Rigidity c-Min_|v_j| leq rho^sigma |v_j|_L^2}  
      \end{align}
      Hence by the same argument as in the proof of lemma \ref{Lem_Strong min-max, Sigma minimize A^*} claim, we have 
      \begin{align}
       \liminf_{j\to \infty} r_j(p)^{2\sigma -2} \int_{\Omega_0} v_j^2 > 0    \label{Rigidity c-Min_r_j(p)^(2sigma - 2)leq c_j^2}      
      \end{align}
      
      Let $F = F^g(x, z, p)$ be the area integrand of graph over $\Sigma$ defined in lemma \ref{Lem_Pre, Vol element and mean curv of graph II}; $\scM:= \scM^g = -div_{\Sigma}(\partial_p F(x, u, \nabla u)) + \partial_z F(x, u, \nabla u)$ be the minimal surface operator.
      Also let $\cU_j := \Sigma\setminus \bigcup_{p\in Sing(\Sigma)}\cB_{4r_j(p)}(p)$; $u_j^{\pm}$ be $C^2$ functions over $\Omega_0$ such that under Fermi coordinates of $M$ near $\Sigma$, i.e. \[
       \{(x, t): x\in \Omega_0, u_j^-(x)< t < u_j^+(x) \} \setminus B^M_{2\tau}(Sing(\Sigma)) = U_j\setminus B^M_{2\tau}(Sing(\Sigma))     \]
      Since $V_j$ are constrained embedded minimal hypersurfaces in $Clos(U_j)$, we have
      \begin{align}
      \begin{cases}
       v_j\cdot \scM v_j \leq 0\ & \text{ weakly on }\Omega_0; \\
       \scM v_j = 0\ & \text{ on } \{u_j^- < v_j <u_j^+\} \cup (\cU_j\setminus \Omega_0); \\
       |\scM v_j| \leq C|v_j| \ & \text{ weakly on }\cU_j.
      \end{cases} \label{Rigidity c-Min_3 diff in/equ for c-Min}
      \end{align}
      Let $c_j^2:= \int_{\Omega_0}|\nabla v_j|^2 + v_j^2/\rho^2$, then by (\ref{Rigidity c-Min_|v_j| leq rho^sigma |v_j|_L^2}), (\ref{Rigidity c-Min_3 diff in/equ for c-Min}), lemma \ref{Lem_Pre, Vol element and mean curv of graph II} and standard elliptic estimate \cite{GilbargTrudinger01} we have that up to a subsequence, $v_j / c_j \to \hat{v}_{\infty}$ in $W^{1,2}_{loc}\cap C^0_{loc}(\Sigma)$ for some $0\neq\hat{v}_{\infty}\in W^{1,2}_{loc}(\Sigma)$ satisfying
      \begin{align}
      \begin{cases}
       -\hat{v}_{\infty}\cdot L_{\Sigma} \hat{v}_\infty \leq 0\ & \text{ weakly on }\Sigma; \\
       -L_{\Sigma} \hat{v}_\infty = 0\ & \text{ on }\big(\{\hat{u}_\infty^- < \hat{v}_\infty < \hat{u}_\infty^+\}\cap \Omega_0 \big) \cup \Sigma\setminus \Omega_0.
      \end{cases} \label{Rigidity c-Min_2 diff in/equ for renormalized lim c-Min}
      \end{align}
      where $\hat{u}^{\pm}_{\infty}$ are the subsequential limit of $u_j^{\pm}/c_j$, which is either everywhere $\pm \infty$ or everywhere finite and nonzero by the assumption in (S4) and that $|v_j|\leq max\{|u_j^+|, |u_j^-|\}$ on $\Omega_0$ for $j\geq 1$. 
      Moreover by (\ref{Rigidity c-Min_|v_j| leq rho^sigma |v_j|_L^2}), $\cA\cR_p(\hat{v}_\infty)\geq \sigma> \gamma_1^-(C_p)$ for every singular point $p$. Hence by corollary \ref{Cor_Linear, Loc Growth Bd and function decomp} (2) and corollary \ref{Cor_Linear, Growth est for G and vaiphi_1}, $\hat{v}_\infty \in \scB_0(\Sigma)$. \\
      \textbf{Claim}: \[
        \int_{\Sigma} |\nabla \hat{v}_\infty|^2 - (|A_{\Sigma}|^2 + Ric_M(\nu, \nu))\hat{v}_\infty^2 \ dx  =  \int_{\Sigma} -\hat{v}_\infty\cdot L_{\Sigma}\hat{v}_\infty \geq 0    \]
      Note that combine this claim with (\ref{Rigidity c-Min_2 diff in/equ for renormalized lim c-Min}), we see that $\hat{v}_{\infty}$ actually satisfies the Jacobi field equation $L_{\Sigma} \hat{v}_\infty = 0$ on $\Sigma$. Together with that $\hat{v}_\infty \in \scB_0(\Sigma)$, this contradicts to the non-degeneracy of $L_\Sigma$ and completes the proof of lemma \ref{Lem_Rigidity c-Min}. \\
      
\noindent    \textbf{Proof of the claim}: By lemma \ref{Lem_Pre, Vol element and mean curv of graph II},
      \begin{align}
      \begin{split}
       0& \leq \|V_j\|(M)-\scH^n(\Sigma) \leq \int_{\cU_j} F(x, v_j, \nabla v_j) - 1\ d\scH^n(x)  + \sum_{p\in Sing(\Sigma)} \|V_j\|(B^M_{8r_j(p)}(p)) \\
        & \leq \int_{\cU_j}\ dx \int_0^1 \partial_p F(x, sv_j, s\nabla v_j)\cdot \nabla v_j + \partial_z F(x, sv_j, s\nabla v_j)\cdot v_j\ ds  + \sum_{p\in Sing(\Sigma)} \|V_j\|(B^M_{8r_j(p)}(p)) \\
        & = \underbrace{ \frac{1}{2}\int_{\cU_j} \partial_p F(x, v_j, \nabla v_j)\cdot \nabla v_j + \partial_z F(x, v_j, \nabla v_j)\cdot v_j\ dx }_{(I)} + \ \scE + \underbrace{ \sum_{p\in Sing(\Sigma)} \|V_j\|(B^M_{8r_j(p)}(p)) }_{(II)}
      \end{split} \label{Rigidity c-Min_ 0 leq I + scE + II}
      \end{align}
      where by the fact that $\int_0^1 f(s)\ ds - (f(0)+f(1))/2 = (\int_0^1 f''(s)(s^2-s)\ ds)/2$  and lemma \ref{Lem_Pre, Vol element and mean curv of graph II}, we have
      \begin{align}
      \begin{split}
       |\scE| & \leq  C\int_{\cU_j} \sup_{s\in [0,1]}  \Big[ |\partial^3_{ppp} F(x, sv_j, s\nabla v_j)||\nabla v_j|^3 + |\partial^3_{ppz} F(x, sv_j, s\nabla v_j)||\nabla v_j|^2|v_j|  \\
        & \ \ \ \ \ \ \ \ \ \ \ + |\partial^3_{pzz} F(x, sv_j, s\nabla v_j)||\nabla v_j||v_j|^2 + |\partial^3_{zzz} F(x, sv_j, s\nabla v_j)||v_j|^3 \Big] \ dx \\
        & \leq  C\int_{\cU_j} (|\nabla v_j| + |v_j|/\rho)^3\ dx = o(c_j^2)
      \end{split}  \label{Rigidity c-Min_scE = o(c_j^2)}
      \end{align}
      And by (\ref{Rigidity c-Min_r_j(p)^(2sigma - 2)leq c_j^2}) and volume monotonicity for stationary varifolds in $M$,
      \begin{align}
       (II)\leq \sum_{p\in Sing(\Sigma)} C(\tau)r_j(p)^n\cdot \|V_j\|(B^M_\tau (p))  = o(c_j^2)   \label{Rigidity c-Min_II = o(c_j^2)}
      \end{align}
      
      Now to estimate $(I)$, first observe by \[
       |\partial_p F(x, z, p) - p| + \rho \big|\partial_z F(x, z, p) + \big(|A_{\Sigma}|^2 + Ric_M(\nu, \nu) \big) z \big| \leq C(\tau)(|p| + |z|/\rho(x))^2    \]
      Hence,
      \begin{align}
      \begin{split}
       (I) & \leq \frac{1}{2}\int_{\cU_j} |\nabla v_j|^2 - (|A_{\Sigma}|^2 + Ric_M(\nu, \nu))v_j^2\ dx  + C(\tau)\int_{\cU_j} (|\nabla v_j|+ |v_j|/\rho)^3 \\
        & \leq \frac{1}{2}\int_{\cU_j} |\nabla v_j|^2 - (|A_{\Sigma}|^2 + Ric_M(\nu, \nu))v_j^2\ dx + o(c_j^2) 
      \end{split}  \label{Rigidity c-Min_I leq Q_Sigma(v_j, v_j) + o(c_j^2)}
      \end{align}
      Note that by (\ref{Rigidity c-Min_|v_j| leq rho^sigma |v_j|_L^2}), as $j\to \infty$, \[
       \int_{\cU_j} |\nabla (v_j/c_j)|^2 - \big(|A_{\Sigma}|^2 + Ric_M(\nu, \nu)\big)(v_j/c_j)^2\ dx \to \int_{\Sigma} |\nabla \hat{v}_\infty|^2 - \big(|A_{\Sigma}|^2 + Ric_M(\nu, \nu)\big)\hat{v}_{\infty}\ dx    \]
      Therefore, combined with (\ref{Rigidity c-Min_I leq Q_Sigma(v_j, v_j) + o(c_j^2)}), (\ref{Rigidity c-Min_II = o(c_j^2)}) and (\ref{Rigidity c-Min_scE = o(c_j^2)}), by dividing (\ref{Rigidity c-Min_ 0 leq I + scE + II}) by $c_j^2$ and taking $j\to \infty$, the proof of claim is completed.
     \end{proof}

    \section{Applications and further discussions} \label{Sec, Apps and Discussion}
    
    We first finish the proof of theorem \ref{Thm_Intro, main thm, deform} and \ref{Thm_Intro, generic smooth}.
    
    \begin{proof}[Proof of theorem \ref{Thm_Intro, main thm, deform}.]
     First recall that by remark \ref{Rem_One-sided Perturb_MainThm Deform true in stable case}, theorem \ref{Thm_Intro, main thm, deform} is already proved when $\Sigma$ is non-degenerate stable. We now assume $ind(\Sigma) \geq 1$.
     
     Since $\Sigma \subset (M, g)$ in theorem \ref{Thm_Intro, main thm, deform} satisfies (S1)-(S3) in section \ref{Sec, Strong min-max prop} in this case, we can choose $K, \tau$ and $\{U_j\}_{j\geq 1}$ as in (S4).
     
     Let $k_1 \geq 1$ be given in theorem \ref{Thm_Strong min-max, Main}, $k_2\geq 1$ be given in lemma \ref{Lem_Rigidity c-Min}; Let $\bar{k}\geq max\{k_1, k_2(1/2)\}$ be fixed; $R_\Sigma \in (0, 1)$ is also fixed such that $spt(\Phi_0(x))\subset U_{\bar{k}+1}$ for all $x\in Clos(\BB^I_{R_\Sigma})$, where $\Phi_0$ is specified in theorem \ref{Thm_Strong min-max, Main}.  Also let $\eta\in C^\infty_c(U_{\bar{k}+1}, [0,1])$ be a fixed cut off function with  $\eta|_{U_{\bar{k}+2}}=1$.
     
     Let $\Pi$ be the family of sequences of $\bfF$-continuous maps $S= \{\Phi_i: Clos(\BB^I_{R_\Sigma}) \to \cZ_n(Clos(U_{\bar{k}}))\}_{i\geq 1}$ such that \[
      \sup\{\bfF(\Phi_i(z), \Phi_0(z)): z\in \partial \BB_{R_\Sigma}^I\}\to 0 \ \ \ \text{ as }i\to \infty    \]
     Hence by theorem \ref{Thm_Strong min-max, Main} we have \[
      \bfL(\Pi; g) := \inf_{S = \{\Phi_i\}_{i\geq 1}\in \Pi} \limsup_{i\to \infty} \sup_{z\in \BB^I_{R_\Sigma}} \bfM^g(\Phi_i(z)) = \scH^n(\Sigma) > \bfL_{\Phi_0, \partial ; g} := \sup_{z\in \partial \BB^I_{R_\Sigma}}\bfM^g(\Phi_0(z))     \]
     Observe that $\bfL(\Pi; g)$ and $\bfL_{\Phi, \partial ;g}$ is Lipschitz in $g$ in $C_{loc}^0$ sense. Therefore, $\exists\ \epsilon_1 >0$ such that for every $\|g'-g\|_{C^4} \leq \epsilon_1$, if write $g'_\eta:= \eta g' + (1-\eta)g$, then $\bfL(\Pi; g'_\eta) > \bfL_{\Phi_0, \partial; g'_\eta}$. By \cite[Theorem 4.3]{WangZH2020_Obst}, there exists a constrained embedded minimal hypersurface $V_{g'}$ in $(Clos(U_t), g'_\eta)$ with $\|V_{g'}\|(M) = \bfL(\Pi; g'_{\eta})$. 
     Moreover, by \cite[Remark 4.4]{WangZH2020_Obst}, there's a constant $N_I$ only depending on $I$ such that, we can choose $V_{g'}$ satisfying the following,
     
     \textsl{ For every $p\in M$, if $\{A_l = A^M_{r_l, s_l}(p)\}_{l=1}^{N_I}$ be a family of disjoint annuli in $M$ centered at $p$ with $s_l> r_l > 4s_{l+1}$; Then $V$ is a constrained embedded} stable \textsl{minimal hypersurface in at least one of $\{A_l\}$. } 
     
     From this fact, we see that if $g_j' \to g$ in $C^4$, then $V_{g_j'}$ will subconverges in varifold sense to some constrained stationary integral varifold $V_g$ in $Clos(U_{\bar{k}}, g)$ with the same mass as $\Sigma$, which is stable locally in every sufficiently small annuli. Hence by \cite[Theorem 2.11 \& 3.2]{WangZH2020_Obst}, $V_g$ is constrained embedded stable minimal hypersurface; And by lemma \ref{Lem_Rigidity c-Min}, we see that $V_g = |\Sigma|$. 
     
     Also note that since $V_{g'_j}$ has uniformly bounded mean curvature (bounded by the mean curvature of $\partial U_{\bar{k}}$), $spt(V_{g'_j})\to spt(V_g) = \Sigma$ in Hausdorff distant sense. Therefore, for $g'$ sufficiently close to $g$, $spt(V_{g'})\subset U_{\bar{k}+2}$. In particular, $V_{g'}$ is a locally stable minimal hypersurface in $(M, g')$ and can be arbitrarily close to $\Sigma$ by taking $\|g'-g\|_{C^4}$ small enough. This completes the proof of theorem \ref{Thm_Intro, main thm, deform}.
    \end{proof}
    
    \begin{proof}[Proof of theorem \ref{Thm_Intro, generic smooth}.]
     Note that for an arbitrary $f\in C^4(M)$ such that $q(f\cdot g) := -\nu(f)\cdot n/2$ is not identically $0$ on $\Sigma$, and any family of Riemannian metric $g_t = g(1+tf) + o_4(t)$, let $\Sigma_t$ be the locally stable minimal hypersurfaces constructed in theorem \ref{Thm_Intro, main thm, deform} in $(M, g_t)$ which converges to $\Sigma$ as $t\to 0$.
     
     By theorem \ref{Thm_Ass Jac, Main thm}, each subsequence of $\{(|\Sigma_t|, g_t)\}$ induces a generalized Jacobi field $u$ on $\Sigma$, satisfying $L_\Sigma u = c\cdot q(f g)$ for some $c\geq 0$.     
     Moreover, since each tangent cone of $\Sigma$ is strictly minimizing and strictly stable, by taking $\sigma \in (-(n-2)/2, \inf_{p\in Sing(\Sigma)}\gamma_1^+(C_p))$ in corollary \ref{Cor_Ass Jac, Growth rate bd near strictly minimizing singularity}, we have $\cA\cR_{p}(u)\geq \sigma$ for every $p\in Sing(\Sigma)$. Hence by lemma \ref{Lem_Linear, Growth rate charact}, corollary \ref{Cor_Linear, Growth est for G and vaiphi_1} and \ref{Cor_Linear, Loc Growth Bd and function decomp}, $u\in \scB_0(\Sigma)$. Hence by nondegeneracy of $L_{\Sigma}$ and proposition \ref{Prop_Linear, Basic prop for L^2-noncon} (3), $u = L_{\Sigma}^{-1}(c\cdot q(fg))$ is unique (up to a renormalization $c\neq 0$).
     
     Now let \[
      \scV_{reg}:= \{f\in C^4(M):\cA\cR_p( L^{-1}_\Sigma \big( q(fg) \big)\geq \gamma_1^+(C_p) \text{ for every singularity }p\}   \]
     By lemma \ref{Lem_Genericness of Sol with Top Growth Rate}, $C^4(M)\setminus \scV_{reg}$ is meager in $C^4(M)$; By corollary \ref{Cor_Ass Jac, Max growth of ass Jac => smooth}, for every $f\in \scV_{reg}$, the minimal hypersurfaces in $(M, g(1+tf)+o_4(t))$ are smooth for $|t|<<1$. 
    \end{proof}
    
    We shall also point out that the strict minimizing assumption in theorem \ref{Thm_Intro, main thm, deform} can not be dropped. 
    \begin{Prop} \label{Prop_App, Deform Thm fails without minimizing assump}
     There exists a Riemannian manifold $(M, g)$ and a two-sided strictly stable minimal hypersurface $(\Sigma, \nu)\subset (M, g)$ with only strongly isolated singularities satisfying (1) and (3) but not (2) in theorem \ref{Thm_Intro, main thm, deform}, together with a neighborhood $U$ of $Clos(\Sigma)$ and a family of metrics $g_j \to g$ in $C^\infty $ on $M$, such that there's no stationary integral varifold supported in $U$ under each $g_j$.
    \end{Prop}
     Recall that by \cite{Simons68, Lin87, Simoes74, Lawlor91}, the Simons cone $C^{5,1}\subset \RR^8$ over $\SSp^5\times \SSp^1$ is strictly stable but only minimizing in one side. 
     Using this cone as singularity model and applying the following lemma \ref{Lem_App, H-S foliation for stable cones}, one can derive proposition \ref{Prop_App, Deform Thm fails without minimizing assump} using the same argument as in theorem \ref{Thm_One-sided Perturb, Main Thm, mean convex foliation} to construct piecewise smooth mean convex hypersurfaces $\{\Sigma_t\}_{t\in (-1,1)}$ foliating $U$ with $\Sigma_0 = \Sigma$, and such that the mean curvature vector of $\Sigma_t$ points into the same direction for $t\neq 0$.
    We leave the detail of this proof to readers.
    
    \begin{Lem} \label{Lem_App, H-S foliation for stable cones}
     Let $C\subset \RR^{n+1}$ be a stable minimal hypercone which is minimizing only in one-side. Let $E_-$ be the component of $\RR^{n+1}\setminus C$ in which $C$ is not minimizing, $\nu_C$ be the normal field pointing away from $E_-$.
     
     Then the area-minimizer $S$ among all integral currents with boundary $\partial [C\cap \BB_1^{n+1}]$ is unique and is a smooth minimal hypersurface supported in $Clos(E_-)$ with boundary $C\cap \partial \BB^{n+1}_1$.
     Moreover,
     \begin{enumerate}
     \item[(1)] $\{r\cdot S\}_{r>0}$ foliates $E_-$;
     \item[(2)] $S$ is a minimal graph over $C$ near $\partial S$ of some function $h$ satisfying 
      \begin{align}
        h(r, \omega)\leq 0 , \ \ \  h(1, \omega) = 0;\ \ \   \partial_r h(1,\omega) >0 \ \ \ \ \forall  \omega\in C\cap \partial \BB_1^{n+1}   \label{App_partial_r h >0}     
      \end{align}
     \end{enumerate}  
    \end{Lem}
    \begin{proof}
     First note that since $C$ is area-minimizing in the other side, by comparing mass, $spt(S) \subset Clos(E_-)$. 
     Hence by strong maximum principle and that $C$ isn't minimizing in $E_-$, we have $spt(S) \cap C = C\cap \partial \BB_1^{n+1}$. In particular, $dist(0, spt(S))>0$.
     Also, by \cite{HardtSimon79_Bdy_Reg}, $S$ is smooth near boundary. 
     We may abuse the notation and do not distinguish between $S$ and $spt(S)$.
     
     The fact that $\{r\cdot S\}_{r>0}$ foliates $E_-$ follows by considering $\sup\{r\in (0 ,1): r\cdot S\cap S = \emptyset\}$ and using strong maximum principle \cite{Simon87_Max, Ilmanen96}. The uniqueness of such $S$ also follows from the foliation structure.
     
     Consider the Jacobi field $u:= \nu_{S}\cdot \nabla_{\RR^{n+1}}|x|^2/2$, where $\nu_{S}$ is the normal field of $S_-$ pointing away from $C$. Then since $\{r\cdot S\}_{r>0}$ foliates $E_-$, $0\leq u \leq 1$ and is not identically $0$. By strong maximum principle for $L_{S}$, $u>0$ on $int(S)$, hence by lemma \ref{Lem_Pre, Diverge of pos Jac field}, $Sing(S)=\emptyset$.
     
     Also by weak maximum principle for $\Delta_{S}$, $u<1$ on $\partial S$. Hence $S$ is a minimal graph over $C$ near $\partial S$. By Hopf boundary lemma, $u>0$ on $\partial S$ and hence (\ref{App_partial_r h >0}) holds.
    \end{proof}

%    \begin{Eg}
%     Now consider $\Sigma:= C\times \RR \cap \SSp^8 \subset M:=\SSp^8\subset \RR^9$. $\Sigma$ has two singularities at south and north pole $\pm P$ of $\SSp^8$. By the same argument as in [Smale '99; sec 3], we can choose a metric $g$ on $M$ such that $\Sigma$ is a strictly stable minimal hypersurface in $(M, g)$ and that near $\pm P$, $\Sigma\subset (M, g)$ is a truncated cone minimal cone in Euclidean space. Choose $\nu$ to be the normal field of $\Sigma$ pointing into the side in which $\Sigma$ is minimizing near each singularity.
     
%     Denote $M_+ \sqcup M_- = M\setminus Clos(\Sigma)$, where $\nu$ points into $M_+$. By the same argument in theorem \ref{Thm_One-sided Perturb, Main Thm, mean convex foliation}, there exists a piecewise smooth mean convex foliation $\{\Sigma^+_t\}_{t\in (0, 1)}$ of some collar neighborhood of $\Sigma$ inside $M_+$ with mean curvature pointing towards $\Sigma$; 
%     On the other hand, let $\delta>0$ TBD; $\Sigma^{-, int}$
%    \end{Eg}

    \paragraph{\textbf{List of Open Problems}}
    To finish this article, we list some open problems related with singular minimal hypersurface. Many of them are well known. \\
    
\noindent    \textbf{$\bullet$ Singularity Model.}
     Let $C\subset \RR^{n+1}$ be a stable minimal hypercone. Recall that $C$ is called \textsl{regular} if it has smooth cross section. Let 
     \begin{align*}
      \cM_k(C):= \{\Sigma \subset \RR^{n+1} & \text{ locally stable minimal hypersurface, with codim}\geq 7 \text{ singular set}\\
      &\text{ and the tangent varifold at }\infty \text{ is k-multiples of }C \}  
     \end{align*}
     and let 
     \begin{align*}
      \cM^s(C) & := \{\Sigma\in \bigcup_{k\geq 1}\cM_k(C): \Sigma \text{ is stable in }\RR^{n+1}\} \\
      \cM^a(C) & := \{\Sigma\in \cM^s(C): \Sigma \text{ is area-minimizing in }\RR^{n+1}\}  \\
      \cM^{a, loc}(C) & := \{\Sigma\in \cM^s(C): \Sigma \text{ is area-minimizing in }U \text{ for some open subset }U\supset \Sigma\}         
     \end{align*}
     Here are some problems regarding these classes of minimal hypersurfaces.
     \begin{enumerate}      
      \item[(P1.1)] Find finite index minimal hypersurfaces in $\cM_k(C)$ for $k\geq 2$. It seems nontrivial even when $C=\RR^n$. When $n=2$, this has long been studied. See \cite{ChodoshKetoverMaximo17_Ind_Top_End_I, ChodoshMaximo18_Ind_Top_End_II} for a great survey and recent developments; When $n\geq 3$, the only connected example the author know is the higher dimensional catenoid. \cite{Schoen83_Sym_Uniq_Emb} show that this is unique up to scaling in $\cM_2(\RR^n)$; \cite{LiChao17_Ind} studied the relations between k (number of ends) and the index for a locally stable minimal hypersurface in $\cM_k(\RR^n)$. Similar index bound in terms of numbers of ends is expected for minimal hypersurfaces in $\cM_k(C)$.
      
      \cite{Mazet14_MS_Asym_Simons, ABPRS05_Sym_MS} construct catenoidal example in $\cM_2(C)$ for quadratic cones $C$. Can we compute its index? Can we find similar examples for arbitrary stable/minimizing cones? Are they unique among the connected minimal hypersurfaces in $\cM_2(C)$ up to scaling? 
      
      \cite{CaoShenZhu97_Stab_One_End} shows that a connected stable minimal hypersurface in $\RR^{n+1}$ has only one end. In particular, this implies for regular cone $C$, every $\Sigma\in \cM_{\geq 2}(C)$ is unstable. Is this true when $C$ has more singularities? 
      
      \item[(P1.2)] Find $\Sigma\in \cM_1(C)$ with index $\geq 1$? 
      
      When $C = C^{5,1}\subset \RR^8$ be the Simons cone over $\SSp^5\times \SSp^1$, \cite{Mazet14_MS_Asym_Simons} constructed two families of $O(6)\times O(2)$-invariant smooth minimal hypersurfaces asymptotic to $C$ near infinity from different side, each is generated by homothety. Recall that by \cite{Simoes74, Lawlor91}, $C^{5,1}$ is minimizing only in one side, hence by theorem \ref{Thm_Ass Jac, Asymp Thm of ext. min. graph}, one of the two families above is unstable. Does it have index $1$?

      Can we find such index $\geq 1$ example asymptotic to an arbitrary one-sided minimizing cone (In particular, for the isoparametric cone in $\RR^{16}$ mentioned in (1.1)) or more generally, a stable cone? Is it unique (up to scaling)? 
      Note that such higher index minimal hypersurface provides an example of a family of higher index minimal hypersurface multiplicity one converges to a strictly stable singular minimal hypersurface, which never happen in smooth case.
      
      \item[(P1.3)] In theorem \ref{Thm_Ass Jac, Asymp Thm of ext. min. graph}, we bound the asymptotic rate of any $\Sigma\in \cM_1^s(C)$ towards infinity provide $C$ is regular.
       In general, if $\Sigma\in \cM_1(C)$, can we bound its asymptotic rate at infinity toward $C$ in terms of its index? A precise and more ambitious conjecture is that when $C$ is regular, the asymptotic rate is bounded from below by $\gamma_1^-(C)$. This is related with the finiteness of associated Jacobi field when we allow index drop. See (P2.2) below.
            
      \item[(P1.4)] If $C$ is regular and strictly minimizing, with cross section $S= C\cap \SSp^n$, then by \cite{Chan97} and \cite{White89}, there's an $I:= ind(L_S)$ dimensional family (but probably not continuously parametrized) of minimal hypersurfaces in $\cM^a(C)$, where $L_S$ be the Jacobi operator of $S \subset \SSp^n$. 
      More precisely, the method in \cite{Chan97} actually shows that for strictly minimizing cones, there's an $\RR^I$ parametrized space of "ends", to each of which one can associate a closed subset of $\RR^{n+1}$, being either a minimizing hypersurface or a domain bounded by 2 disjoint minimizing hypersurfaces; Moreover, every $\Sigma\in \cM^a(C)$ with polynomial decaying rate near infinity towards $C$ lies in exactly one of these closed subset.
      Can we show that each of these closed sets is simply one minimizing hypersurfaces? A less ambitious conjecture may be that if one of such close subset has nonempty interior, show that it can be foliated by minimizing hypersurfaces. 
      
      Examples of area-minimizing hypercone has been constructed in \cite{Simons68, HsiangLawson71, FerusKarcherMunzner81, FerusKarcher85}, which are all isoparametric minimal cones; By \cite{HardtSimon85, SterlingWang94, TangZhang20}, every area minimizing hypercone in the family above is strictly minimizing and strictly stable.
      Can we determine these family of closed set in the examples above? \cite{SimonSolomon86, EdelenSpolaor19, Mazet14_MS_Asym_Simons} determine $\cM_1(C^{p,q})$ for Simons' cones $C^{p,q}$; \cite{Solomon90_Cubic_I, Solomon90_Cubic_II, Solomon92_Quartic} study the index for the cross section of cubic and quartic isoparametric hypercones, which is the first step toward this problem.
            
      \item[(P1.5)] When a minimizing hypercone $C$ has singularities other than the origin, is it expected that in general an analogue of \cite{HardtSimon85} is also valid (i.e. existence and uniqueness of smooth minimizing foliation with $C$ to be a leaf)? Perhaps one should first look at cylindrical hypercones and try to prove the uniqueness. \cite{Lohkamp18_Smoothing} proposed an argument towards this.
     \end{enumerate}

\noindent     \textbf{$\bullet$ First Order Model.}
     \begin{enumerate}
      \item[(P2.1)] Does every closed locally stable minimal hypersurface $\Sigma$ satisfies the $L^2$-nonconcentration property? This should be the starting point of linear theory for Jacobi operators.
      
      \cite{Lohkamp19_Potential_I, Lohkamp20_Potential_II} studied the Martin boundary for a $\cS$-adapted Schr\"odinger operator, can we do the same thing for the Jacobi operator to a stable minimal hypersurface? In particular, can we get similar asymptotic behavior of Green's functions like in section \ref{Subsec, Linear asymp near sing} when $Sing(\Sigma)$ is not necessarily isolated?
      
      \item[(P2.2)] Consider a family of pairwise different locally stable minimal hypersurfaces $\Sigma_j$ converges in varifold sense to a multiplicity one $\Sigma$ in $(M, g)$. Theorem \ref{Thm_Ass Jac, Main thm} tells us that when $ind(\Sigma_j) = ind(\Sigma)$, there exists a non-trivial Jacobi field $u\in C^\infty(\Sigma)$ associated to $\{\Sigma_j\}$.
       Is this true when $\Sigma$ has higher dimensional singularities?  Is this true when $\Sigma_j$ are of higher index than $\Sigma$? (This may be the first step to construct infinitely many closed embedded smooth minimal hypersurface in a generic eight manifold; And this is closely related with (P1.3) above.)
       When $Sing(\Sigma)$ is of higher dimension, this seems non-trivial even for edge-typed singularities.
     \end{enumerate}

\noindent     \textbf{$\bullet$ Neighborhood Foliation.}
     \begin{enumerate}
     \item[(P3.1)] Let $(\Sigma, \nu)\subset (M, g)$ be a closed two-sided locally stable minimal hypersurfaces. Can we find a (piecewise smooth) foliation of some neighborhood $U$ of $Clos(\Sigma)$ by hypersurfaces with mean curvature of fixed sign (may be different between different leaves)? 

      Theorem \ref{Thm_Intro, loc minimiz} confirms this under strong assumptions on singularities. It is expected that when some singularity of $\Sigma$ is not minimizing in one side, then there exists a family of mean concave collar neighborhood with boundary foliates some collar neighborhood of $\Sigma$.
      
      In \cite{BrayBrendleNeves10_Rigid_2Sphere, Song18_YauConj}, when $\Sigma$ is smooth and uniquely locally minimizing in one side, a smooth mean convex foliation for some collar neighborhood of $\Sigma$ is constructed and used to show the min-max solution in the manifold with cylindrical end must have portions inside the core manifold. Does similar smooth mean convex foliation exist when $Sing(\Sigma)\neq \emptyset$?
      
     \item[(P3.2)] A natural approach to (3.1) is by mean curvature flow. Can we construct mean convex ancient flow from or immortal flow towards $\Sigma$ when $Sing(\Sigma) \neq \emptyset$? Are they smooth when sufficiently close to $\Sigma$? See also \cite[Problem 5.7]{Brothers86_OpenProblems} and \cite{Ding20_MeanConvex_SelfExpander}. 
%     [Choi-Mantolidis '18] studies the dynamics of mean curvature flow near a smooth minimal hypersurface. What's the analogue for singular ones?
     \end{enumerate}

\noindent     \textbf{$\bullet$ Deformations and Local Moduli.}
     Let $(\Sigma, \nu)\subset (M, g)$ be a two-sided locally stable minimal hypersurface with only strongly isolated singularities.
      \begin{enumerate}
       \item[(P4.1)] In theorem \ref{Thm_Intro, main thm, deform}, we see the existence of nearby minimal hypersurfaces under perturbed metrics if (1)-(3) is satisfied by $\Sigma$. Assumption (3) is conjectured to be dropped.
        
        Are these nearby minimal hypersurfaces unique in a fixed metric?
        A local version of this problem seems also unclear, even for Simons cones: Given a regular area minimizing hypercone $C\subset \RR^{n+1}$, for sufficiently small $C^4$ perturbation $S'$ of $\partial \BB^{n+1} \cap C$, is the area-minimizing hypersurface with boundary $S'$ in $\RR^{n+1}$ unique? What about strictly minimizing or strictly stable cones? What about Simons cones?
       \item[(P4.2)] A less ambitious problem is, if $\{g_t\}_{t\in [-1,1]}$ is a smooth family of Riemannian metric sufficiently close to $g$, can we find a continuous family of locally stable minimal hypersurface $\{\Sigma_t \subset (M ,g_t)\}_{t\in [-1,1]}$ close to $\Sigma$?   
      \end{enumerate}

\bibliographystyle{alpha}

\bibliography{GMT}

\end{document}